\documentclass[12pt,oneside,reqno,english]{amsart}
\usepackage[T1]{fontenc}
\usepackage[latin9]{inputenc}
\usepackage{geometry}
\usepackage{babel}
\usepackage{array}
\usepackage{float}
\usepackage{units}
\usepackage{mathrsfs}
\usepackage{enumitem}
\usepackage{amstext}
\usepackage{amsthm}
\usepackage{amssymb}
\usepackage{stmaryrd}
\usepackage{graphicx}
\usepackage[all]{xy}
\PassOptionsToPackage{normalem}{ulem}
\usepackage{ulem}
\usepackage[unicode=true,pdfusetitle,
 bookmarks=true,bookmarksnumbered=false,bookmarksopen=false,
 breaklinks=false,pdfborder={0 0 1},backref=false,colorlinks=false]
 {hyperref}
\hypersetup{
 colorlinks=true,citecolor=blue,linkcolor=blue,linktocpage=true}

\makeatletter

\providecommand{\tabularnewline}{\\}

\numberwithin{equation}{section}
\numberwithin{figure}{section}
\numberwithin{table}{section}
\theoremstyle{plain}
\newtheorem{thm}{\protect\theoremname}[section]
\theoremstyle{plain}
\newtheorem{lem}[thm]{\protect\lemmaname}
\theoremstyle{definition}
\newtheorem{defn}[thm]{\protect\definitionname}
\theoremstyle{remark}
\newtheorem{rem}[thm]{\protect\remarkname}
\theoremstyle{plain}
\newtheorem{cor}[thm]{\protect\corollaryname}
\theoremstyle{definition}
\newtheorem{example}[thm]{\protect\examplename}
\theoremstyle{plain}
\newtheorem{prop}[thm]{\protect\propositionname}
\theoremstyle{remark}
\newtheorem*{acknowledgement*}{\protect\acknowledgementname}

\@ifundefined{date}{}{\date{}}
\allowdisplaybreaks
\usepackage{needspace}
\usepackage{refstyle}
\usepackage{enumitem}
\usepackage{bbm}

\newref{lem}{refcmd={Lemma \ref{#1}}}
\newref{thm}{refcmd={Theorem \ref{#1}}}
\newref{cor}{refcmd={Corollary \ref{#1}}}
\newref{sec}{refcmd={Section \ref{#1}}}
\newref{sub}{refcmd={Section \ref{#1}}}
\newref{subsec}{refcmd={Section \ref{#1}}}
\newref{chap}{refcmd={Chapter \ref{#1}}}
\newref{prop}{refcmd={Proposition \ref{#1}}}
\newref{exa}{refcmd={Example \ref{#1}}}
\newref{tab}{refcmd={Table \ref{#1}}}
\newref{rem}{refcmd={Remark \ref{#1}}}
\newref{def}{refcmd={Definition \ref{#1}}}
\newref{fig}{refcmd={Figure \ref{#1}}}

\setlist[enumerate]{itemsep=5pt,topsep=3pt}
\setlist[enumerate,1]{label=\textup{(}\roman*\textup{)},ref=\roman*}
\setlist[enumerate,2]{label=(\alph*),ref=\theenumi \alph*}

\@ifundefined{showcaptionsetup}{}{%
 \PassOptionsToPackage{caption=false}{subfig}}
\usepackage{subfig}
\AtBeginDocument{
  
}

\makeatother

\providecommand{\acknowledgementname}{Acknowledgement}
\providecommand{\corollaryname}{Corollary}
\providecommand{\definitionname}{Definition}
\providecommand{\examplename}{Example}
\providecommand{\lemmaname}{Lemma}
\providecommand{\propositionname}{Proposition}
\providecommand{\remarkname}{Remark}
\providecommand{\theoremname}{Theorem}

\begin{document}

\title{Martin boundaries of representations of the Cuntz algebra}

\dedicatory{Dedicated to the memory of J\o rgen Hoffmann-J\o rgensen.}

\author{Palle Jorgensen and Feng Tian}

\address{(Palle E.T. Jorgensen) Department of Mathematics, The University
of Iowa, Iowa City, IA 52242-1419, U.S.A. }

\email{palle-jorgensen@uiowa.edu}

\urladdr{http://www.math.uiowa.edu/\textasciitilde jorgen/}

\address{(Feng Tian) Department of Mathematics, Hampton University, Hampton,
VA 23668, U.S.A.}

\email{feng.tian@hamptonu.edu}
\begin{abstract}
In a number of recent papers, the idea of generalized boundaries has
found use in fractal and in multiresolution analysis; many of the
papers having a focus on specific examples. Parallel with this new
insight, and motivated by quantum probability, there has also been
much research which seeks to study fractal and multiresolution structures
with the use of certain systems of non-commutative operators; non-commutative
harmonic/stochastic analysis. This in turn entails combinatorial,
graph operations, and branching laws. The most versatile, of these
non-commutative algebras are the Cuntz algebras; denoted $\mathcal{O}_{N}$,
$N$ for the number of isometry generators. $N$ is at least 2. Our
focus is on the representations of $\mathcal{O}_{N}$. We aim to develop
new non-commutative tools, involving both representation theory and
stochastic processes. They serve to connect these parallel developments.

In outline, boundaries, Poisson, or Martin, are certain measure spaces
(often associated to random walk models), designed to encode the asymptotic
behavior, e.g., how trajectories diverge when the number of steps
goes to infinity. We stress that our present boundaries (commutative
or non-commutative) are purely measure-theoretical objects. Although,
as we show, in some cases our boundaries may be compared with more
familiar topological boundaries. 
\end{abstract}



\maketitle

\tableofcontents{}

\section{Introduction}

We propose a new notion of Martin boundary for representations of
the Cuntz algebras. It bridges two ideas which have been studied extensively
in the literature, but so far have not been connected in a systematic
fashion. In summary, they are: (i) \emph{the non-commutativity of
the Cuntz algebras} (see, e.g., \cite{MR0467330,MR1913212,MR2030387}),
and the subtleties of their representations \cite{Gli60,Gli61}, on
the one hand; and (ii) \emph{symbolic representations} of Markov chains
and their classical Martin boundaries, on the other (see, e.g., \cite{MR3450534,MR2357627,MR2384480,MR2851247}).
Applications include an harmonic analysis of \emph{iterated function
system} (IFS) measures. (See \cite{MR3124323,MR3459161,MR3375595,MR2176941,MR2129258,MR3368972}.)

In the study of representations of $\mathcal{O}_{N}$ on a Hilbert
space $\mathscr{H}$, an identification of suitably closed invariant
subspaces of $\mathscr{H}$ plays a central role. Here we refer to
a representation in the form of a system operators $S_{i}$ and their
adjoints $S_{i}^{*}$ satisfying the Cuntz relations. Of the possibilities
for subspaces, invariance under the $S_{i}^{*}$ operators is more
interesting: i.e., invariance under a system of generalized backwards
shifts. In many cases, these invariant subspaces have small dimension,
and they help us define new isomorphism invariants for the representations
under discussion. For example, a permutative representation is one
with the property that the vectors in some choice of ONB are permuted
by the $S_{i}^{*}$ operators. Moreover, in important applications
to quantum statistical mechanics, certain subspaces of states that
are invariant under the adjoints $S_{i}^{*}$, are often finite-dimensional.
They are called \emph{finitely correlated states}. And they are one
of the main features of interest in statistical mechanics, see e.g.,
\cite{MR1158756,MR1266319,MR1444086,MR1665280,MR2294415,MR1740897}.
They are analogues of \textquotedblleft attractors\textquotedblright{}
in classical (commutative) symbolic dynamics.

As we show in Section \ref{sec:IFS} in the present paper, there is
a way to associate families of representations of the Cuntz algebras
to a certain analyses of iterated function systems (IFSs), and for
this reason, some of the earlier work on IFSs is relevant to our present
considerations. For this part of the literature, we refer the reader
to the papers \cite{MR2833578,MR3130523,MR3709130,MR2891314,MR2954650,MR3601650,2017arXiv170900592F,2018arXiv180308779F,2018arXiv180403455F},
and the papers cited there.

\section{\label{sec:Pre}Preliminaries}

We begin with a technical lemma regarding projections in Hilbert space;
to be used inside the arguments throughout the paper.

Let $\mathscr{H}$ be a Hilbert space. By an orthogonal projection
$P$ on $\mathscr{H}$, we mean an operator satisfying $P=P^{*}=P^{2}$.
There is a bijective correspondence between projections $P$ (we shall
assume that $P$ is orthogonal even if not stated) on the one hand,
and closed subspaces $\mathscr{F}=\mathscr{F}_{P}$ in $\mathscr{H}$
on the other, given by $\mathscr{F}=P\mathscr{H}$; see e.g., \cite{MR3642406}.

We shall use the following 
\begin{lem}
\label{lem:pp}Let $P$ and $Q$ be projections, and let $\mathscr{F}_{P}$
and $\mathscr{F}_{Q}$ be the corresponding closed subspaces, then
TFAE:
\begin{enumerate}
\item $P=PQ$;
\item $P=QP$;
\item $\mathscr{F}_{P}\subseteq\mathscr{F}_{Q}$;
\item $\left\Vert Ph\right\Vert \leq\left\Vert Qh\right\Vert $, $\forall h\in\mathscr{H}$;
\item $\left\langle h,Ph\right\rangle \leq\left\langle h,Qh\right\rangle $,
$\forall h\in\mathscr{H}$. 
\end{enumerate}
When the conditions hold we say that $P\leq Q$.

\end{lem}

\begin{proof}
This is standard in operator theory. We refer to \cite{MR3642406}
for details. 
\end{proof}
\begin{defn}
\label{def:peq}~
\begin{enumerate}
\item \label{enu:peq1}Let $\mathscr{H}$ be a Hilbert space, and $V$ an
operator in $\mathscr{H}$. If $P:=V^{*}V$ is a projection, we say
that $V$ is a \emph{partial isometry}. In that case, $Q=VV^{*}$
is also a projection: We say that $P$ is the initial projection of
$V$, and that $Q$ is the final projection. 
\item \label{enu:peq2}If $\mathfrak{A}$ is a $C^{*}$-algebra, and $V$,
$P$, $Q$ are as above. If $V$ is in $\mathfrak{A}$, then we say
that the two projections $P$ and $Q$ are \emph{$\mathfrak{A}$-equivalent. }
\end{enumerate}
\end{defn}

\begin{lem}
\label{lem:mp}Let $\left\{ P_{k}\right\} _{k\in\mathbb{N}}$ be monotone,
i.e., 
\begin{equation}
P_{1}\leq P_{2}\leq\cdots,\label{eq:a11}
\end{equation}
then the limit 
\begin{equation}
P_{\infty}:=\lim_{k\rightarrow\infty}P_{k}\label{eq:a12}
\end{equation}
$($in the strong operator topology of $\mathscr{B}\left(\mathscr{H}\right)$$)$
exists, and $P_{\infty}$ is the projection onto the closed span of
the subspaces $\left\{ \mathscr{F}_{P_{k}}\right\} _{k\in\mathbb{N}}$. 

The analogous conclusion holds for monotone decreasing sequence of
projections
\begin{equation}
\cdots\leq Q_{n+1}\leq Q_{n}\leq\cdots\leq Q_{2}\leq Q_{1}.\label{eq:a13}
\end{equation}
In this case 
\begin{equation}
Q_{\infty}=\lim_{k\rightarrow\infty}Q_{k}\label{eq:a14}
\end{equation}
is the projection onto $\bigcap_{k}\mathscr{F}_{Q_{k}}$.
\end{lem}

\section{\label{sec:PVR}A Projection Valued Random Variable}

The theme of our paper falls at the crossroads of \emph{representation
theory} and the study of \emph{fractal measures} and their stochastic
processes.

The past two decades has seen a burst of research dealing with representations
of classes of \emph{infinite} $C^{*}$-algebras, which includes the
Cuntz algebras \cite{MR0467330}, $\mathcal{O}_{N}$ (see (\ref{eq:a5}))
as well as other graph-$C^{*}$-algebras \cite{2017arXiv170900592F,2018arXiv180308779F}.
(For details on a number of such earlier studies and applications,
readers are referred to the papers cited below.) A source of motivation
for our present work includes more recent research which includes
both pure and applied mathematics: branching laws for endomorphisms,
subshifts, endomorphisms from measurable partitions, Markov measures
and topological Markov chains, wavelets and multiresolutions, signal
processing and filters, iterated function systems (IFSs) and fractals,
complex projective spaces, quasi-crystals, orbit equivalence, and
substitution dynamical systems, and tiling systems \cite{MR3402823,MR3450534,MR3687240,MR3642406,MR3796644}.

A projection $P$ is said to be \emph{infinite} iff (Def) it contains
proper subprojections, say $Q$, $Q\lneqq P$, such that $P$ and
$Q$ are\emph{ equivalent}; (see Definition \ref{def:peq} (\ref{enu:peq2})).
The Cuntz algebras $\mathcal{O}_{N}$ contain infinite projections;
see Sections \ref{sec:PVM}-\ref{sec:IFS}.

The questions considered here for \emph{representations of the Cuntz
algebras} are of independent interest as part of \emph{non-commutative
harmonic analysis}, i.e., the study of representations of non-abelian
groups and $C^{*}$-algebras. A basic question in representation theory
is that of determining parameters for the \emph{equivalence classes
of representations}, where \textquotedblleft \emph{equivalence}\textquotedblright{}
refers to unitary equivalence. Since analysis and synthesis of representations
must entail direct integral decompositions, a minimal requirement
for a list of parameters for the equivalence classes of representations,
is that it be Borel. When such a choice is possible, we say that there
is a \emph{Borel cross section} for the representations under consideration.

A pioneering paper by J. Glimm \cite{Gli60} showed that there are
infinite $C^{*}$-algebras whose representations do not have Borel
cross sections. (Loosely speaking, the representations do not admit
classification.) It is known that the Cuntz algebras, and $C^{*}$-algebras
of higher-rank graphs, fall in this class. Hence, the approach to
representations must narrow to suitable and amenable classes of representations
which arise naturally in \emph{applications}, and which do admit \emph{Borel
cross sections}.

A leading theme in our paper is a formulation of a \emph{boundary
theory} for representations of the Cuntz algebra. This in turn ties
in with multiresolutions and with iterated function system (IFS) measures.
A boundary theory for the latter has recently been suggested in various
special cases.

A multiresolution approach to the study of representations of the
Cuntz algebras was initiated by the first named author and O. Brattelli
\cite{MR1869063,MR1889566,MR1913212,MR2030387}; and it includes such
applications as construction of new \emph{multiresolution wavelets},
and of wavelet algorithms from multi-band wavelet filters. And yet
other applications studied by the first named author and D. Dutkay
lead to the study of such classes of representations as \emph{monic},
and \emph{permutative} \cite{MR3217056,2018arXiv180403455F}; and
their use in fractal analysis. The introduction of these classes begins
with the fact that every representation of the Cuntz algebra corresponds
in a canonical fashion to a certain \emph{projection valued measure}.
We begin with these projection valued measures in this and the next
section below.\\

Let $N$ be a positive integer, and let $A$ be an alphabet with $\left|A\right|=N$;
set 
\begin{equation}
\Omega_{N}:=A^{\mathbb{N}}=\underset{\aleph_{0}-\text{infinite Cartesian product}}{\underbrace{A\times A\times A\times\cdots\cdots\cdots}}.\label{eq:a1}
\end{equation}
Points in $\Omega_{N}$ are denoted $\omega:=\left(x_{1},x_{2},\cdots\right)$,
and we set 
\begin{equation}
\pi_{n}\left(\omega\right):=x_{n},\quad\forall\omega\in\Omega_{N}.\label{eq:a2}
\end{equation}
When $k\in\mathbb{N}$ is fixed, and $\omega=\left(x_{i}\right)\in\Omega_{N}$,
we set 
\begin{equation}
\omega|_{k}=\left(x_{1},x_{2},\cdots,x_{k}\right)=\text{the \ensuremath{k}-truncated (\emph{finite}) word.}\label{eq:a3}
\end{equation}

Let $\mathscr{H}$ be a separable Hilbert space, $\dim\mathscr{H}=\aleph_{0}$,
and let $\mathfrak{M}$ be a commutative family of orthogonal projections
in $\mathscr{H}$. 

By an $\mathfrak{M}$-valued random variable $X$, we mean a measurable
function
\begin{equation}
X:\Omega_{N}\longrightarrow\mathfrak{M}.\label{eq:a4}
\end{equation}
See, e.g., \cite{MR3402823,MR3687240,MR3796644}.

Let $\mathcal{O}_{N}$ denote the Cuntz algebra with $N$ generators
\cite{MR0467330}, i.e., the $C^{*}$-algebra on symbols $\left\{ s_{i}\right\} _{i=1}^{N}$,
satisfying the following two relations: 
\begin{equation}
s_{i}^{*}s_{j}=\delta_{ij}\mathbbm{1},\quad\text{and}\quad\sum_{i=1}^{N}s_{i}s_{i}^{*}=\mathbbm{1},\label{eq:a5}
\end{equation}
where $\mathbbm{1}$ denotes the unit element in $\mathcal{O}_{N}$. 

By a representation of $\mathcal{O}_{N}$ we mean a function $s_{i}\mapsto S_{i}=\pi\left(s_{i}\right)$
such that 
\begin{equation}
S_{i}^{*}S_{j}=\delta_{ij}I,\quad\text{and}\quad\sum_{i=1}^{N}S_{i}S_{i}^{*}=I\label{eq:a7}
\end{equation}
where $\delta_{ij}$ denotes the Kronecker delta, and $I$ denotes
the identity operator in $\mathscr{H}$; we say that $\pi\in Rep\left(\mathcal{O}_{N},\mathscr{H}\right)$
if (\ref{eq:a7}) holds. 

The following lemma is basic and will be used throughout in the remaining
of the paper.
\begin{lem}
\label{lem:con}Let $\pi=\left(S_{i}\right)_{i=1}^{N}$ be a representation
of $\mathcal{O}_{N}$ acting in a fixed Hilbert space $\mathscr{H}$,
i.e., $\pi\in Rep\left(\mathcal{O}_{N},\mathscr{H}\right)$. For finite
words $f=\left(x_{1},\cdots,x_{n}\right)$ in the alphabet $A=\left\{ 1,2,\cdots,N\right\} $,
set 
\begin{equation}
P_{f}:=S_{x_{1}}S_{x_{2}}\cdots S_{x_{k}}S_{x_{k}}^{*}S_{x_{k-1}}^{*}\cdots S_{x_{1}}^{*}\label{eq:a71}
\end{equation}
with the conventions: 
\begin{equation}
P_{i}:=S_{i}S_{i}^{*},\quad\text{and}\quad P_{\emptyset}=0.\label{eq:a72}
\end{equation}
\begin{enumerate}
\item Then as $f$ varies over all finite non-empty words, the projections
$\text{\ensuremath{\left\{  P_{f}\right\} } }$ \textup{form an abelian
family. }
\item Moreover, 
\begin{equation}
\sum_{i=1}^{N}P_{\left(fi\right)}=P_{f},\label{eq:a73}
\end{equation}
and in particular, 
\begin{equation}
P_{\left(fg\right)}\leq P_{f}\label{eq:a74}
\end{equation}
for any pair of finite non-empty words $f$ and $g$. Here $\left(fg\right)$
denotes concatenation of the two words. 
\end{enumerate}
\end{lem}

\begin{proof}
This is an application of (\ref{eq:a7}), and the details are left
for the reader. 
\end{proof}
\begin{thm}
\label{thm:LP}Let $N$, $\mathscr{H}$, $\mathcal{O}_{N}$, and $\pi\in Rep\left(\mathcal{O}_{N},\mathscr{H}\right)$
be as above. 
\begin{enumerate}
\item \label{enu:s1}Then there is a unique random variable $X$ (projection-valued,
see (\ref{eq:a4})) such that 
\begin{equation}
X\left(\omega\right)=\lim_{k\rightarrow\infty}S_{\omega|_{k}}S_{\omega|_{k}}^{*},\;\omega\in\Omega_{N},\label{eq:a8}
\end{equation}
where 
\begin{equation}
S_{\omega|_{k}}=S_{x_{1}}S_{x_{2}}\cdots S_{x_{k}}\label{eq:a9}
\end{equation}
and $\omega|_{k}$ is the corresponding truncated word as in (\ref{eq:a3}). 
\end{enumerate}
\begin{enumerate}[resume]
\item \label{enu:s2}Moreover, the following relations hold: 

If $a\in A$, and $\omega=\left(x_{1},x_{2},x_{3},x_{4},\cdots\right)\in\Omega_{N}$,
then 
\begin{align}
S_{a}X\left(\omega\right)S_{a}^{*} & =X\left(a\omega\right),\label{eq:b1}
\end{align}
and
\begin{equation}
S_{a}^{*}X\left(\omega\right)S_{a}=\delta_{a,\pi_{1}\left(\omega\right)}X\left(\sigma\left(\omega\right)\right),\label{eq:b2}
\end{equation}
where
\[
\sigma\left(\omega\right)=\left(x_{2},x_{3},x_{4},\cdots\right),\;\text{and}\quad a\omega=\left(a,x_{1},x_{2},x_{3},x_{4},\cdots\right).
\]
 
\item \label{enu:s3}Finally, we have: 
\begin{equation}
X\left(\omega\right)X\left(\omega'\right)=\delta_{\omega,\omega'}X\left(\omega\right),\label{eq:a10}
\end{equation}
for all $\omega,\omega'\in\Omega_{N}$. 
\end{enumerate}
\end{thm}

\begin{proof}
Let $N\in\mathbb{N}$, $N\geq2$, and let $\mathscr{H}$ be a separable
Hilbert space. Let $\mathcal{O}_{N}$ denote the Cuntz algebra with
$N$ generators $\left\{ s_{i}\right\} _{1}^{N}$, and let $\pi\in Rep\left(\mathcal{O}_{N},\mathscr{H}\right)$,
$\pi\left(s_{i}\right):=S_{i}$, $1\leq i\leq N$, be fixed. 

Let $\Omega_{N}$ be as in (\ref{eq:a1}). For $\omega=\left(i_{1},i_{2},i_{3},\cdots\right)\in\Omega_{N}$,
and $k\in\mathbb{N}$, set $\omega|_{k}=\left(i_{1},i_{2},\cdots,i_{k}\right)$,
the \emph{truncated word}; see (\ref{eq:a3}). Set $S_{\omega|_{k}}=S_{i_{1}}\cdots S_{i_{k}}$,
and 
\begin{equation}
P\left(\omega|_{k}\right):=S_{\omega|_{k}}S_{\omega|_{k}}^{*}.\label{eq:a16}
\end{equation}
Then $\left\{ P\left(\omega|_{k}\right)\right\} _{k\in\mathbb{N}}$
in (\ref{eq:a16}) is a monotone decreasing family of projections,
i.e., 
\begin{equation}
P\left(\omega|_{1}\right)\geq P\left(\omega|_{2}\right)\geq\cdots\geq P\left(\omega|_{k}\right)\geq P\left(\omega|_{k+1}\right)\geq\cdots;\label{eq:a17}
\end{equation}
and so, by Lemma \ref{lem:mp}, the following limit projection exists:
\[
X\left(\omega\right):=\lim_{k\rightarrow\infty}P\left(\omega|_{k}\right).
\]
Specifically, the limit exists in the strong operator topology on
$\mathscr{B}\left(\mathscr{H}\right)$, and $X\left(\omega\right)$
is the projection onto the intersection of the closed subspaces in
$\mathscr{H}$ given by 
\[
S_{\omega|_{k}}\mathscr{H}=S_{i_{1}}S_{i_{2}}\cdots S_{i_{k}}\mathscr{H}
\]
as $k$ varies over $\mathbb{N}$. 

The proof of \emph{monotonicity} in (\ref{eq:a17}) is the following
estimate in the order of projections (see Lemma \ref{lem:pp}): 
\[
\underset{P\left(\omega|_{k+1}\right)}{\underbrace{S_{i_{1}}\cdots S_{i_{k}}S_{i_{k+1}}S_{i_{k+1}}^{*}S_{i_{k}}^{*}\cdots S_{i_{1}}^{*}}}\leq\underset{P\left(\omega|_{k}\right)}{\underbrace{S_{i_{1}}\cdots S_{i_{k}}S_{i_{k}}^{*}\cdots S_{i_{1}}^{*}}}.
\]
See also Lemma \ref{lem:con}. Once the limits are established, conclusion
(\ref{enu:s1}) in the theorem is clear. 

The remaining conclusions (\ref{enu:s2})-(\ref{enu:s3}) follow from
passing to the limit $k\rightarrow\infty$ as follows. Here $\omega,\omega'\in\Omega_{N}$
are fixed infinite words, $\omega=\left(i_{1},i_{2},i_{3},\cdots\right)$,
$\omega'=\left(j_{1},j_{2},j_{3},\cdots\right)$, and let $a\in A=\left\{ 1,2,\cdots,N\right\} $. 

Conclusion (\ref{eq:b1}) follows from: 
\[
S_{a}S_{i_{1}}\cdots S_{i_{k}}S_{i_{k}}^{*}\cdots S_{i_{1}}^{*}S_{a}^{*}=S_{a\omega|_{k}}S_{a\omega|_{k}}^{*},
\]
and then passing to the limit, $k\rightarrow\infty$, using (\ref{enu:s1}). 

Conclusion (\ref{eq:b2}) follows from 
\[
S_{a}^{*}S_{i_{1}}\cdots S_{i_{k}}S_{i_{k}}^{*}\cdots S_{i_{1}}^{*}S_{a}=\delta_{a,i_{1}}S_{i_{2}}\cdots S_{i_{k}}S_{i_{k}}^{*}\cdots S_{i_{2}}^{*}.
\]
Here we used (\ref{eq:a7}), since $\pi$ is given to be in $Rep\left(\mathcal{O}_{N},\mathscr{H}\right)$.
Now (\ref{eq:b2}) follows from taking the limit $k\rightarrow\infty$,
and using again Lemma \ref{lem:mp}.

Finally, conclusion (\ref{enu:s3}) in the theorem follows from the
computation: 
\[
S_{i_{1}}\cdots S_{i_{k}}\left(S_{i_{k}}^{*}\cdots S_{i_{1}}^{*}S_{j_{1}}\cdots S_{j_{k}}\right)S_{j_{k}}^{*}\cdots S_{j_{1}}^{*}=\delta_{i_{1}j_{1}}\delta_{i_{2}j_{2}}\cdots\delta_{i_{k}j_{k}}S_{i_{1}}\cdots S_{i_{k}}S_{i_{k}}^{*}\cdots S_{i_{1}}^{*}.
\]
Passing to the limit $k\rightarrow\infty$, the desired conclusion
(\ref{enu:s3}) now follows.
\end{proof}

\section{\label{sec:PVM}A Projection Valued Path-Space Measure}

Let $\mathscr{H}$ be a separable Hilbert space, and fix $N\geq2$,
and $\pi\in Rep\left(\mathcal{O}_{N},\mathscr{H}\right)$. We shall
be concerned with two tools directly related to the study of representations
of $\mathcal{O}_{N}$ on $\mathscr{H}$. (Also see \cite{MR1222649,MR2362879,MR2976663}.) 

With $\pi\left(s_{i}\right)=S_{i}$, $1\leq i\leq N$ fixed, set $\beta=\beta_{\pi}\in End\left(\mathscr{B}\left(\mathscr{H}\right)\right)$,
\emph{endomorphism}, and $\mathbb{Q}=\mathbb{Q}_{\pi}$, a canonical
\emph{projection-valued path-space measure}. Before giving the precise
details, we shall need a few facts about the path space, 
\begin{equation}
\Omega_{N}=\left\{ 1,2,\cdots,N\right\} ^{\mathbb{N}}.\label{eq:e1}
\end{equation}
This version of path-space is chosen for simplicity: We have taken
as alphabet the set $A:=\left\{ 1,2,\cdots,N\right\} $, but the fixed
alphabet could be any finite set $A$ with $\left|A\right|=N$; and
so $\Omega_{N}=A^{\mathbb{N}}=\underset{\aleph_{0}}{\underbrace{A\times A\times\cdots}}$,
the infinite Cartesian product; see Section \ref{sec:Pre}.
\begin{defn}
\label{def:end}For $T\in\mathscr{B}\left(\mathscr{H}\right)$, set
\begin{equation}
\beta_{\pi}\left(T\right)=\sum_{i=1}^{N}S_{i}TS_{i}^{*}.\label{eq:e2}
\end{equation}
Then $\beta_{\pi}\in End\left(\mathscr{B}\left(\mathscr{H}\right)\right)$,
i.e., 
\begin{align}
\beta_{\pi}\left(TT'\right) & =\beta_{\pi}\left(T\right)\beta_{\pi}\left(T'\right),\;\forall T,T'\in\mathscr{B}\left(\mathscr{H}\right);\label{eq:e3}\\
\beta_{\pi}\left(T^{*}\right) & =\beta_{\pi}\left(T\right)^{*},\;\text{and}\label{eq:e4}\\
\beta_{\pi}\left(I\right) & =I.\label{eq:e5}
\end{align}
\end{defn}

Given a representation $\pi\in Rep\left(\mathcal{O}_{N},\mathscr{H}\right)$,
then the corresponding endomorphism, 
\begin{equation}
\beta_{\pi}:\mathscr{B}\left(\mathscr{H}\right)\longrightarrow\mathscr{B}\left(\mathscr{H}\right)\label{eq:e5a}
\end{equation}
plays an important role in representation theory. For example, for
decompositions of $\pi$, by Schur, we will need the commutant $\left\{ \pi\right\} '$,
defined as follows: 
\begin{equation}
\left\{ \pi\right\} ':=\left\{ T\in\mathscr{B}\left(\mathscr{H}\right)\mathrel{{;}}T\,\pi\left(A\right)=\pi\left(A\right)T,\;\forall A\in\mathcal{O}_{N}\right\} .\label{eq:e5b}
\end{equation}

\begin{lem}
\label{lem:com}Let $\pi$ and $\beta_{\pi}$ be as in (\ref{eq:e2})
and (\ref{eq:e5a}); then 
\begin{equation}
\left\{ \pi\right\} '=\left\{ T\in\mathscr{B}\left(\mathscr{H}\right)\mathrel{{;}}\beta_{\pi}\left(T\right)=T\right\} \left(=\text{Fix}\ensuremath{\left(\beta_{\pi}\right)}.\right)\label{eq:e5c}
\end{equation}
\end{lem}

\begin{proof}
We have the following bi-implications: 
\begin{eqnarray*}
\beta_{\pi}\left(T\right) & = & T\\
 & \Updownarrow\\
S_{i}^{*}\beta_{\pi}\left(T\right) & = & S_{i}^{*}T,\quad1\leq i\leq N\\
 & \Updownarrow & \left(\text{by \ensuremath{\left(\ref{eq:e2}\right)}}\right)\\
TS_{i}^{*} & = & S_{i}^{*}T,\quad1\leq i\leq N\\
 & \Updownarrow\\
T & \in & \left\{ \pi\right\} '.
\end{eqnarray*}
\end{proof}
In applications to \emph{statistical mechanics}, given $\pi\in Rep\left(\mathcal{O}_{N},\mathscr{H}\right)$,
it is important to determine the closed subspaces $\mathscr{K}\subset\mathscr{H}$,
invariant under the operators $S_{i}^{*}$, $1\leq i\leq N$. 

\emph{Notation}: When $\mathscr{K}$ is a closed subspace, we shall
denote the corresponding projection by $P\left(=P_{\mathscr{K}}\right)$
(see Lemma \ref{lem:pp}.)
\begin{lem}
\label{lem:com2}Let $\left(\pi,\mathscr{H},\mathscr{K}\left(\text{with projection }P\right)\right)$
be as above; then TFAE:
\begin{enumerate}
\item $S_{i}^{*}\mathscr{K}\subseteq\mathscr{K}$, $1\leq i\leq N$;
\item $PS_{i}^{*}P=S_{i}^{*}P$, $1\leq i\leq N$;
\item $P\leq\beta_{\pi}\left(P\right)$, in the order of projections (see
Section \ref{sec:Pre}); and 
\item $P\leq\beta_{\pi}\left(P\right)\leq\cdots\leq\beta_{\pi}^{k}\left(P\right)\leq\beta_{\pi}^{k+1}\left(P\right)\leq\cdots$.
\end{enumerate}
\end{lem}

\begin{proof}
The argument is the same as that used in the proof of Lemma \ref{lem:com}. 
\end{proof}
\begin{lem}
Let $\left(\pi,\mathscr{H},\mathscr{K}\left(\text{with projection }P\right)\right)$
be as in Lemma \ref{lem:com2}, and set 
\begin{equation}
Q=\bigvee_{k=1}^{\infty}\beta_{\pi}^{k}\left(P\right);\label{eq:e5d}
\end{equation}
then $Q\in\left\{ \pi\right\} '$, and $Q$ is the smallest projection
in $\left\{ \pi\right\} '$ satisfying $P\leq Q$. 
\end{lem}

\begin{proof}
The conclusion is immediate from the formula:
\[
S_{i}^{*}\beta_{\pi}^{k+1}\left(P\right)=\beta_{\pi}^{k}\left(P\right)S_{i}^{*},\;1\leq i\leq N.
\]
Assuming (\ref{eq:e5d}), we then get 
\[
S_{i}^{*}Q=QS_{i}^{*},\;1\leq i\leq N,
\]
and by taking adjoints 
\[
QS_{i}=S_{i}Q;
\]
so $Q\in\left\{ \pi\right\} '$. The remaining parts of the proof
are immediate.
\end{proof}
\begin{defn}
We shall use the standard $\sigma$-algebra $\mathscr{C}$ of subsets
of $\Omega_{N}$ (the path-space). The $\sigma$-algebra is generated
by \emph{cylinder sets} $E_{f}$. Here $f=\left(i_{1},i_{2},\cdots,i_{k}\right)$
is a finite word, $\left|f\right|=k$; and
\begin{equation}
E_{f}:=\left\{ \omega\in\Omega_{N}\mid\omega_{j}=i_{j},\;1\leq j\leq k\right\} ;\label{eq:e6}
\end{equation}
is one of the \emph{basic cylinder sets} (see Fig \ref{fig:cyl}).
The $\sigma$-algebra $\mathscr{C}$ is the smallest $\sigma$-algebra
containing the sets $E_{f}$ as $f$ varies over all finite words
in the fixed alphabet $A$. 
\end{defn}

\begin{figure}[H]
\includegraphics[width=0.8\textwidth]{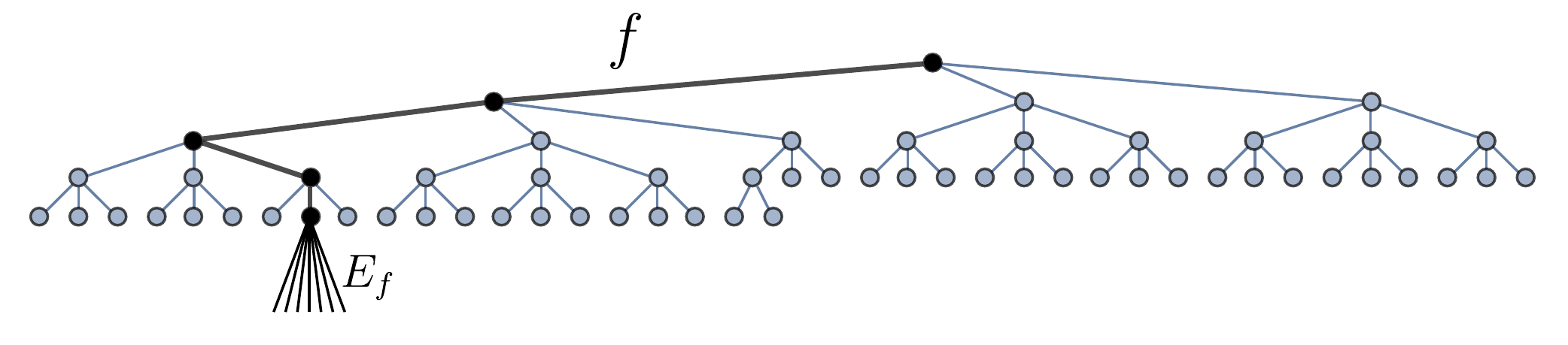}

\caption{\label{fig:cyl}A basic cylinder set.}

\end{figure}

\begin{defn}[Operators on path-space]
 We recall the shift operators on $\Omega_{N}$, as follows: If $\omega=\left(i_{1},i_{2},i_{3},\cdots\right)\in\Omega_{N}$
, set 
\begin{align}
\sigma\left(\omega\right) & :=\left(i_{2},i_{3},i_{4},\cdots\right),\;\text{and}\label{eq:e7}\\
\hat{\tau}_{j}\left(\omega\right) & :=\left(j,i_{1},i_{2},i_{3},\cdots\right).\label{eq:e8}
\end{align}
If $E\subset\Omega_{N}$ is a subset, and $f$ is a finite word, we
set 
\begin{align}
\sigma\left(E\right) & =\left\{ \sigma\left(\omega\right)\mid\omega\in E\right\} ,\label{eq:e9}\\
\hat{\tau}_{j}\left(E\right) & =\left\{ \hat{\tau}_{j}\left(\omega\right)\mid\omega\in E\right\} ,\label{eq:e10}\\
\sigma^{-1}\left(E\right) & =\left\{ \omega\in\Omega_{N}\mid\sigma\left(\omega\right)\in E\right\} ,\;\text{and}\label{eq:e11}\\
fE & =\big\{\underbrace{f\omega}\mid\omega\in E\big\}\label{eq:e12}\\
 & \qquad\text{concatination of words}.\nonumber 
\end{align}
Note 
\begin{equation}
\sigma^{-1}\left(E\right)=\bigcup_{j=1}^{N}\hat{\tau}_{j}\left(E\right).\label{eq:e13}
\end{equation}
\end{defn}

\begin{lem}
~
\begin{enumerate}
\item The sample space $\Omega_{N}=\left\{ 1,2,\cdots,N\right\} ^{\mathbb{N}}$
is a compact metric space when equipped with the metric $d_{N}$ as
follows: Given $\omega,\omega'\in\Omega_{N}$, and set 
\[
k:=\sup\left\{ j\in\mathbb{N}\mid\omega_{i}=\omega_{i}',\;1\leq i\leq j\right\} 
\]
with the convention that $k=\infty$ iff $\omega=\omega'$. Then 
\begin{equation}
d_{N}\left(\omega,\omega'\right)=N^{-k}.\label{eq:e13a}
\end{equation}
\item The shift maps $\left\{ \hat{\tau}_{j}\right\} _{j=1}^{N}$ in (\ref{eq:e8})
are contractive as follows: 
\[
d_{N}\left(\hat{\tau}_{j}\left(\omega\right),\hat{\tau}_{j}\left(\omega'\right)\right)\leq N^{-1}d_{N}\left(\omega,\omega'\right)
\]
for all $1\leq j\leq N$, $\forall\omega,\omega'\in\Omega_{N}$. 
\end{enumerate}
\end{lem}

\begin{proof}
Uses standard facts about infinite products, and is left to the reader. 
\end{proof}
When discussing measures $\mathbb{Q}$ on $\left(\Omega_{N},\mathscr{C}\right)$,
we refer to $\sigma$-additivity; i.e., if $\left\{ E_{j}\right\} _{j\in\mathbb{N}}$,
$E_{j}\in\mathscr{C}$, $E_{i}\cap E_{j}=\emptyset$, $i\neq j$,
is given, we require that 
\[
\mathbb{Q}\left(\bigcup_{j=1}^{\infty}E_{j}\right)=\sum_{j=1}^{\infty}\mathbb{Q}\left(E_{j}\right).
\]

We say that a measure $\mathbb{Q}$ on $\left(\Omega_{N},\mathscr{C}\right)$
is \emph{projection valued}, i.e., $\mathbb{Q}\left(E\right)$ is
a projection in $\mathscr{H}$, for all $E\in\mathscr{C}$, and 
\begin{equation}
\mathbb{Q}\left(\Omega_{N}\right)=I,\quad\text{and}\quad\mathbb{Q}\left(\emptyset\right)=0.\label{eq:e15}
\end{equation}

Let $\mathbb{Q}$ be the projection valued measure on $\left(M,\mathscr{C}\right)$
taking values in $\mathscr{B}\left(\mathscr{H}\right)$ for a fixed
Hilbert space $\mathscr{H}$, and let $\psi\in\mathscr{H}$; we then
get a \emph{scalar valued} measure 
\[
\mu_{\psi}\left(E\right):=\left\langle \psi,\mathbb{Q}\left(E\right)\psi\right\rangle _{\mathscr{H}}=\left\Vert \mathbb{Q}\left(E\right)\psi\right\Vert _{\mathscr{H}}^{2},\;\forall E\in\mathscr{C}.
\]
Conversely, $\mathbb{Q}\left(\cdot\right)$ is determined by these
measures. 

In the discussion below, the projection valued measures will depend
on a prescribed (fixed) representation $\pi\in Rep\left(\mathcal{O}_{N},\mathscr{H}\right)$: 
\begin{thm}
\label{thm:pv}Given $\pi\in Rep\left(\mathcal{O}_{N},\mathscr{H}\right)$,
then there is a unique projection valued measure $\mathbb{Q}=\mathbb{Q}_{\pi}$
on $\left(\Omega_{N},\mathscr{C}\right)$ which is specified on the
basic cylinder sets $E_{f}$, $f=\left(i_{1},\cdots,i_{k}\right)$,
as follows: 
\begin{equation}
\mathbb{Q}_{\pi}\left(E_{f}\right)=S_{f}S_{f}^{*}=S_{i_{1}}\cdots S_{i_{k}}S_{i_{k}}^{*}\cdots S_{i_{1}}^{*}.\label{eq:e16}
\end{equation}
The measure satisfies the following properties: 
\begin{equation}
\beta_{\pi}\left(\mathbb{Q}_{\pi}\left(E\right)\right)=\mathbb{Q}_{\pi}\left(\sigma^{-1}\left(E\right)\right),\;\forall E\in\mathscr{C};\label{eq:e17}
\end{equation}
(see (\ref{eq:e2}) for the definition of $\beta_{\pi}$.)
\begin{equation}
S_{i}\mathbb{Q}_{\pi}\left(E\right)=\mathbb{Q}\left(iE\right)S_{i},\;\forall i\in A,\:\forall E\in\mathscr{C};\label{eq:e18}
\end{equation}
and 
\begin{equation}
S_{i}^{*}\mathbb{Q}_{\pi}\left(E\right)=\delta_{i,\pi_{1}\left(E\right)}\mathbb{Q}_{\pi}\left(\sigma\left(E\right)\right)S_{i}^{*}.\label{eq:e19}
\end{equation}
\end{thm}

\begin{proof}
We begin with $\mathbb{Q}_{\pi}$ defined initially only on the basic
cylinder sets $E_{f}$, $f\in\left\{ \text{finite words}\right\} $;
see (\ref{eq:e16}). To show that it extends to the full $\sigma$-algebra
$\mathscr{C}$, we make use of \emph{Kolmogorov's consistency principle}
(see \cite{MR735967,MR562914}). Specifically, we must check from
(\ref{eq:e16}) that 
\begin{equation}
\mathbb{Q}_{\pi}\left(E_{f}\right)=\sum_{i=1}^{N}\mathbb{Q}_{\pi}\left(E_{\left(fi\right)}\right)\label{eq:e20}
\end{equation}
where $E_{f}$ is one of basic cylinder-sets. But (\ref{eq:e20})
is immediate from: 
\[
S_{f}S_{f}^{*}=\sum_{i=1}^{N}S_{f}S_{i}S_{i}^{*}S_{f}^{*}=\sum_{i=1}^{N}S_{\left(fi\right)}S_{\left(fi\right)}^{*}.
\]
The Kolmogorov extension also implies that the values $\mathbb{Q}_{\pi}\left(E\right)$,
$E\in\mathscr{C}$, are determined by those on $E_{f}$, $f$ finite;
this is a standard inductive limit argument; see e.g., \cite{MR562914,MR735967,MR2393052,MR1278486,MR863346,MR0443014}. 

Hence, to verify that these three conditions (\ref{eq:e17})-(\ref{eq:e19})
in the theorem, we may restrict the checking to the case when $E$
has the form $E_{f}$, for some finite word $f=\left(i_{1},\cdots,i_{k}\right)$
fixed. 

The argument for (\ref{eq:e17}) is: 
\[
\sum_{i=1}^{N}S_{i}S_{f}S_{f}^{*}S_{i}^{*}=\sum_{i=1}^{N}S_{\left(if\right)}S_{\left(if\right)}^{*}=\mathbb{Q}_{\pi}\left(\sigma^{-1}\left(E_{f}\right)\right).
\]

The argument for (\ref{eq:e18}) is: 
\[
S_{i}\left(S_{f}S_{f}^{*}\right)=\left(S_{\left(if\right)}S_{\left(if\right)}^{*}\right)S_{i};
\]
and finally the argument for (\ref{eq:e19}) is: 
\[
S_{j}^{*}S_{f}S_{f}^{*}=\delta_{ji_{1}}S_{\left(i_{2},\cdots,i_{k}\right)}S_{\left(i_{2},\cdots,i_{k}\right)}^{*}S_{j}^{*}.
\]

When these identities are combined with the \emph{Kolmogorov consistency}
/ inductive limit arguments \cite{MR735967,MR562914}, the conclusions
of the theorem now follow. We turn to the details of this in Section
\ref{subsec:kol} below.
\end{proof}
\begin{defn}
A projection-valued measure $\mathbb{Q}$ on $\left(\Omega_{N},\mathscr{C}\right)$,
taking values in $\mathscr{B}\left(\mathscr{H}\right)$, is said to
be \emph{orthogonal }iff (Def.) 
\begin{equation}
\mathbb{Q}\left(E\cap E'\right)=\mathbb{Q}\left(E\right)\mathbb{Q}\left(E'\right)\label{eq:f1}
\end{equation}
for all sets $E$ and $E'$ in $\mathscr{C}$. 
\end{defn}

\begin{rem}
The condition in (\ref{eq:f1}) is called \emph{orthogonality} because
of the following: If (\ref{eq:f1}) is satisfied, and if $E\cap E'=\emptyset$
where $E$ and $E'$ are picked from $\mathscr{C}$ (the $\sigma$-algebra),
then 
\begin{equation}
\mathbb{Q}_{\pi}\left(E\right)\mathscr{H}\perp\mathbb{Q}_{\pi}\left(E'\right)\mathscr{H}.\label{eq:f2}
\end{equation}
To see this, compute the inner products of vectors $h,h'\in\mathscr{H}$:
\begin{align*}
\left\langle \mathbb{Q}_{\pi}\left(E\right)h,\mathbb{Q}_{\pi}\left(E'\right)h'\right\rangle  & =\left\langle h,\mathbb{Q}_{\pi}\left(E\right)\mathbb{Q}_{\pi}\left(E'\right)h'\right\rangle \\
 & =\langle h,\mathbb{Q}_{\pi}(\underset{=\emptyset}{\underbrace{E\cap E'})}h'\rangle=\left\langle h,0h'\right\rangle =0,
\end{align*}
which is the orthogonality (\ref{eq:f2}).
\end{rem}

\subsection{\label{subsec:kol}The Kolmogorov consistency construction}

Fix $N>1$, and a Hilbert space $\mathscr{H}$. Let $\pi\in Rep\left(\mathcal{O}_{N},\mathscr{H}\right)$,
$\pi\left(s_{i}\right)=S_{i}$, $1\leq i\leq N$. Let $\Omega=\Omega_{N}:=\left(\mathbb{Z}_{N}\right)^{\mathbb{N}}$
(= the infinite Cartesian product). For $\omega\in\Omega$, and $k\in\mathbb{N}$,
set $\omega|_{k}=\left(\omega_{1},\omega_{2},\cdots,\omega_{k}\right)$,
the truncated word. Let $C\left(\Omega\right):=$ all continuous functions
on $\Omega$. Set 
\begin{equation}
\mathscr{F}_{k}=\left\{ F\in C\left(\Omega\right)\mid F\left(\omega\right)=F\left(\omega|_{k}\right)\right\} ,\label{eq:f21}
\end{equation}
i.e., $\mathscr{F}_{k}$ consists of functions depending on only the
first $k$ coordinates. $\mathscr{F}_{0}=$ the constant functions
on $\Omega$. Finally, we set 
\begin{equation}
\mathscr{F}_{\infty}:=\bigcup_{k=0}^{\infty}\mathscr{F}_{k}.\label{eq:f22}
\end{equation}

\begin{lem}
With the notation from above, $\mathscr{F}_{\infty}$ is a dense subalgebra
in $C\left(\Omega\right)$, i.e., dense in the uniform norm on $C\left(\Omega\right)$. 
\end{lem}

\begin{proof}[Proof sketch]
 The conclusion follows from the Stone-Weierstrass theorem \cite{MR0243367}:
We only need to show that $\mathscr{F}_{\infty}$ is an algebra, contains
the constant function $\mathbbm{1}$, and separates points in $\Omega$. 

But the properties are immediate from (\ref{eq:f21}). Indeed, if
$\omega,\omega'\in\Omega$, and $\omega\neq\omega'$. Pick $k$ such
that $\omega_{k}\neq\omega'_{k}$; then take $F=\pi_{k}$ (= the coordinate
projection); it is in $\mathscr{F}_{k}$, and satisfies $F\left(\omega\right)\neq F\left(\omega'\right)$. 
\end{proof}
\begin{proof}[Proof of Theorem (\ref{thm:pv}) continued]
 We now turn to the projection-valued measure $\mathbb{Q}_{\pi}$,
defined initially only for $\mathscr{F}_{\infty}$. We define $\mathbb{Q}_{\pi}$
as a positive linear functional, taking values in the projections
in $\mathscr{H}$; see Fig \ref{fig:proj}. 

\begin{figure}[H]
\begin{tabular}{|c|c|c|c|c|c|c|c|}
\hline 
\multicolumn{8}{|c|}{$\mathscr{H}$}\tabularnewline
\hline 
\multicolumn{4}{|c|}{$P_{0}$} & \multicolumn{4}{c|}{$P_{1}$}\tabularnewline
\hline 
\multicolumn{2}{|c|}{$P_{00}$} & \multicolumn{2}{c|}{$P_{01}$} & \multicolumn{2}{c|}{$P_{10}$} & \multicolumn{2}{c|}{$P_{11}$}\tabularnewline
\hline 
$P_{000}$ & $P_{001}$ & $P_{001}$ & $P_{011}$ & $P_{100}$ & $P_{101}$ & $P_{110}$ & $P_{111}$\tabularnewline
\hline 
\multicolumn{1}{c}{$\vdots$} & \multicolumn{1}{c}{$\vdots$} & \multicolumn{1}{c}{$\vdots$} & \multicolumn{1}{c}{$\vdots$} & \multicolumn{1}{c}{$\vdots$} & \multicolumn{1}{c}{$\vdots$} & \multicolumn{1}{c}{$\vdots$} & \multicolumn{1}{c}{$\vdots$}\tabularnewline
\end{tabular}
\[
P\left(i_{1},i_{2},\cdots,i_{k}\right)=S_{i_{1}}\cdots S_{i_{k}}S_{i_{k}}^{*}\cdots S_{i_{1}}=S_{I}S_{I}^{*}
\]

\caption{\label{fig:proj}Multiresolution as a nested family of projections.}
\end{figure}

In detail: If $F\in\mathscr{F}_{k}$, set 
\begin{equation}
\mathbb{Q}_{\pi}^{\left(k\right)}\left(F\right)=\sum_{\substack{I=\left(i_{1},\cdots,i_{k}\right)\\
\in\mathbb{Z}_{N}^{k}
}
}F\left(I\right)S_{I}S_{I}^{*}.\label{eq:f23}
\end{equation}
To show that $\mathbb{Q}_{\pi}^{\left(k\right)}$, as defined in (\ref{eq:f23})
is positive, we need to check that 
\begin{equation}
\mathbb{Q}_{\pi}^{\left(k\right)}\left(F^{2}\right)\geq0.\label{eq:f24}
\end{equation}
(Recall, we have restricted the checking to real valued functions,
but this can easily be modified to apply to the complex valued case.
In that case, we must consider $\mathbb{Q}_{\pi}^{\left(k\right)}\left(\left|F\right|^{2}\right)$
in (\ref{eq:f24}).) 

For $I,J\in\left(\mathbb{Z}_{N}\right)^{k}$, we have 
\begin{equation}
S_{I}S_{I}^{*}S_{J}S_{J}^{*}=\delta_{I,J}S_{I}S_{I}^{*}\label{eq:f25}
\end{equation}
where $\delta_{I,J}=\prod_{l=1}^{k}\delta_{i_{l},j_{l}}$. Now, combining
(\ref{eq:f23}) and (\ref{eq:f25}), we get 
\begin{align*}
\mathbb{Q}_{\pi}^{\left(k\right)}\left(F\right)^{2} & =\sum_{I}\sum_{J}F\left(F\right)F\left(J\right)S_{I}S_{I}^{*}S_{J}S_{J}^{*}\\
 & =\sum_{I}F\left(I\right)^{2}S_{I}S_{I}^{*}=\mathbb{Q}_{\pi}^{\left(k\right)}\left(F^{2}\right),
\end{align*}
and the desired positivity (\ref{eq:f24}) follows. 

To get the desired \emph{Kolmogorov extension} (see \cite{MR735967,MR562914}),
we only need to check \emph{consistency}: Let $F\in\mathscr{F}_{k}\subseteq\mathscr{F}_{k+1}$,
i.e., $F$ is considered as a function on $\left(\mathbb{Z}_{N}\right)^{k+1}$,
but constant in the last variable $i_{k+1}$. 

We now have:
\begin{equation}
\mathbb{Q}_{\pi}^{\left(k+1\right)}\left(F\right)=\mathbb{Q}_{\pi}^{\left(k\right)}\left(F\right).\label{eq:f26}
\end{equation}
Indeed, 
\begin{align*}
\text{LHS}_{\left(\ref{eq:f26}\right)} & =\sum_{I\in\left(\mathbb{Z}_{N}\right)^{k}}\sum_{j\in\mathbb{Z}_{N}}F\left(I\right)S_{I}S_{j}S_{j}^{*}S_{I}^{*}\\
 & =\sum_{I\in\left(\mathbb{Z}_{N}\right)^{k}}F\left(I\right)S_{I}S_{I}^{*}=\text{RHS}_{\left(\ref{eq:f26}\right)},
\end{align*}
since $\sum_{j}S_{j}S_{j}^{*}=I$ by (\ref{eq:a7}).

Now Kolmogorov consistency, and an application of the Riesz representation
theorem (see \cite{MR0243367}), yields the final conclusion: The
projection valued measure $\mathbb{Q}_{\pi}$ arises as a projective
limit of the individual measures $\left(\mathbb{Q}_{\pi}^{\left(k\right)}\left(\cdot\right),\mathscr{F}_{k}\right)$
introduced above in (\ref{eq:f23}).
\end{proof}
\begin{rem}
Consider the family $\left\{ \mathscr{F}_{k}\right\} _{k\in\mathbb{N}_{0}}$
in (\ref{eq:f21}). By abuse of notation, we may also consider this
\emph{as a family of $\sigma$-algebras}, i.e., 
\begin{equation}
\mathscr{F}_{k}=\left(\text{the \ensuremath{\sigma}-algebra generated by \ensuremath{\left\{  \pi_{1},\pi_{2},\cdots,\pi_{k}\right\} } }\right),\label{eq:f27}
\end{equation}
see (\ref{eq:a2}). Moreover, $\mathscr{C}=\bigvee_{k}\mathscr{F}_{k}$,
where we use the lattice operation for $\sigma$-algebras. 

From (\ref{eq:f26}), we obtain the projection valued measure $\mathbb{Q}_{\pi}$
as a solution to the problem
\begin{equation}
\mathbb{Q}_{\pi}^{\left(k\right)}\left(\cdot\cdot\right)=\mathbb{Q}_{\pi}\left(\cdot\cdot\mid\mathscr{F}_{k}\right)\label{eq:f28}
\end{equation}
where ``$\quad\mid\mathscr{F}_{k}$'' refers to conditional expectation. 

Hence the solution $\mathbb{Q}_{\pi}\left(\cdot\right)$ may be viewed
as a \emph{martingale limit}: We have for all $k,l\in\mathbb{N}$,
$k<l$: 
\begin{equation}
\mathbb{Q}_{\pi}\left(\cdot\cdot\mid\mathscr{F}_{k}\right)=\mathbb{Q}_{\pi}\left(\cdot\cdot\mid\mathscr{F}_{l}\mid\mathscr{F}_{k}\right);\label{eq:f29}
\end{equation}
and for all measurable functions $F$ on $\left(\Omega,\mathscr{C}\right)$,
we have 
\[
\mathbb{Q}_{\pi}\left(F\right)=\lim_{k\rightarrow\infty}\mathbb{Q}_{\pi}^{\left(k\right)}\left(F\right)=\lim_{k\rightarrow\infty}\mathbb{Q}_{\pi}\left(F\mid\mathscr{F}_{k}\right)
\]
where 
\[
\mathbb{Q}_{\pi}\left(F\right):=\int_{\Omega}F\left(\omega\right)\mathbb{Q}_{\pi}\left(d\omega\right).
\]
\end{rem}

\begin{rem}
Let $E\subset\Omega$, and assume $E\in\mathscr{F}_{k}=\sigma$-algebra($\left\{ \pi_{i}\right\} _{i=1}^{k}$): 

Let $j\in\mathbb{Z}_{N}$. Then $Ej:=\bigcup_{e\in E}\left(ej\right)\in\mathscr{F}_{k+1}$,
and 
\begin{equation}
\mathbb{Q}_{\pi}\left(E\right)=\sum_{j\in\mathbb{Z}_{N}}\mathbb{Q}_{\pi}\left(Ej\right);
\end{equation}
but, in general,
\begin{equation}
\sum_{i\in\mathbb{Z}_{N}}\mathbb{Q}_{\pi}\left(iE\right)\neq\mathbb{Q}_{\pi}\left(E\right).\label{eq:f2a}
\end{equation}

Note in general, 
\begin{equation}
\bigcup_{i\in\mathbb{Z}_{N}}iE=\sigma^{-1}\left(E\right),
\end{equation}
(as a disjoint union on the left hand side) where $\sigma$ is the
\emph{shift} in $\Omega$; see (\ref{eq:e7}) and (\ref{eq:e12}).
So the assertion in (\ref{eq:f2a}) above is that, in general, we
may have: 
\[
\mathbb{Q}_{\pi}\left(\sigma^{-1}E\right)\neq\mathbb{Q}_{\pi}\left(E\right).
\]
\end{rem}

\begin{cor}
\label{cor:QO}Let $\pi\in Rep\left(\mathcal{O}_{N},\mathscr{H}\right)$,
and let $\mathbb{Q}_{\pi}$ be the corresponding projection valued
measure introduced in Theorem \ref{thm:pv}. Then $\mathbb{Q}_{\pi}$
is orthogonal, i.e., (\ref{eq:f1}) holds. 
\end{cor}

\begin{proof}
Because of the Kolmogorov-consistency construction, it is enough to
verify the orthogonality (\ref{eq:f1}) for $\mathbb{Q}_{\pi}$ in
the special case when the two sets have the form $E_{f}$, $E_{g}$,
where $f$ and $g$ are \emph{finite} words in the alphabet, say $f=\left(i_{1},i_{2},\cdots,i_{k}\right)$
and $g=\left(j_{1},j_{2},\cdots,j_{l}\right)$ where $k$ and $l$
denote the respective word lengths. We say that containment holds
for the two words iff one of the two contains the other in the following
manner: say $f\subseteq g$, if $k\leq l$ and $i_{1}=j_{1}$, $\cdots$,
$i_{k}=j_{k}$. In this case $g=\left(fh\right)$ where $h$ is the
tail end in the word $g$. (There is a symmetric condition when instead
$g\subseteq f$.)

When $f\subseteq g$, then 
\begin{equation}
E_{f}\cap E_{g}=E_{g}.\label{eq:f3}
\end{equation}
Hence we must verify that, in this case, 
\begin{equation}
\mathbb{Q}_{\pi}\left(E_{g}\right)=\mathbb{Q}_{\pi}\left(E_{f}\right)\mathbb{Q}_{\pi}\left(E_{g}\right).\label{eq:f4}
\end{equation}
But using $g=\left(fh\right)$ for some finite word $h$, we get for
the RHS in (\ref{eq:f4}): 
\begin{align*}
\text{RHS}_{\left(\ref{eq:f4}\right)} & =S_{f}\left(S_{f}^{*}S_{f}\right)S_{h}S_{h}^{*}S_{f}^{*}=S_{f}S_{h}S_{h}^{*}S_{f}^{*}\\
 & =S_{g}S_{g}^{*}\quad\left(\text{since }g=fh\right)\\
 & =\mathbb{Q}_{\pi}\left(E_{g}\right)=\text{LHS}_{\left(\ref{eq:f4}\right)}
\end{align*}
and the desired conclusion follows.

The remaining case is, if none of the possible containment holds,
i.e., $f$ not contained in $g$, and $g$ not contained in $f$.
In this case, $E_{f}\cap E_{g}=\emptyset$, and so both sides in equation
(\ref{eq:f4}) are zero. See also Fig \ref{fig:cyl2}.

Having verified that $\mathbb{Q}_{\pi}$ satisfies condition (\ref{eq:f1})
for basic cylinder-sets, it now follows that it must also hold for
all pairs of sets $E,E'\in\mathscr{C}$. This is an application of
the Kolmogorov extension principle. The proof of the theorem is concluded.
\end{proof}
\begin{cor}
Let the setting be as above, $\pi\in Rep\left(\mathcal{O}_{N},\mathscr{H}\right)$,
and let $\mathbb{Q}_{\pi}\left(\cdot\right)$ be the corresponding
projection valued measure. 
\begin{enumerate}
\item \label{enu:m1}For $\omega\in\Omega_{N}$, and $k\in\mathbb{N}$,
set $Z_{k}\left(\omega\right)=\omega_{k}\left(\in A_{N}\simeq\left\{ 1,2,\cdots,N\right\} \right)$,
then the following projection-valued Markov property holds: Let $k>1$,
then 
\begin{equation}
\text{\ensuremath{\mathbb{P}}rob}^{\left(\pi\right)}\left(Z_{k+1}=j\mid Z_{k}=i\right)=\beta_{\pi}^{k-1}\left(S_{i}S_{j}S_{j}^{*}S_{i}^{*}\right),\label{eq:f4a}
\end{equation}
where $\beta_{\pi}$ is the endomorphism in Definition \ref{def:end}
(eq. (\ref{eq:e2})).
\item \label{enu:m2}If $\psi\in\mathscr{H}$, $\left\Vert \psi\right\Vert =1$,
let $\mu_{\psi}\left(\cdot\right):=\left\langle \psi,\mathbb{Q}_{\pi}\left(\cdot\right)\psi\right\rangle _{\mathscr{H}}$
be the corresponding scalar valued measure. Then the associated transition
probabilities are 
\begin{align}
\text{\ensuremath{\mathbb{P}}rob}^{\left(\mu_{\psi}\right)}\left(Z_{k+1}=j\mid Z_{k}=i\right) & =\frac{\left\langle \psi,\beta_{\pi}^{k-1}\left(S_{i}S_{j}S_{j}^{*}S_{i}^{*}\right)\psi\right\rangle _{\mathscr{H}}}{\left\langle \psi,\beta_{\pi}^{k-1}\left(S_{i}S_{i}^{*}\right)\psi\right\rangle _{\mathscr{H}}}\nonumber \\
 & =\frac{\left\Vert \beta_{\pi}^{k-1}\left(S_{j}^{*}S_{i}^{*}\right)\psi\right\Vert _{\mathscr{H}}^{2}}{\left\Vert \beta_{\pi}^{k-1}\left(S_{i}^{*}\right)\psi\right\Vert _{\mathscr{H}}^{2}}.\label{eq:f4b}
\end{align}
\item \label{enu:m3}The Markov property holds for the process in (\ref{enu:m2})
if and only if $\beta_{\pi}$-invariance holds, in the following sense:
\[
\left\langle \psi,\beta_{\pi}\left(\mathbb{Q}_{\pi}\left(\cdot\right)\right)\psi\right\rangle _{\mathscr{H}}=\left\langle \psi,\mathbb{Q}_{\pi}\left(\cdot\right)\psi\right\rangle _{\mathscr{H}}=\mu_{\psi}\left(\cdot\right).
\]
\end{enumerate}
\end{cor}

\begin{proof}
For (\ref{eq:f4a}), we have 
\begin{align*}
\text{\ensuremath{\mathbb{P}}rob}^{\left(\pi\right)}\left(Z_{k+1}=j\mid Z_{k}=i\right) & =\sum_{I\in\mathbb{Z}_{N}^{k-1}}\mathbb{Q}_{\pi}\left(E\left(Iij\right)\right)\\
 & =\sum_{I\in\mathbb{Z}_{N}^{k-1}}S_{I}S_{i}S_{j}S_{j}^{*}S_{i}^{*}S_{I}^{*}\underset{\left({\scriptscriptstyle \text{by \ensuremath{\left(\ref{eq:e2}\right)}}}\right)}{=}\beta_{\pi}^{k-1}\left(S_{i}S_{j}S_{j}^{*}S_{i}^{*}\right).
\end{align*}

Parts (\ref{enu:m2}) and (\ref{enu:m3}) follow immediately from
this. 
\end{proof}
\begin{flushleft}
\textbf{Monic Representations.}
\par\end{flushleft}

Let $\pi\in Rep\left(\mathcal{O}_{N},\mathscr{H}\right)$, and let
$\mathbb{Q}_{\pi}$ be the corresponding projection valued measure.
Let $\mathfrak{M}_{\pi}$ be the abelian $*$-algebra generated by
$\mathbb{Q}_{\pi}$, i.e., the operators 
\begin{equation}
\int_{\Omega_{N}}f\left(\omega\right)\mathbb{Q}_{\pi}\left(d\omega\right)\label{eq:f5}
\end{equation}
where $f$ ranges over the measurable functions on $\left(\Omega_{N},\mathscr{C}\right)$. 

Following \cite{MR3217056}, we make the following:
\begin{defn}
We say that $\pi$ is \emph{monic} iff (Def.) there is a vector $\psi_{0}\in\mathscr{H}$,
$\left\Vert \psi_{0}\right\Vert =1$, such that 
\begin{equation}
\left[\mathfrak{M}_{\pi}\psi_{0}\right]=\mathscr{H},\label{eq:f6}
\end{equation}
i.e., $\psi_{0}$ is $\mathfrak{M}_{\pi}$-cyclic. 
\end{defn}

Starting with $\mathbb{Q}_{\pi}$ and (\ref{eq:f6}), we use the construction
outlined before Theorem \ref{thm:pv}, to get a scalar measure via:
\begin{equation}
\mu_{0}\left(E\right)=\left\langle \psi_{0},\mathbb{Q}_{\pi}\left(E\right)\psi_{0}\right\rangle _{\mathscr{H}},\;E\in\mathscr{C}.
\end{equation}
Using \cite{MR3217056}, we then get a\emph{ random variable} $Y:\Omega_{N}\rightarrow M$
for a measure space $\left(M,\mathscr{B}\right)$ such that the measure
$\mu:=\mu_{0}\circ Y^{-1}$ satisfies the conditions listed below: 

It was proved in \cite{MR3217056} that a representation $\pi\left(\in Rep\left(\mathcal{O}_{N},\mathscr{H}\right)\right)$
is monic iff it is unitarily equivalent to one realized in $L^{2}\left(M,\mu\right)$
as follows for some measure space $\left(M,\mu\right)$: 

There are endomorphisms $\left(\left\{ \tau_{i}\right\} _{i=1}^{N},\sigma\right)$,
such that $\sigma\circ\tau_{i}=id_{M}$, $\mu\circ\tau_{i}^{-1}\ll\mu$,
and $L^{2}\left(\mu\right)$-function $f_{i}$ on $M$, such that
\begin{align}
\frac{d\left(\mu\circ\tau_{i}^{-1}\right)}{d\mu} & =\left|f_{i}\right|^{2},\;1\leq i\leq N,\label{eq:f7}\\
f_{i}\neq0 & \quad a.e.\;\mu\;\text{in }\tau_{i}\left(M\right).\label{eq:f8}
\end{align}
Then the isometries $S_{i}$ are as follows: 
\begin{equation}
S_{i}^{\left(\mu\right)}\varphi=f_{i}\left(\varphi\circ\sigma\right),\;1\leq i\leq N,\label{eq:f9}
\end{equation}
i.e., $\left\{ S_{i}^{\left(\mu\right)}\right\} _{i=1}^{N}\in Rep\left(\mathcal{O}_{N},L^{2}\left(\mu\right)\right)$;
see (\ref{eq:a7}).

\subsection{Atoms of the Path-Space Measure $\mathbb{Q}_{\pi}\left(\cdot\right)$}

In Section \ref{sec:PVR} we introduced a random variable $X$ on
path space $\left(\Omega_{N},\mathscr{C}\right)$; and in Section
\ref{sec:PVM}, a \emph{path-space-measure} $\mathbb{Q}_{\pi}\left(\cdot\right)$.
The starting point in both cases is a fixed $\pi\in Rep\left(\mathcal{O}_{N},\mathscr{H}\right)$.
Specifically, a separable Hilbert space $\mathscr{H}$ is given, $N\in\mathbb{N}$,
$N\geq2$, fixed; and $\pi$ is a representation of $\mathcal{O}_{N}$,
$\pi\in Rep\left(\mathcal{O}_{N},\mathscr{H}\right)$, $\pi\left(s_{i}\right)=:S_{i}$,
$1\leq i\leq N$, with the isometries $S_{i}:\mathscr{H}\rightarrow\mathscr{H}$
satisfying (\ref{eq:a7}). To summarize, the random variable $X=X_{\pi}$
is specified in (\ref{eq:a8}) in Theorem \ref{thm:LP}, and the path-space
measure $\mathbb{Q}_{\pi}$ in Theorem \ref{thm:pv}. Both take values
in the projections in $\mathscr{H}$, see Section \ref{sec:Pre},
and also \cite{MR3402823,MR3687240,MR3796644}.

The question addressed here is: \emph{What are the atoms of $\mathbb{Q}_{\pi}$?} 

We say that a sample path $\omega\in\Omega_{N}$ is an atom if the
singleton $\left\{ \omega\right\} $ satisfies $\mathbb{Q}\left(\left\{ \omega\right\} \right)>0$;
so the closed subspace $\mathbb{Q}_{\pi}\left(\left\{ \omega\right\} \right)\mathscr{H}=\mathscr{H}_{\omega}$
is non-zero. The answer to the question is given in the corollary
below where we prove the following: 
\begin{equation}
\mathbb{Q}_{\pi}\left(\left\{ \omega\right\} \right)=X_{\pi}\left(\omega\right),\;\forall\omega\in\Omega_{N}.\label{eq:h1}
\end{equation}

We note that (\ref{eq:h1}) holds even if one of the two sides (and
hence both) is zero. 

\emph{Details: }
\begin{cor}
We have: 
\begin{equation}
X_{\pi}\left(\omega\right)=\mathbb{Q}_{\pi}\left(\left\{ \omega\right\} \right),\;\forall\omega\in\Omega_{N}.\label{eq:h2}
\end{equation}
\end{cor}

\begin{proof}
Fix $\omega\in\Omega_{N}$, then 
\begin{equation}
\left\{ \omega\right\} =\bigcap_{k=1}^{\infty}E\left(\omega|_{k}\right);\label{eq:h3}
\end{equation}
see also Fig \ref{fig:cyl2}. To see (\ref{eq:h3}), note that 
\begin{equation}
E\left(\omega|_{k}\right)=\left\{ \xi\in\Omega_{N}\mid\xi_{i}=\omega_{i},\;1\leq i\leq k\right\} ,\label{eq:h4}
\end{equation}
so that $\xi\in\bigcap_{k}E\left(\omega|_{k}\right)\Longleftrightarrow\xi=\omega$
(see Fig \ref{fig:cyl2}).
\end{proof}
\begin{rem}
The notation in (\ref{eq:h3}) is consistent with \emph{convergence
along a filter} (see e.g. \cite{MR1726779,MR2048350}) as follows:
Let $\Omega_{N}=A^{\mathbb{N}}$ be as above where $A$ is a given
(and fixed) alphabet. If $f=\left(x_{1},x_{2},\cdots,x_{k}\right)$,
$x_{i}\in A$, is a finite path, we introduced the sets $E_{f}$ (or
$E\left(f\right)$) where 
\begin{equation}
E_{f}:=\left\{ \omega\in\Omega_{N}\mid\omega_{i}=x_{i},\;1\leq i\leq k\right\} .\label{eq:h5}
\end{equation}
In particular, if $\omega\in\Omega_{N}$, $k\in\mathbb{N}$, set $f=\omega|_{k}=\left(\omega_{1},\omega_{2},\cdots,\omega_{k}\right)$
and we get the sets in (\ref{eq:h4}). 

Now consider the \emph{filter} $\mathscr{F}$ of subsets of $\Omega_{N}$,
defined as follows: 
\begin{equation}
\text{A subset \ensuremath{E} is in \ensuremath{\mathscr{F}} iff (Def.) \ensuremath{\big[}\ensuremath{\ensuremath{\exists}f} (a finite word) such that \ensuremath{E_{f}\subseteq E}\big].}\label{eq:h6}
\end{equation}
(We say that $\left\{ E_{f}\right\} $ forms a \emph{filter basis}.)

The sets in $\mathscr{F}$ satisfy the two filter axioms: 
\end{rem}

\begin{enumerate}
\item If $E_{i}$, $i=1,2$, are in $\mathscr{F}$, then $\exists F\in\mathscr{F}$
such that $F\subseteq E_{i}$, for $i=1,2$. 
\item If $E\in\mathscr{F}$, and $F$ is a subset of $\Omega_{N}$ such
that $E\subseteq F$, then $F\in\mathscr{F}$. 
\end{enumerate}
Now condition (\ref{eq:h3}) above is the assertion that 
\begin{equation}
\lim_{k\:\mathscr{F}}\omega|_{k}=\omega;\label{eq:h7}
\end{equation}
stating that $E\left(\omega|_{k}\right)$, $k\in\mathbb{N}$, \emph{converges}
to $\left\{ \omega\right\} $ \emph{along the filter} $\mathscr{F}$. 

\subsection{\label{subsec:ss}Symbol Space Representations as Groups}

In the study of iterated function systems (IFSs), and more generally,
in symbolic dynamics, we consider a fixed finite alphabet $A$, as
well as words in $A$. Both \emph{finite} as well as \emph{infinite}
words are needed. For many purposes, it is helpful to give $A$ in
the form of a cyclic group $\mathbb{Z}\big/N\mathbb{Z}\simeq\left\{ 0,1,2,\cdots,N-1\right\} $.
In this case both the finite words $\Omega_{N}^{*}$, as well as infinite
words $\Omega_{N}:=A^{\mathbb{N}}$ become groups. In the representation
below, we identify $\Omega_{N}^{*}$, and $\Omega_{N}$, as a pair
of abelian groups in duality. Since $\Omega_{N}^{*}$ (finite words)
is discrete, we get $\Omega_{N}$ realized as a compact abelian group. 
\begin{lem}
\label{lem:DU}Let $N\in\mathbb{N}$, $N\geq2$, be fixed, and let
$\Omega_{N}^{*}$, resp. $\Omega_{N}$, denote the finite, resp.,
infinite words in $\mathbb{Z}_{N}\simeq\mathbb{Z}\big/N\mathbb{Z}$. 
\begin{enumerate}
\item If $x=\left(x_{j}\right)_{j=1}^{\infty}\in\Omega_{N}$, and $y=\left(y_{j}\right)_{j=1}^{finite}\in\Omega_{N}^{*}$,
are fixed, then set 
\begin{equation}
\left\langle x,y\right\rangle :=\prod_{k=1}^{\infty}\exp\left(i2\pi\left(\frac{x_{k}y_{k}}{N^{k}}\right)\right),\label{eq:j1}
\end{equation}
so we have 
\begin{align}
\left\langle x+x',y\right\rangle  & =\left\langle x,y\right\rangle \left\langle x',y\right\rangle ,\;\text{and}\label{eq:j2}\\
\left\langle x,y+y'\right\rangle  & =\left\langle x,y\right\rangle \left\langle x,y'\right\rangle ,\label{eq:j3}
\end{align}
for all $x,x'\in\Omega_{N}$, and $y,y'\in\Omega_{N}^{*}$. 
\item In the category of abelian groups, we get 
\begin{alignat}{2}
dual\left(\Omega_{N}\right) & =\Omega_{N}^{*}, & \; & \text{and}\label{eq:j4}\\
dual\left(\Omega_{N}^{*}\right) & =\Omega_{N}, &  & \text{where}\label{eq:j5}
\end{alignat}
``dual'' refers to Pontryagin duality. Note $\Omega_{N}^{*}=\bigcup_{k=1}^{\infty}N^{-k}\mathbb{Z}$;
and 
\begin{equation}
\mathbb{Z}\subset N^{-1}\mathbb{Z}\subset N^{-2}\mathbb{Z}\subset\cdots\subset N^{-k}\mathbb{Z}\subset N^{-\left(k+1\right)}\mathbb{Z}\subset\cdots.\label{eq:j51}
\end{equation}
\item The Haar measure on $\Omega_{N}$ is the infinite product norm on
$\left(\mathbb{Z}_{N}\right)^{\mathbb{N}}$ with weights $\left(\frac{1}{N},\frac{1}{N},\cdots,\frac{1}{N}\right)$
on each factor. 
\end{enumerate}
\end{lem}

\begin{proof}
The lemma follows from results in the literature (see \cite{MR3250475,MR3394108}),
and is left to the reader. 

Identify a finite word $y=\left(y_{1},\cdots,y_{k}\right)\in\Omega_{N}^{*}$
($y_{j}\in\mathbb{Z}_{N}=\left\{ 0,1,\cdots,N-1\right\} $) with 
\begin{align}
\tilde{y} & =\frac{y_{1}N^{k-1}+\cdots+y_{k-1}N+y_{k}}{N^{k}}\label{eq:j52}\\
 & =y_{1}/N+\cdots+y_{k}/N^{k}\in N^{-k}\mathbb{Z};\nonumber 
\end{align}
see (\ref{eq:j51}). Set 
\begin{equation}
S_{y}=S_{y_{1}}S_{y_{2}}\cdots S_{y_{k}};\label{eq:j53}
\end{equation}
and if $x=\left(x_{j}\right)_{j=1}^{\infty}\in\Omega_{N}$ ($x_{j}\in\mathbb{Z}_{N}$),
define an automorphism action $\alpha\left(x\right)$ of $\mathcal{O}_{N}$
by its values on generators $S_{y}$ as follows: 
\begin{equation}
\alpha\left(x\right)S_{y}=\left\langle x,y\right\rangle S_{y};\label{eq:j54}
\end{equation}
called the gauge-action. 

In particular, 
\begin{equation}
\alpha\left(x\right)\left(S_{y}S_{y}^{*}\right)=S_{y}S_{y}^{*}.\label{eq:j55}
\end{equation}
The abelian $*$-subalgebra $\mathfrak{M}_{N}$ in $\mathcal{O}_{N}$
generated by the projections $\left\{ S_{y}S_{y}^{*}\right\} _{y\in\Omega_{N}^{*}}$
($\Omega_{N}^{*}=$ finite words) is $\mathfrak{M}_{N}=\left\{ M\in\mathcal{O}_{N}\mid\alpha\left(x\right)M=M,\;\forall x\in\Omega_{N}\right\} $. 
\end{proof}
\begin{rem}
It follows from Lemma \ref{lem:DU} that the projection valued measures
from Theorems \ref{thm:LP} and \ref{thm:pv} may be realized on the
compact group $\Omega_{N}$. 

For the study of Markov chains, the following extension of the lemma
will be useful: 
\end{rem}

\begin{lem}
Let $M$ be a fixed $N\times N$ matrix over $\mathbb{Z}$, and assume
its eigenvalues $\lambda_{j}$ satisfy $\left|\lambda_{j}\right|>1$. 

From the nested chain of groups we then obtain inductive, and projective
limits, in the form of discrete groups $\Omega_{M}^{*}$, and compact
dual $\Omega_{M}$. 

Case 1 (inductive)
\begin{equation}
\mathbb{Z}^{N}\big/M^{k+1}\mathbb{Z}^{N}\hookrightarrow\mathbb{Z}^{N}\big/M^{k}\mathbb{Z}^{N}\label{eq:j7}
\end{equation}
and the dual projective group formed from the groups
\begin{equation}
\left(M^{T}\right)^{k}\mathbb{Z}^{N}\label{eq:j8}
\end{equation}
where $M^{T}$ denotes the transposed matrix:
\[
\Omega_{M}^{*}=\bigcup_{k=1}^{\infty}M^{-k}\left(\mathbb{Z}_{N}\right);
\]
and note $\mathbb{Z}_{N}\subset M^{-1}\mathbb{Z}_{N}\subset M^{-2}\mathbb{Z}_{N}\subset\cdots\subset M^{-k}\mathbb{Z}_{N}\subset M^{-\left(k+1\right)}\mathbb{Z}_{N}\subset\cdots$. 

When $M$ is fixed, and pair $x=\left(x_{j}\right)$ and $y=\left(y_{j}\right)$
are infinite, resp., finite, words in $\mathbb{Z}^{N}\big/M\mathbb{Z}^{N}$,
then the Pontryagin duality is then 
\begin{equation}
\left\langle x,y\right\rangle _{M}:=\prod_{k=1}^{\infty}\exp\left(i2\pi\left(M^{T}\right)^{-k}x_{j}\cdot y_{j}\right).\label{eq:j9}
\end{equation}
\end{lem}

\begin{proof}
See, e.g., \cite{MR1869063,MR1889566,MR2030387}. 

Note that if $x$ and $y$ $\in\mathbb{Z}^{N}$, and $k\in\mathbb{N}$,
then in the quotient group we have 
\[
\left(M^{T}\right)^{-k}x\cdot y=\left(M^{T}\right)^{-\left(k+1\right)}x\cdot My.
\]
 
\end{proof}

\section{\label{sec:IFS}Iterated Function Systems (IFS), and $Rep\left(\mathcal{O}_{N},\mathscr{H}\right)$}

Recall, when $N$ is a fixed integer, at least 2, the corresponding
Cuntz algebra $\mathcal{O}_{N}$ has a rich family of representations
(see, e.g., \cite{Gli60,Gli61,MR0467330,MR1913212,MR2030387}). They
are studied in the previous two sections, with the use of the associated
projection-valued measures. As noted in Section \ref{subsec:ss},
some of the $\mathcal{O}_{N}$ representations correspond to iterated
function systems (IFSs), where the iteration of branching laws is
given by a system of $N$ prescribed endomorphisms in a measure space.
One reason the use of IFSs is powerful is that the framework allows
one to make precise iteration of \emph{self-similarity} in Cantor-dynamics,
and, more generally, in \emph{non-reversible dynamics}, as well as
the corresponding \textquotedblleft chaos-limits.\textquotedblright{}
(See \cite{MR625600,MR1656855,MR3217056,MR3687240}.) The setting
of IFS-systems includes a rich class of fractals, e.g., those corresponding
to affine IFSs, and others to complex dynamics.

Two themes are addressed in this section: (i) We present the correspondence
between \emph{representations of the Cuntz algebra} $\mathcal{O}_{N}$,
on one hand, and IFSs with $N$ generating endomorphisms, on the other.
(ii) Our focus will be a use of the $\mathcal{O}_{N}$ representations
\emph{in a realization of generalized Martin boundaries} for the IFSs
under consideration. For this purpose, it will be convenient to first
fix an alphabet $A$, of size $N$. We then consider kernels indexed
by both finite words in $A$, as well as by infinite words; see Section
\ref{sec:PVM} for details.

In Theorem \ref{thm:qint} below, we show that such a boundary theory
may be derived from the random variables $Y$ which we introduced
in Section \ref{sec:PVR}. In broad outline, our boundary representations
will be obtained as limits of kernels indexed initially by finite
words in the alphabet A; -- the limit referring to finite vs infinite
words in the symbolic representations. This theme will be expanded
further in Section \ref{sec:uh} below.

The present section concludes with a number of explicit examples.\\

Let $\left(M,d\right)$ be a compact metric space, $N\in\mathbb{N}$
fixed, $N\geq2$, 
\begin{equation}
p_{1},\cdots,p_{N},\;,p_{i}>0,\;\sum_{i=1}^{N}p_{i}=1,\text{ fixed.}\label{eq:d1}
\end{equation}
Let $\tau_{i}:M\rightarrow M$, $1\leq i\leq N$, be a system of strict
contractions in $\left(M,d\right)$. Let $\Omega_{N}=\left\{ 1,2,\cdots,N\right\} ^{\mathbb{N}}$,
and let 
\begin{equation}
\mathbb{P}=\vartimes{}_{1}^{\infty}p=\underset{\aleph_{0}\text{ product measure}}{\underbrace{p\times p\times p\cdots\cdots}}\label{eq:d2}
\end{equation}
(see \cite{MR0014404,MR562914}.)

In this section, we construct random variables $Y$ with values in
$M$ (some measure space $\left(M,\mathscr{B}_{M}\right)$), so $Y:\Omega\rightarrow M$,
such that the corresponding distribution $\mu:=\mathbb{P}\circ Y^{-1}$
satisfies 
\[
\mu=\sum_{i=1}^{N}p_{i}\mu\circ\tau_{i}^{-1}.
\]
Here $\mathbb{P}$ is the infinite -product measure (\ref{eq:d2}).
\begin{example}[A Julia construction]
Although the early analysis (e.g., \cite{MR625600}) of many of the
iterated function systems (IFSs) focused on iteration of systems of
affine maps in some ambient $\mathbb{R}^{d}$, there is also a rich
literature dealing with complex dynamics, and iteration of conformal
maps, see e.g., \cite{MR2193309}. Also in these cases, there are
IFS measures, see Theorem \ref{thm:pp} \ref{enu:pp2}. In the simplest
cases these Julia iteration limits arise from an iteration of branches
of the inverse of complex polynomials. The corresponding IFS limits
are typically Julia sets; named after Gaston Julia. Examples are included
in Fig \ref{fig:ju}.
\end{example}

\begin{figure}[H]
\subfloat[$c=0.125+0.625i$]{\includegraphics[width=0.35\textwidth]{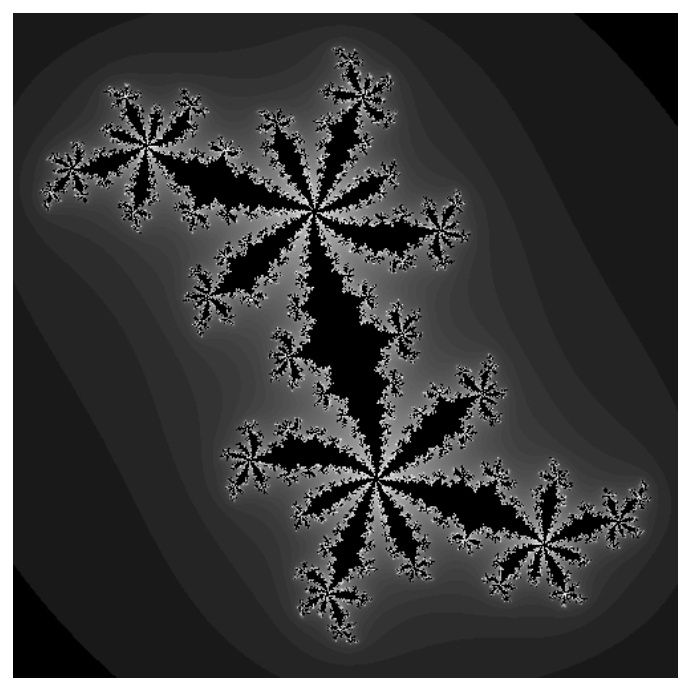}

}\qquad\subfloat[$c=0.375-0.125i$]{\includegraphics[width=0.35\textwidth]{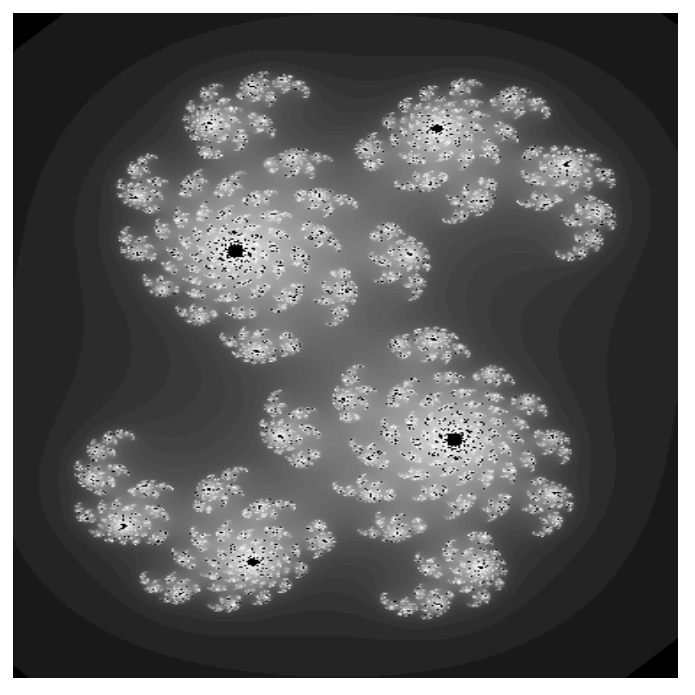}

}\caption{\label{fig:ju}$\mathbb{C}\ni z\rightarrow z^{2}+c$ ($c\in\mathbb{C}\backslash\left\{ 0\right\} $
fixed), $\tau_{\pm}:z\rightarrow\pm\sqrt{z-c}$.}

\end{figure}

\begin{thm}
\label{thm:pp}For points $\omega=\left(i_{1},i_{2},i_{3},\cdots\right)\in\Omega_{N}$
and $k\in\mathbb{N}$, set 
\begin{align}
\omega\big|_{k} & =\left(i_{1},i_{2},\cdots,i_{k}\right),\;\text{and}\label{eq:d3}\\
\tau_{\omega|_{k}} & =\tau_{i_{1}}\circ\tau_{i_{2}}\circ\cdots\circ\tau_{i_{k}}.\label{eq:d4}
\end{align}
Then $\bigcap_{k=1}^{\infty}\tau_{\omega|_{k}}$$\left(M\right)$
is a singleton, say $\left\{ x\left(\omega\right)\right\} $. Set
$Y\left(\omega\right)=x\left(\omega\right)$, i.e., 
\begin{equation}
\left\{ Y\left(\omega\right)\right\} =\bigcap_{k=1}^{\infty}\tau_{\omega|_{k}}\left(M\right);\label{eq:d5}
\end{equation}
then: 
\begin{enumerate}
\item $Y:\Omega_{N}\rightarrow M$ is an $\left(M,d\right)$-valued random
variable.
\item \label{enu:pp2}The distribution of $Y$, i.e., the measure 
\begin{equation}
\mu=\mathbb{P}\circ Y^{-1}\label{eq:d6}
\end{equation}
is the unique Borel probability measure on $\left(M,d\right)$ satisfying:
\begin{equation}
\mu=\sum_{i=1}^{N}p_{i}\mu\circ\tau_{i}^{-1};\label{eq:d7}
\end{equation}
equivalently, 
\begin{equation}
\int_{M}fd\mu=\sum_{i=1}^{N}p_{i}\int_{M}\left(f\circ\tau_{i}\right)d\mu,\label{eq:d8}
\end{equation}
holds for all Borel functions $f$ on $M$.
\item The support $M_{\mu}=supp\left(\mu\right)$ is the minimal closed
set (IFS), $\neq\emptyset$, satisfying 
\begin{equation}
M_{\mu}=\bigcup_{i=1}^{\infty}\tau_{i}\left(M_{\mu}\right).\label{eq:d9}
\end{equation}
\end{enumerate}
\end{thm}

\begin{figure}[h]
\[
E\left(\omega|_{1}\right)\supset E\left(\omega|_{2}\right)\supset\cdots\supset E\left(\omega|_{k}\right)\supset E\left(\omega|_{k+1}\right)\supset\cdots
\]

\includegraphics[width=0.5\textwidth]{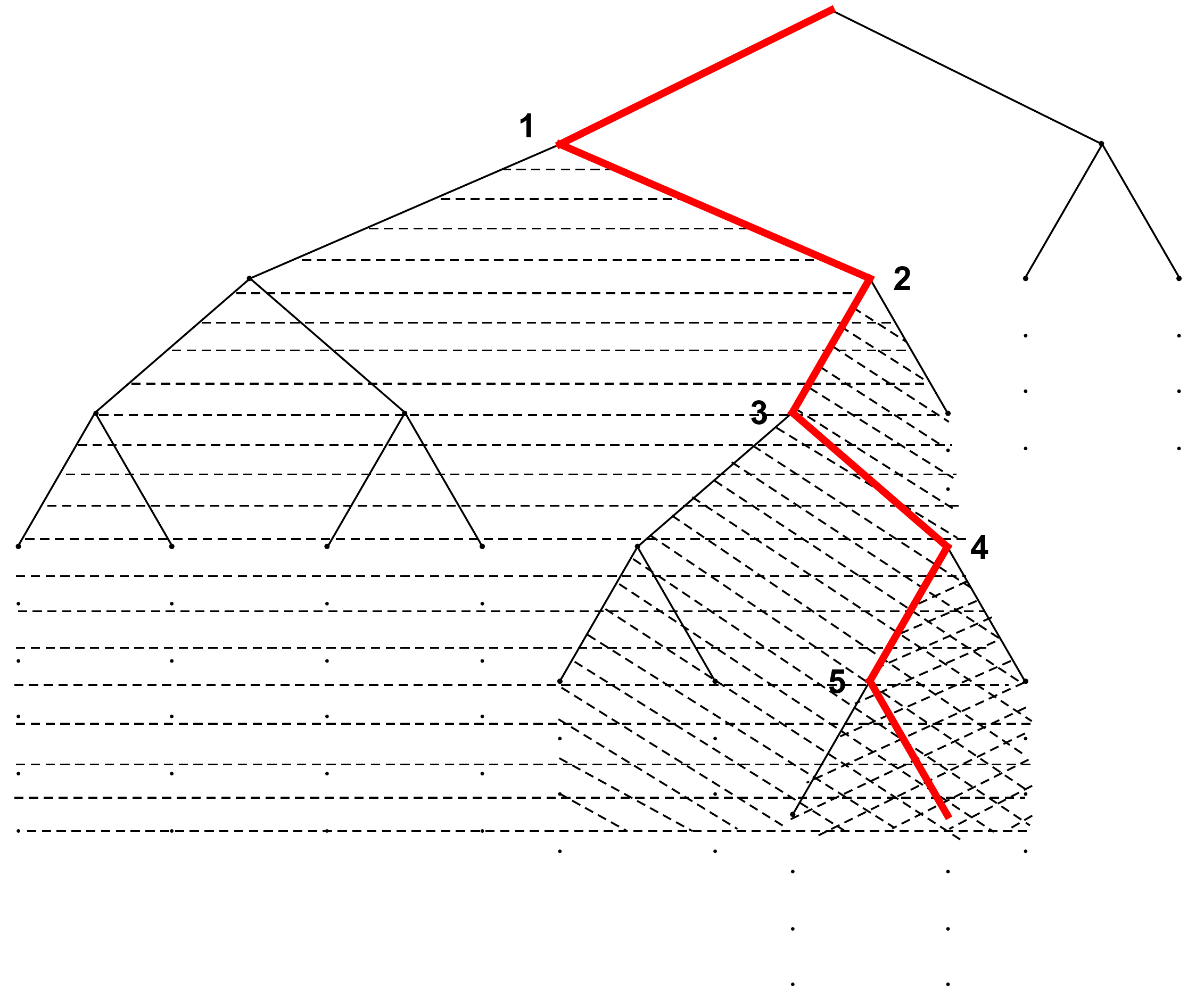}

\caption{\label{fig:cyl2}$\left\{ \omega\right\} =\bigcap_{k=1}^{\infty}E\left(\omega|_{k}\right)$.
Monotone families of \emph{tail sets}. Let $\Omega_{N}$ be the set
of all infinite words, i.e., the infinite Cartesian product. Start
with a fixed infinite word $\omega$, so $\omega$ in $\Omega_{N}$
(highlighted in \ref{fig:cyl2}.) For every positive $k$, we truncate
$\omega$, thus forming a finite word $\omega|_{k}$. Then the set
$E\left(\omega|_{k}\right)$ is the set of all infinite words that
begin with $\omega|_{k}$, but unrestricted after $k$. The intersection
in $k$ of all these sets $E\left(\omega|_{k}\right)$ is then the
singleton $\left\{ \omega\right\} $.}
\end{figure}

\begin{proof}
We shall make use of standard facts from the theory of iterated function
systems (IFS), and their measures; see e.g., \cite{MR625600,MR2431670,MR1656855}.

\emph{Proof of (\ref{eq:d5})}. We use that when $\omega\in\Omega_{N}$
is fixed then the sets $\tau_{\omega|_{k}}\left(M\right)$ is a monotone
family of compact subsets 
\begin{equation}
\tau_{\omega|_{k+1}}\left(M\right)\subset\tau_{\omega|_{k}}\left(M\right),\label{eq:d10}
\end{equation}
and since $\tau_{i}$ is strictly contractive for all $i$, we get
\begin{equation}
\lim_{k\rightarrow\infty}diameter\left(\tau_{\omega|_{k}}\left(M\right)\right)=0,\label{eq:d11}
\end{equation}
and so (\ref{eq:d5}) follows; i.e., the intersection $\bigcap_{i=1}^{\infty}$
is a singleton depending only on $\omega$. 

\emph{Monotonicity}: This conclusion again follows from the assumptions
placed on $\left\{ \tau_{i}\right\} _{i=1}^{N}$, but we shall specify
the respective $\sigma$-algebras, the one on $\Omega_{N}$ and the
one on $M$. 

The $\sigma$-algebra of subsets of $\Omega_{N}$ will be generated
by cylinder sets: If $f=\left(i_{1},i_{2},\cdots,i_{k}\right)$ is
a \emph{finite word}, the corresponding cylinder set $E\left(f\right)\subset\Omega_{N}$
is 
\begin{equation}
E\left(f\right)=\left\{ \omega\in\Omega_{N}\mid\omega_{j}=i_{j},\;1\leq j\leq k\right\} .\label{eq:d12}
\end{equation}
On $M$, we pick the Borel $\sigma$-algebra determined from the fixed
metric $d$ on $M$. The measure $\mathbb{P}=\mathbb{P}_{p}$ is specified
by its values on cylinder sets; i.e, set 
\begin{equation}
\mathbb{P}\left(E\left(f\right)\right)=p_{i_{1}}p_{i_{2}}\cdots p_{i_{k}}=:p_{f}\label{eq:d13}
\end{equation}
where the numbers $p_{1},\cdots,p_{N}$ are as in (\ref{eq:d1}). 

\emph{Proof of (\ref{eq:d7})}. The argument is based on the following:
On $\Omega_{N}$, introduce the \emph{shifts} $\hat{\tau}_{b}\left(i_{1},i_{2},i_{3},\cdots\right)=\left(b,i_{1},i_{2},i_{3},\cdots\right)$,
$b\in\left\{ 1,2,\cdots,N\right\} $, and let $Y$ be as in (\ref{eq:d5})-(\ref{eq:d6}).
Then 
\begin{equation}
\tau_{b}Y=Y\hat{\tau}_{b},\label{eq:d14}
\end{equation}
or equivalently, 
\[
\xymatrix{\Omega_{N}\ar[rr]^{Y}\ar[d]_{\hat{\tau}_{b}} &  & M\ar[d]^{\tau_{b}}\\
\Omega_{N}\ar[rr]_{Y} &  & M
}
\]
\begin{equation}
\tau_{b}\left(Y\left(\omega\right)\right)=Y\left(\hat{\tau}_{b}\left(\omega\right)\right),\;\forall\omega\in\Omega_{N}.\label{eq:d15}
\end{equation}
Now (\ref{eq:d15}) is immediate from (\ref{eq:d5}).

We now show (\ref{eq:d8}), equivalently (\ref{eq:d7}). Let $f$
be a Borel function on $M$, then 
\begin{alignat*}{2}
\int_{M}f\,d\mu & =\int_{\Omega_{N}}\left(f\circ Y\right)d\mathbb{P} & \quad & \text{\ensuremath{\left(\text{by }\left(\ref{eq:d6}\right)\right)}}\\
 & =\sum_{i=1}^{N}p_{i}\int_{\Omega_{N}}f\circ Y\circ\hat{\tau}_{i}\:d\mathbb{P} &  & \left(\begin{matrix}\text{since \ensuremath{\mathbb{P}} is the product}\\
\text{ measure \ensuremath{\vartimes_{1}^{\infty}p}, see \ensuremath{\left(\ref{eq:d13}\right)}}
\end{matrix}\right)\\
 & =\sum_{i=1}^{N}p_{i}\int_{\Omega_{N}}f\circ\tau_{i}\circ Y\:d\mathbb{P} &  & \text{\ensuremath{\left(\text{by }\left(\ref{eq:d14}\right)\right)}}\\
 & =\sum_{i=1}^{N}p_{i}\int_{M}f\circ\tau_{i}\:d\mu &  & \text{\ensuremath{\left(\text{by }\left(\ref{eq:d6}\right)\right)}}
\end{alignat*}
which is the desired conclusion. 
\end{proof}
\begin{figure}
\includegraphics[width=0.8\textwidth]{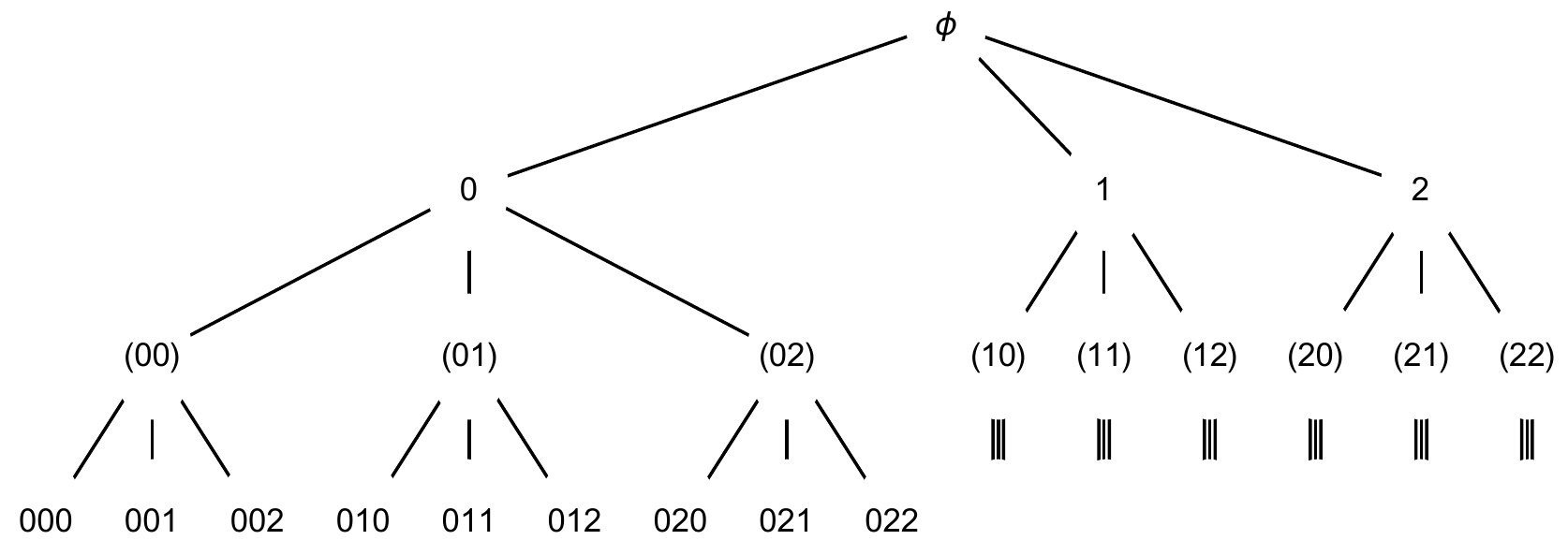}

\vspace{1em}

\includegraphics[width=0.5\textwidth]{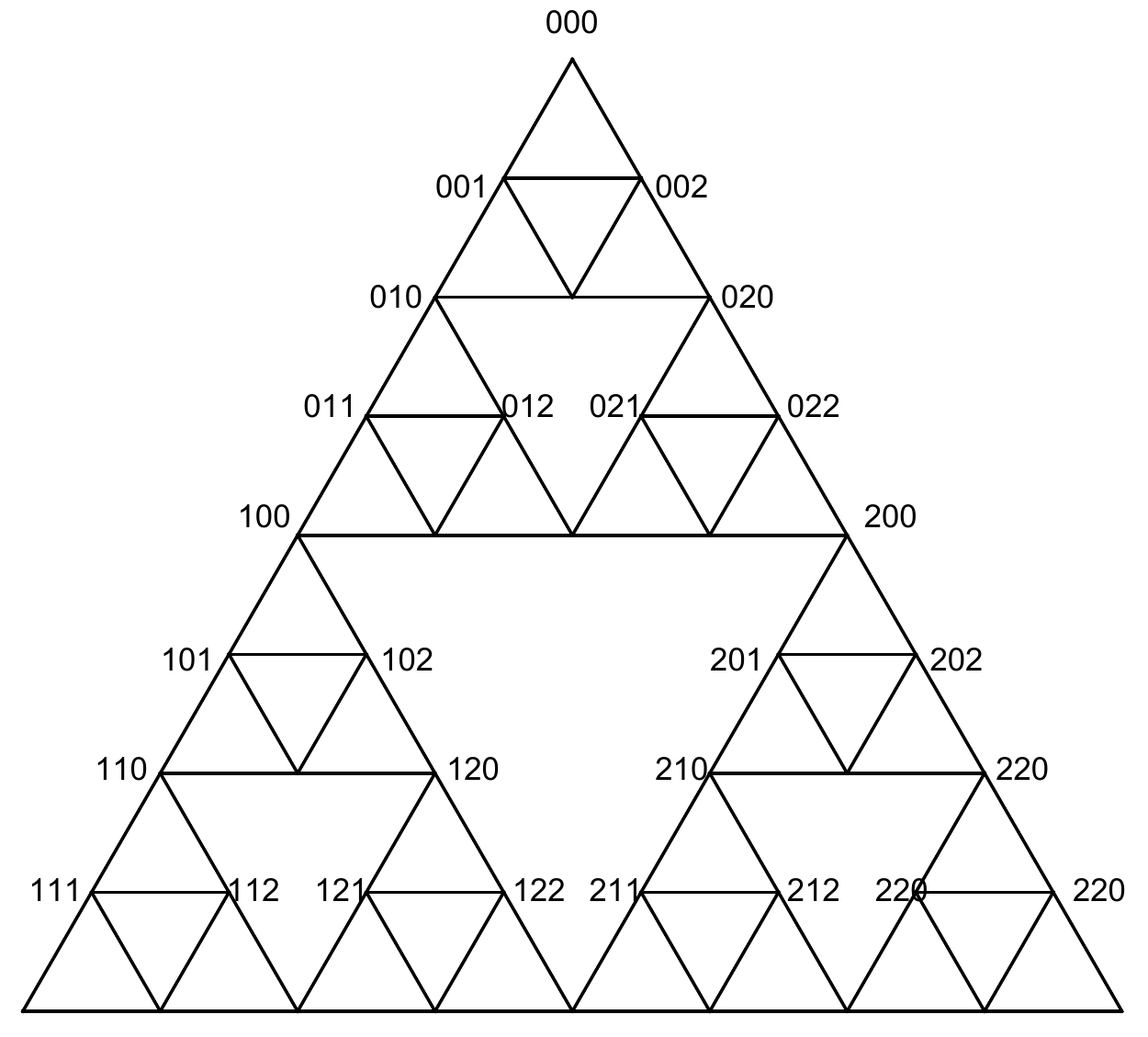}

\caption{\label{fig:sg}\emph{Encoding of words} into IFS. Infinite words $\omega\in\Omega$
$\protect\longrightarrow$ singletons in the Sierpinski gasket. }
\end{figure}

Using $\Omega_{N}=\left\{ 1,2,\cdots,N\right\} ^{\mathbb{N}}$ for
encoding iterated function systems (IFS).
\begin{example}[\emph{Sierpinski gasket}]
$M=\left[0,1\right]\times\left[0,1\right]$ with the usual metric,
\[
\tau_{0}\left(x,y\right)=\left(\frac{x}{2},\frac{y}{2}\right),\quad\tau_{1}\left(x,y\right)=\left(\frac{x+1}{2},\frac{y}{2}\right),\quad\tau_{2}\left(x,y\right)=\left(\frac{x}{2},\frac{y+1}{2}\right),
\]
and so the Sierpinski gasket $M_{Si}$ satisfies 
\[
M_{Si}=\tau_{0}\left(M_{Si}\right)\bigcup\tau_{1}\left(M_{Si}\right)\bigcup\tau_{2}\left(M_{Si}\right).
\]
See Figure \ref{fig:sg}. 
\end{example}

\subsection{The Projection Valued Path Space Measure Corresponding to IFS Representations}

In Theorem \ref{thm:pp}, we introduced the class of contractive iterated
function systems (IFSs), $\left\{ \tau_{i}\right\} _{i=1}^{N}$. We
shall point out that, for each of these IFSs, there is a \emph{natural
representation} of $\mathcal{O}_{N}$; as follows: 

For simplicity, we shall assume in addition to the conditions listed
in Theorem \ref{thm:pp}, that we also have non-overlap as follows:
If $i\neq j$, then we assume 
\begin{equation}
\mu\left(\tau_{i}\left(M\right)\cap\tau_{j}\left(M\right)\right)=0.\label{eq:g1}
\end{equation}

We also fix weights $\left\{ p_{i}\right\} _{i=1}^{N}$, and we let
$\mu$ be the corresponding IFS-measure, see (\ref{eq:d7}) in the
theorem. Also we recall the associated endomorphism $\sigma$ in $\left(M,d\right)$
satisfying: 
\begin{equation}
\sigma\circ\tau_{i}=id_{M},\;\forall i=1,\cdots,N.\label{eq:g2}
\end{equation}
Once (\ref{eq:g1}) is assumed, it is easy to construct $\sigma$
such that (\ref{eq:g2}) holds, i.e., the system $\left\{ \tau_{i}\right\} _{i=1}^{N}$
constitutes branches of the inverse for $\sigma$, and 
\begin{equation}
\sigma^{-1}\left(E\right)=\bigcup_{i=1}^{N}\tau_{i}\left(E\right)\label{eq:g3}
\end{equation}
for all Borel subsets $E\subset M$. 
\begin{prop}
\label{prop:ifs}Let $\left(\left\{ \tau_{i}\right\} _{i=1}^{N},\left\{ p_{i}\right\} _{i=1}^{N},\mu\right)$
be as stated; and set $\mathscr{H}=L^{2}\left(M,\mu\right)$. Then
the following operators $\left\{ S_{i}\right\} _{i=1}^{N}$, $\left\{ S_{i}^{*}\right\} _{i=1}^{N}$
constitute a representation 
\begin{equation}
\pi\in Rep\left(\mathcal{O}_{N},L^{2}\left(\mu\right)\right).\label{eq:g4}
\end{equation}
We set, for $f\in L^{2}\left(\mu\right)$: 
\begin{equation}
S_{i}f=\frac{1}{\sqrt{p_{i}}}\chi_{\tau_{i}\left(M\right)}f\circ\sigma\label{eq:g5}
\end{equation}
and 
\begin{equation}
S_{i}^{*}f=\sqrt{p_{i}}f\circ\tau_{i}.\label{eq:g6}
\end{equation}
\end{prop}

\begin{proof}
Let $S_{i}$, $S_{i}^{*}$ denote the system of $2N$ operators in
$L^{2}\left(\mu\right)$, given in (\ref{eq:g5})-(\ref{eq:g6}).
It is immediate that $S_{i}^{*}S_{j}=\delta_{ij}I_{L^{2}\left(\mu\right)}$.
For this we use that 
\begin{equation}
\chi_{\tau_{i}\left(M\right)}\circ\tau_{j}=\delta_{ij}\mathbbm{1},\label{eq:g7}
\end{equation}
the constant function $\mathbbm{1}$ in $L^{2}\left(\mu\right)$.

A direct computation using (\ref{eq:d7}) in Theorem \ref{thm:pp}
yields
\begin{equation}
\int_{M}f\,\left(S_{i}g\right)d\mu=\int_{M}\left(S_{i}^{*}f\right)g\,d\mu,\label{eq:g8}
\end{equation}
valid for all $f,g\in L^{2}\left(\mu\right)$.

Moreover, we have 
\begin{equation}
S_{i}S_{i}^{*}=\text{multiplication by \ensuremath{\chi_{\tau_{i}\left(M\right)}} in \ensuremath{L^{2}\left(M\right)}. }\label{eq:g9}
\end{equation}
Since $\bigcup_{i=1}^{N}\tau_{i}\left(M\right)=M$, as a disjoint
union, we also get 
\begin{equation}
\sum_{i=1}^{N}S_{i}S_{i}^{*}=I_{L^{2}\left(\mu\right)},\label{eq:g10}
\end{equation}
and so $\pi\in Rep\left(\mathcal{O}_{N},L^{2}\left(\mu\right)\right)$
as asserted in the Proposition. 
\end{proof}
\begin{cor}
\label{cor:ifs1}Let $\left\{ \tau_{i}\right\} _{i=1}^{N}$ and $\pi\in Rep\left(\mathcal{O}_{N},L^{2}\left(\mu\right)\right)$
be as in Proposition \ref{prop:ifs}; then for the projection-valued
path space measure $\mathbb{Q}_{\pi}$ in Theorem \ref{thm:pv}, we
have the following formula:

Let $f=\left(i_{1},\cdots,i_{N}\right)$ be a finite word, and let
$E_{f}$ denote the corresponding basic cylinder subset, $E_{f}\subseteq\Omega_{N}$,
then 
\begin{align}
\mathbb{Q}_{\pi}\left(E_{f}\right) & =\text{multiplication by the indicator }\nonumber \\
 & \quad\quad\text{function \ensuremath{\chi_{\tau_{i_{1}}\tau_{i_{2}}\cdots\tau_{i_{k}}\left(M\right)}}}.\label{eq:g11}
\end{align}
\end{cor}

\begin{proof}
Immediate from Theorem \ref{thm:pv}.
\end{proof}
\needspace{5\baselineskip}
\begin{flushleft}
\textbf{A self-dual representation.}
\par\end{flushleft}
\begin{cor}
Let the setting be as in Proposition \ref{prop:ifs} and Corollary
\ref{cor:ifs1}. In particular, we fix $\big(N,\left\{ p_{i}\right\} _{i=1}^{N},\mu\big)$,
and the $\mathcal{O}_{N}$-representation $\pi^{\left(\mu\right)}:=\big\{ S_{i}^{\left(\mu\right)}\big\}_{i=1}^{N}$
as specified in (\ref{eq:g5})-(\ref{eq:g6}). This representation
is \uline{self-dual} in the following sense.

Let $\mathbbm{1}$ denote the constant function $\mathbbm{1}$ in
$L^{2}\left(M,\mu\right)$, and let $\mathbb{Q}_{\pi}$ be the corresponding
projection valued measure (see Corollary \ref{cor:QO}). Set 
\begin{equation}
\left\langle \mathbbm{1},\mathbb{Q}_{\pi}\left(\cdot\right)\mathbbm{1}\right\rangle _{L^{2}\left(\mu\right)}=\nu_{\pi}\left(\cdot\right),\label{eq:g12}
\end{equation}
as a measure on $\Omega_{N}$; then 
\begin{equation}
\nu_{\pi}\circ Y^{-1}=\mu.\label{eq:g13}
\end{equation}
\end{cor}

\begin{proof}
Since $\pi^{\left(\mu\right)}=\big\{ S_{i}^{\left(\mu\right)}\big\}_{i=1}^{N}$
in (\ref{eq:g5})-(\ref{eq:g6}) is in $Rep\left(\mathcal{O}_{N},L^{2}\left(\mu\right)\right)$,
we get from (\ref{eq:g10}): 
\begin{equation}
\sum_{i=1}^{N}\left\Vert S_{i}^{\left(\mu\right)*}f\right\Vert _{L^{2}\left(\mu\right)}^{2}=\left\Vert f\right\Vert _{L^{2}\left(\mu\right)}^{2},\;\forall f\in L^{2}\left(\mu\right).\label{eq:g14}
\end{equation}
Introducing (\ref{eq:g6}), we then conclude: 
\[
\sum_{i=1}^{N}p_{i}\int_{M}\left|f\right|^{2}\circ\tau_{i}\,d\mu=\int_{M}\left|f\right|^{2}d\mu,
\]
and so, by (\ref{eq:g9}), 
\[
\int_{M}\left|f\right|^{2}d\mu=\int_{\Omega_{N}}\left(\left|f\right|^{2}\circ Y\right)d\nu_{\pi}=\int_{M}\left|f\right|^{2}d\left(\nu_{\pi}\circ Y^{-1}\right),
\]
valid for all $f\in L^{2}\left(\mu\right)$. The desired conclusion
(\ref{eq:g13}) follows.
\end{proof}

\subsection{\label{subsec:br}Boundaries of Representations}

Let $M$ be a compact Hausdorff space, with Borel $\sigma$-algebra
$\mathscr{B}$, and let $A$ be a finite alphabet, $\left|A\right|=N$.
Let $\left\{ \tau_{i}\right\} _{i\in A}$ be a system of endomorphisms.
For every $\omega\in\Omega_{N}\left(=A^{\mathbb{N}}\right)$, and
$k\in\mathbb{N}$, set $\omega|_{k}=\left(\omega_{1},\cdots,\omega_{k}\right)$
(= the truncated finite word), and set 
\begin{equation}
\tau_{\omega|_{k}}=\tau_{\omega_{1}}\circ\cdots\circ\tau_{\omega_{k}}.\label{eq:k1}
\end{equation}

\begin{defn}
We say that $\left\{ \tau_{i}\right\} _{i\in A}$ is \emph{tight}
iff 
\begin{equation}
\bigcap_{k=1}^{\infty}\tau_{\omega|_{k}}\left(M\right)=\left\{ Y\left(\omega\right)\right\} \label{eq:k2}
\end{equation}
is a singleton for $\forall\omega\in\Omega_{N}$; and we define $Y:\Omega_{N}\rightarrow M$
by eq. (\ref{eq:k2}).
\end{defn}

\begin{thm}
\label{thm:qint}Let $\left(M,\left\{ \tau_{i}\right\} _{i\in A}\right)$
be as above, assume \uline{tight}. Let $\pi\in Rep\left(\mathcal{O}_{N},\mathscr{H}\right)$
for some Hilbert space, and let $\mathbb{Q}_{\pi}$ be the corresponding
projection-valued measure. Assume $\mathbb{Q}_{\pi}$ has one-dimensional
range; see Corollary \ref{cor:QO}; set 
\begin{equation}
\mu:=\mathbb{Q}_{\pi}\circ Y^{-1}.\label{eq:k3}
\end{equation}

Then for all $\omega\in\Omega_{N}=A^{\mathbb{N}}$, we have 
\begin{equation}
\mu\circ\tau_{\omega|_{k}}^{-1}\xrightarrow[\;k\rightarrow\infty\;]{}\delta_{Y\left(\omega\right)},\label{eq:k4}
\end{equation}
i.e., for all $f\in C\left(M\right)$, we have 
\begin{equation}
\lim_{k\rightarrow\infty}\int_{M}f\circ\tau_{\omega|_{k}}d\mu=f\left(Y\left(\omega\right)\right).\label{eq:k5}
\end{equation}
\end{thm}

\begin{proof}
Let $\varepsilon>0$. Since $f$ is uniformly continuous, there is
a neighborhood $O_{\omega}$ of $Y\left(\omega\right)$ such that
\begin{equation}
\left|f\left(y\right)-f\left(y'\right)\right|<\varepsilon\quad\text{for}\;\forall y,y'\in O_{\omega}.\label{eq:k6}
\end{equation}
Since by assumption $\mu\left(M\right)=1$, we conclude from (\ref{eq:k6})
and (\ref{eq:k2}), that for $\forall k,l\geq k_{0}$, we have 
\begin{equation}
\left|f\circ\tau_{\omega|_{k}}-f\circ\tau_{\omega|_{l}}\right|\leq\varepsilon;\label{eq:k7}
\end{equation}
as a uniform estimate on $M$. Since 
\begin{equation}
\int_{M}f\circ\tau_{\omega|_{k}}d\mu=\int_{M}f\,d\left(\mu\circ\tau_{\omega|_{k}}^{-1}\right),\label{eq:k8}
\end{equation}
a second application of (\ref{eq:k2}) now yields:
\[
\lim_{k\rightarrow\infty}\int_{M}f\circ\tau_{\omega|_{k}}d\mu=f\left(Y\left(\omega\right)\right)
\]
which is the desired conclusion. 
\end{proof}

\subsection{\label{subsec:3ex}Three Examples}

Below we give three examples of IFS-measures, as in Theorem \ref{thm:pp}:
(i) \emph{the} \emph{Lebesgue measure} restricted to the unit interval
$\left[0,1\right]$, (ii) \emph{the} \emph{middle-third Cantor measure}
$\mu_{3}$, and (iii) \emph{the $\nicefrac{1}{4}$-Cantor measure}
$\mu_{4}$ with two gaps. Their respective properties follow from
Theorem \ref{thm:pp}, and are summarized in Table \ref{tab:IFSmeas}.
Also see Figures \ref{fig:cantor}, \ref{fig:cum}, and \ref{fig:ex3}.

The difference in the graphs of the \emph{cumulative distributions}
in Ex 2 and Ex 3, is explained by the following: In Ex 3, we have
\emph{two omitted intervals in each iteration step}, as opposed to
just one in Ex 2, the Middle-third Cantor construction. See Fig \ref{fig:u4}.

\begin{table}[H]
\begin{tabular*}{1\textwidth}{@{\extracolsep{\fill}}|>{\centering}p{0.28\textwidth}|c|>{\centering}p{0.08\textwidth}|>{\raggedright}m{0.28\textwidth}|}
\hline 
$\left\{ \tau_{i}\right\} _{i=1}^{2}$ & $\sigma$ & $\left(p_{i}\right)_{i=1}^{2}$  & Scaling dimension (SD) of the IFS-measure $\left(\mu,M_{\mu}\right)$\tabularnewline
\hline 
$\tau_{0}\left(x\right)=\frac{x}{2}$, $\tau_{1}\left(x\right)=\frac{x+1}{2}$ & $\sigma\left(x\right)=2x$ mod 1 & $\left(\frac{1}{2},\frac{1}{2}\right)$ & $\mu=\lambda=$ Lebesgue measure, SD = 1\tabularnewline
\hline 
$\tau_{0}\left(x\right)=\frac{x}{3}$, $\tau_{1}\left(x\right)=\frac{x+2}{3}$ & $\sigma\left(x\right)=3x$ mod 1 & $\left(\frac{1}{2},\frac{1}{2}\right)$ & $\mu=\mu_{3}=$ middle-third Cantor measure, SD = $\frac{\ln2}{\ln3}$\tabularnewline
\hline 
$\tau_{0}\left(x\right)=\frac{x}{4}$, $\tau_{1}\left(x\right)=\frac{x+2}{4}$ & $\sigma\left(x\right)=4x$ mod 1 & $\left(\frac{1}{2},\frac{1}{2}\right)$ & $\mu=\mu_{4}=$ the $\nicefrac{1}{4}$-Cantor measure, SD = $\frac{1}{2}$\tabularnewline
\hline 
\end{tabular*}

\caption{\label{tab:IFSmeas}Three inequivalent examples, each with $\Omega_{N}=A^{\mathbb{N}}$,
$\left|A\right|=2$, and infinite product measure $\vartimes_{1}^{\infty}\left(\frac{1}{2},\frac{1}{2}\right)$.
See also Fig \ref{fig:ex3}. }
\end{table}

\begin{figure}[H]
\subfloat[The middle-third Cantor set.]{\includegraphics[width=0.45\textwidth]{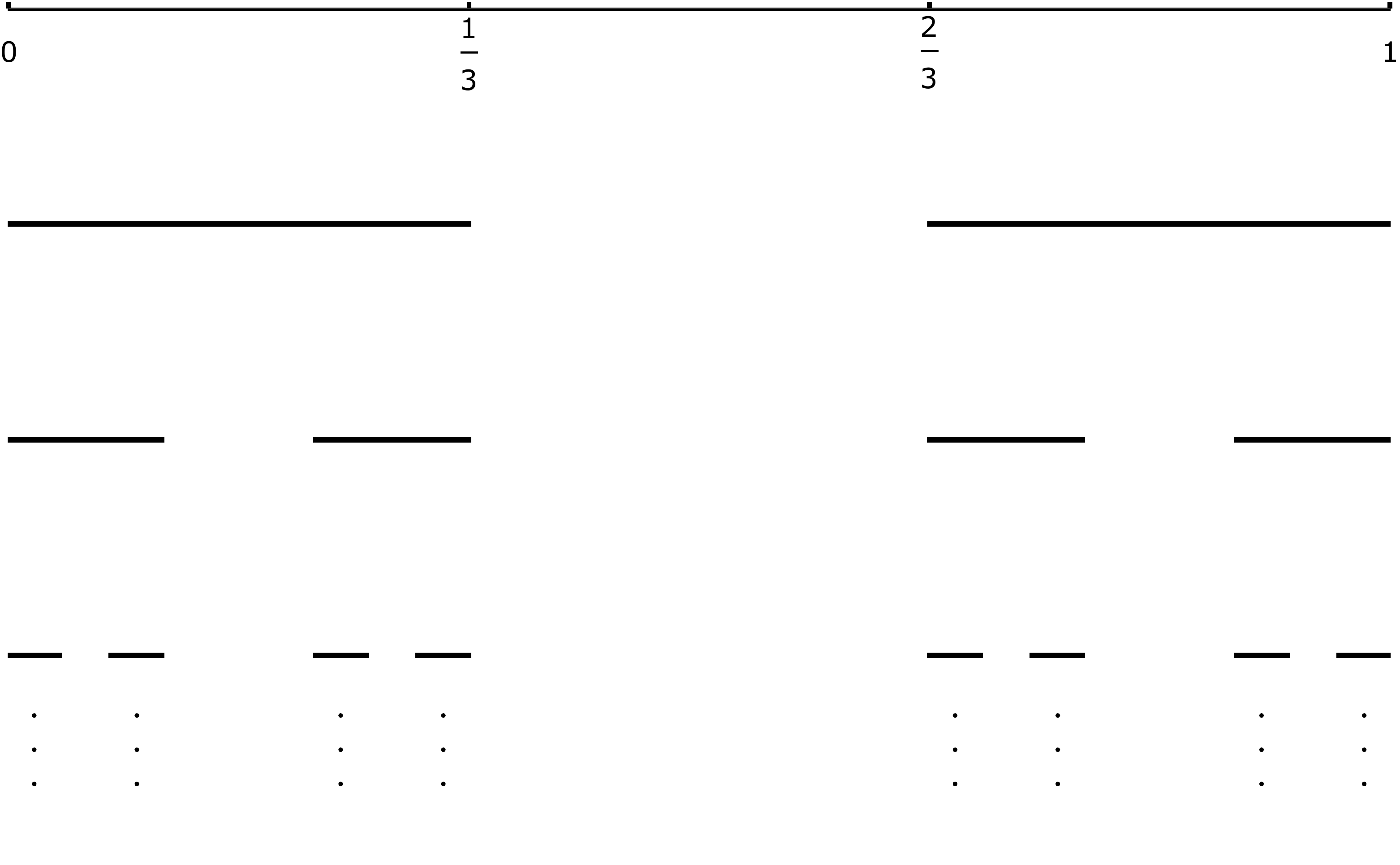}

}\qquad\subfloat[The $\nicefrac{1}{4}$-Cantor set.]{\includegraphics[width=0.45\textwidth]{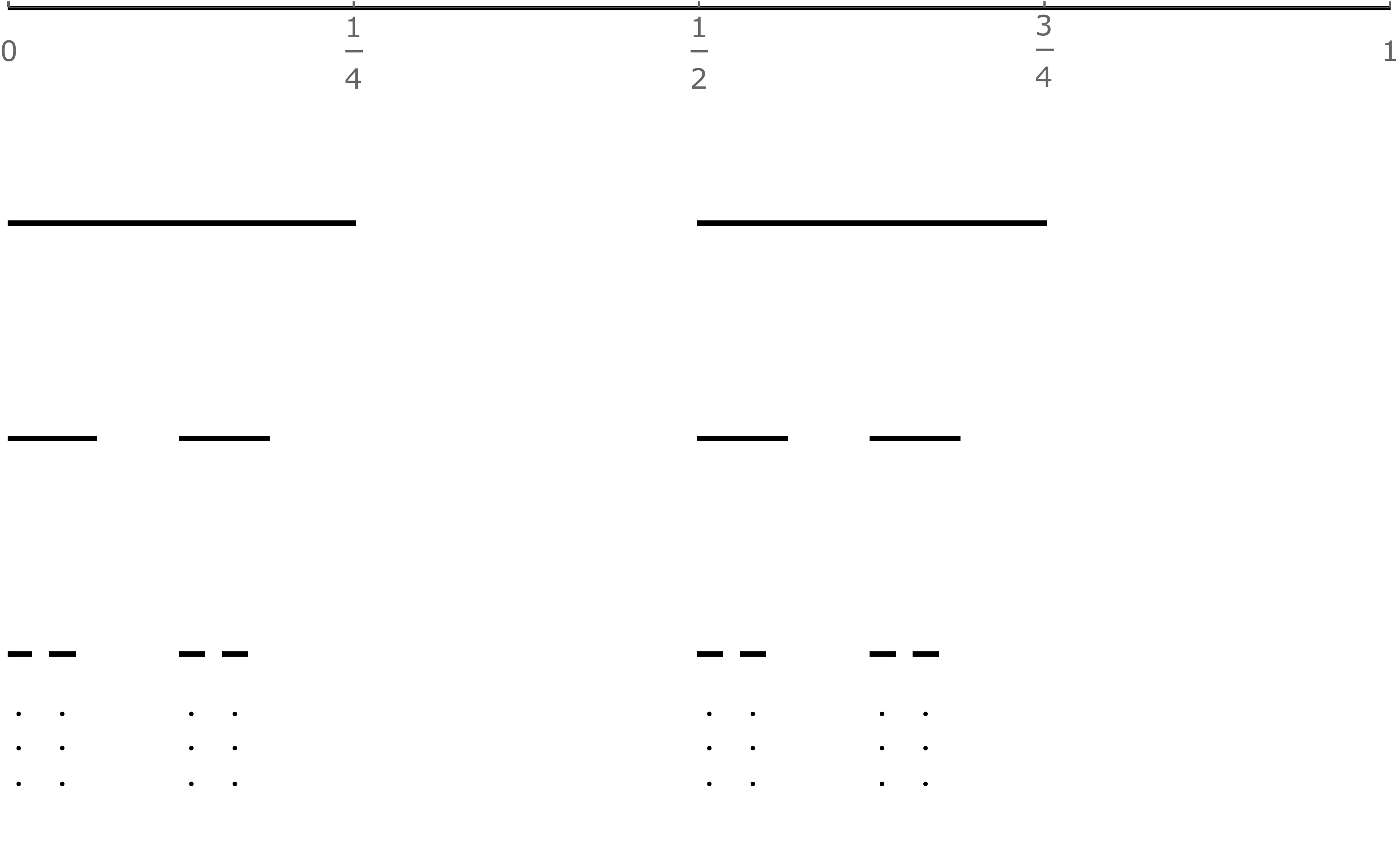}

}

\caption{\label{fig:cantor}Examples of Cantor sets.}
\end{figure}

In each of the three examples in Table \ref{tab:IFSmeas}, we give
the initial step in the IFS iteration. Each IFS-limit yields a measure,
and a support set. The second and the third examples are the fractal
limits known as the Cantor measure $\mu_{3}$, and the Cantor measure
$\mu_{4}$. The details of the iteration steps are outlined in the
subsequent figures and algorithms. Figures \ref{fig:cum} and \ref{fig:u4}
deal with the associated cumulative distribution $F\left(x\right):=\mu\left(\left[0,x\right]\right)$.
The latter will be used in Section \ref{subsec:fc} at the end of
our paper.

\begin{figure}[H]
\begin{tabular}{>{\raggedright}p{0.3\textwidth}>{\raggedright}p{0.3\textwidth}>{\raggedright}p{0.3\textwidth}}
\includegraphics[width=0.25\textwidth]{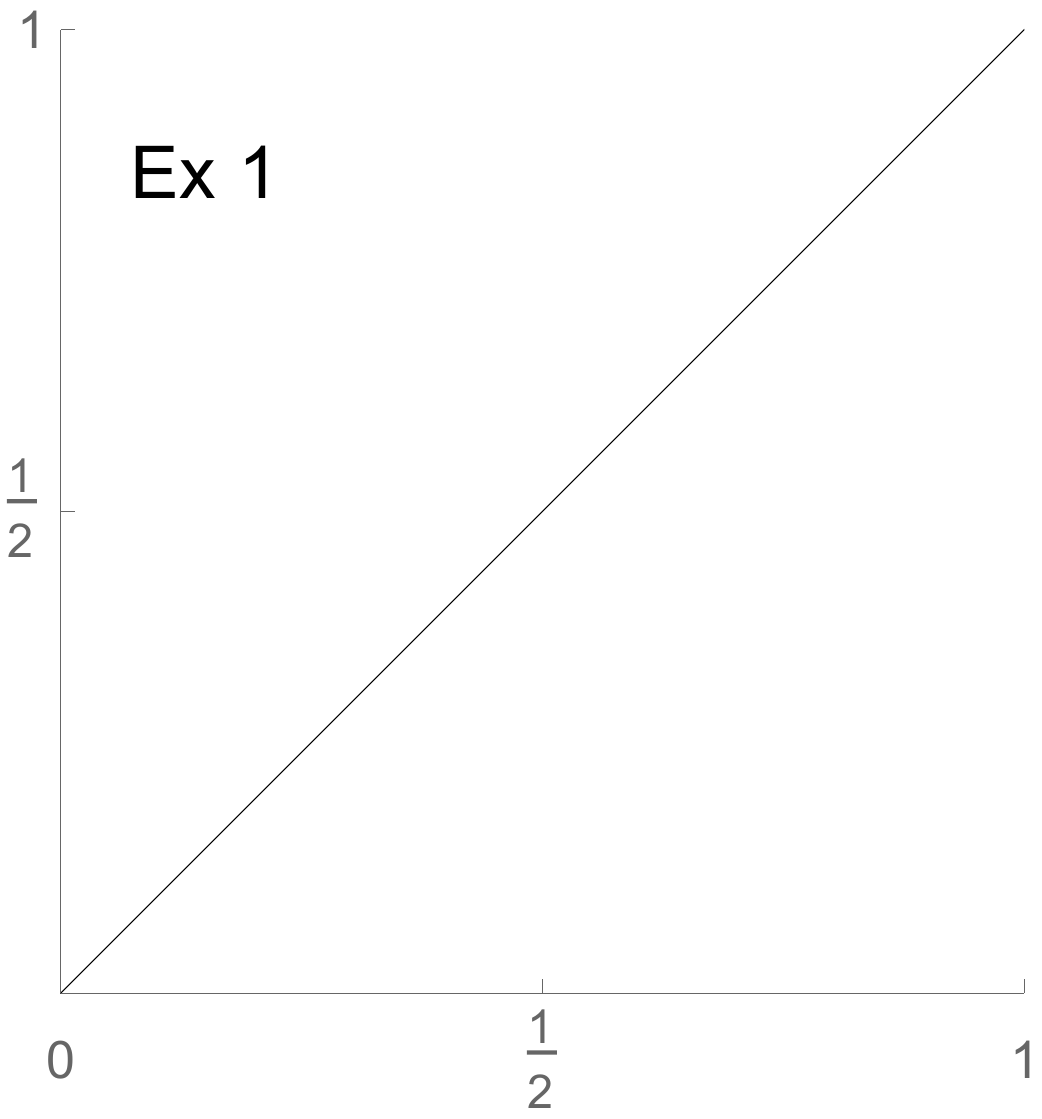} & \includegraphics[width=0.25\textwidth]{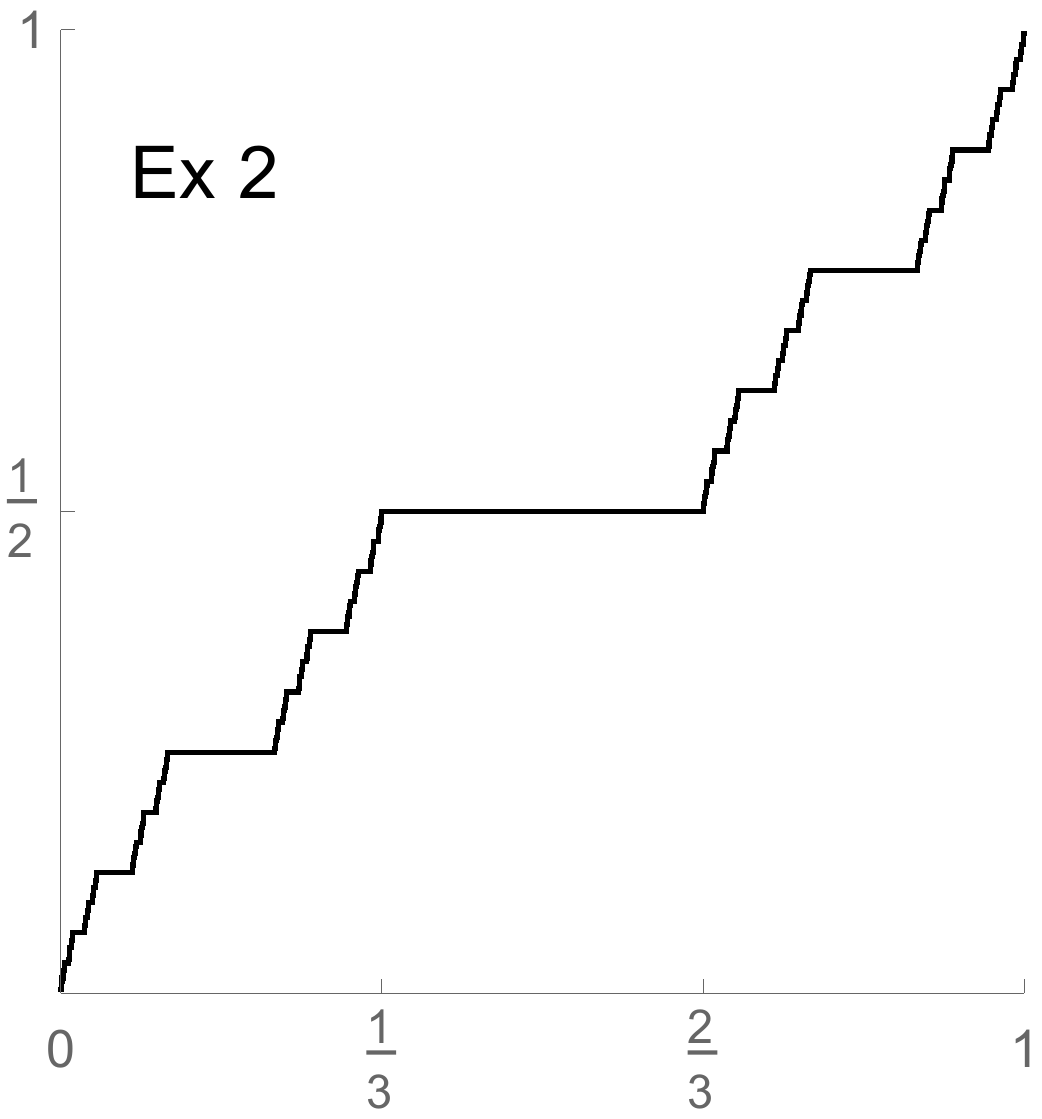} & \includegraphics[width=0.25\textwidth]{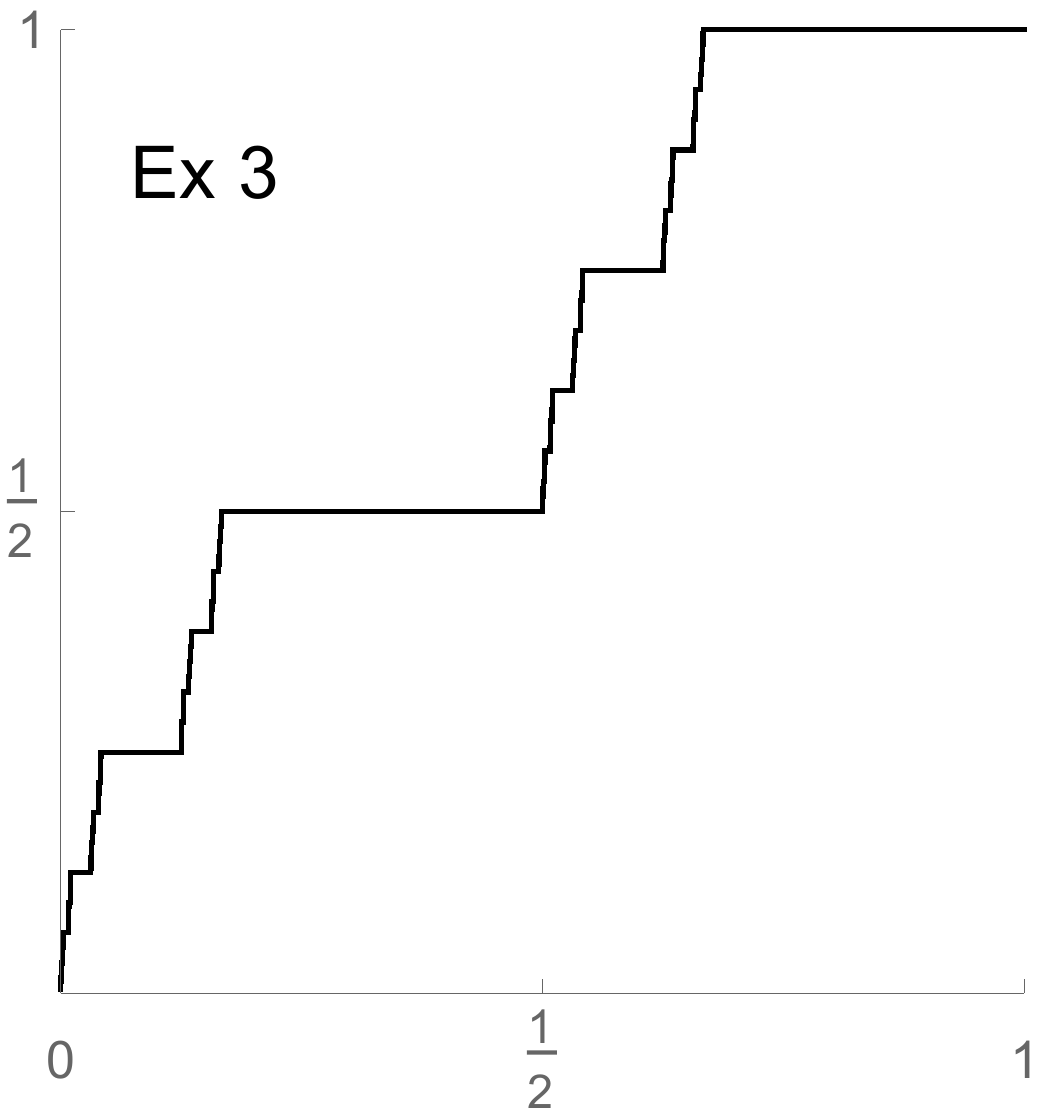}\tabularnewline
{\footnotesize{} $F_{\lambda}\left(x\right)=\lambda\left(\left[0,x\right]\right)$;
points of increase = the support of the normalized $\lambda$, so
the interval $[0,1]$. } & {\footnotesize{} $F_{\nicefrac{1}{3}}\left(x\right)=\mu_{3}\left(\left[0,x\right]\right)$;
points of increase = the support of $\mu_{3}$, so the middle third
Cantor set $C_{\nicefrac{1}{3}}$ (the Devil's staircase).} & {\footnotesize{} $F_{\nicefrac{1}{4}}\left(x\right)=\mu_{4}\left(\left[0,x\right]\right)$;
points of increase = the support of $\mu_{4}$, so the double-gap
Cantor set $C_{\nicefrac{1}{4}}$. }\tabularnewline
\end{tabular}

\caption{\label{fig:cum}The three \emph{cumulative distributions}. The three
support sets, $[0,1]$, $C_{\nicefrac{1}{3}}$, and $C_{\nicefrac{1}{4}}$
are IFSs, and they are also presented in detail inside Table \ref{tab:IFSmeas}
above.}
\end{figure}

\begin{figure}[H]
\includegraphics[width=0.65\textwidth]{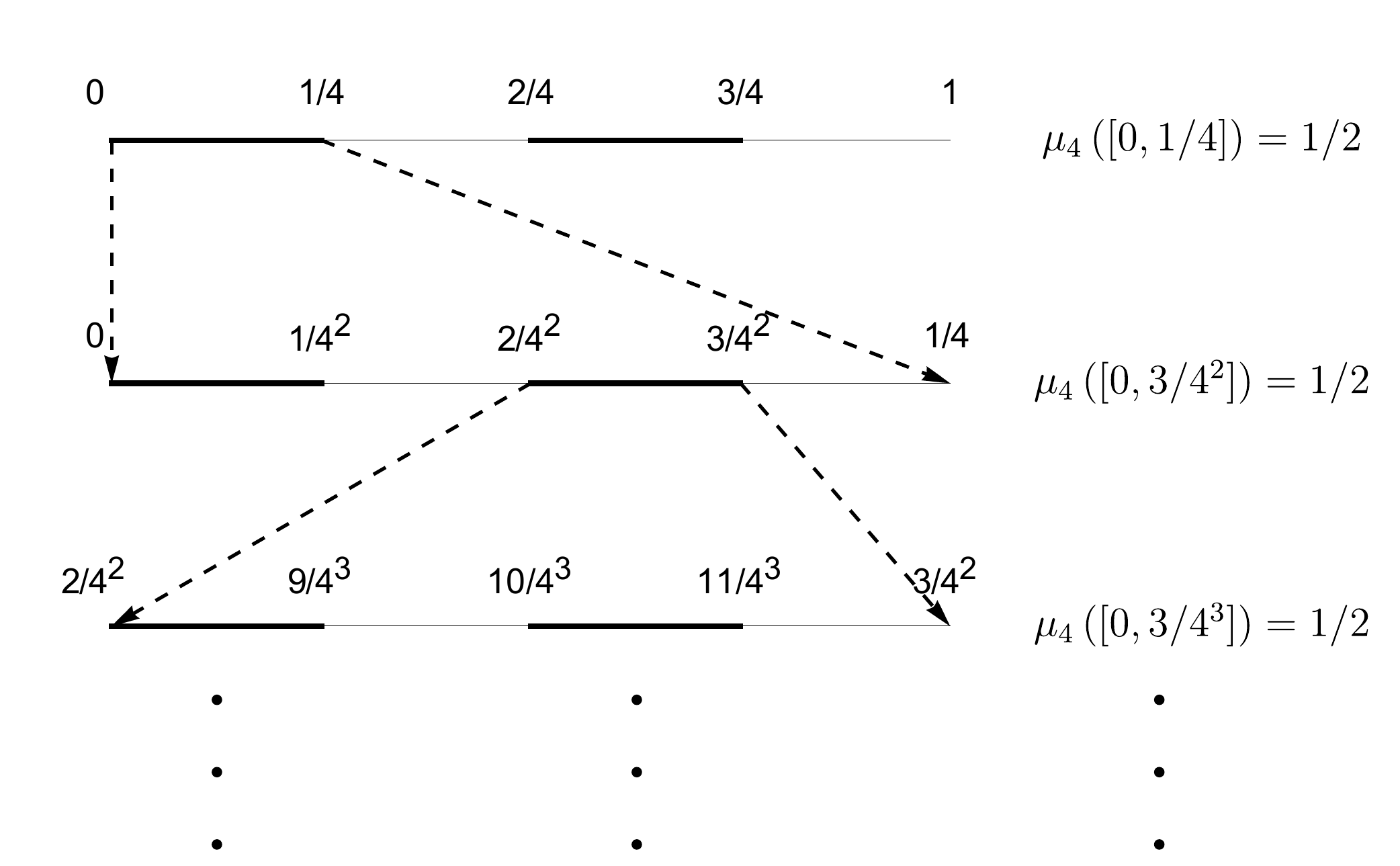}

\caption{\label{fig:u4}Illustration of $F_{\nicefrac{1}{4}}\left(x\right)=\mu_{4}\left(\left[0,x\right]\right)$
in Ex 3. Note that $\inf\{F_{\nicefrac{1}{4}}^{-1}\left(1/2\right)\}=\frac{1}{4}-\left(\sum_{n=2}^{\infty}\frac{1}{4^{n}}\right)=\frac{1}{6}$,
and $\inf\{F_{\nicefrac{1}{4}}^{-1}\left(1\right)\}=\frac{2}{3}$\@. }
\end{figure}

\begin{figure}[H]
\begin{tabular}{c>{\centering}m{0.25\textwidth}>{\centering}m{0.25\textwidth}>{\centering}m{0.25\textwidth}}
Ex 1 & \includegraphics[width=0.25\textwidth]{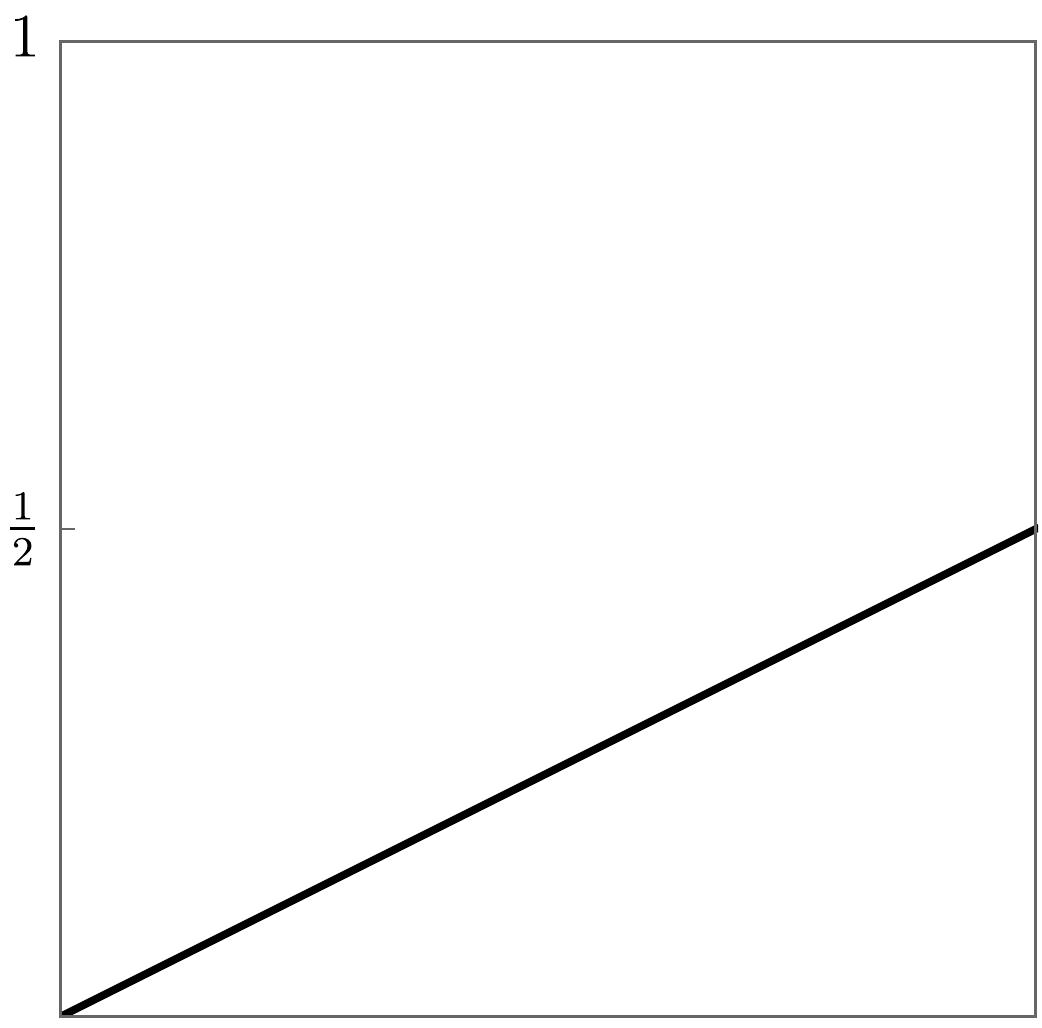} & \includegraphics[width=0.25\textwidth]{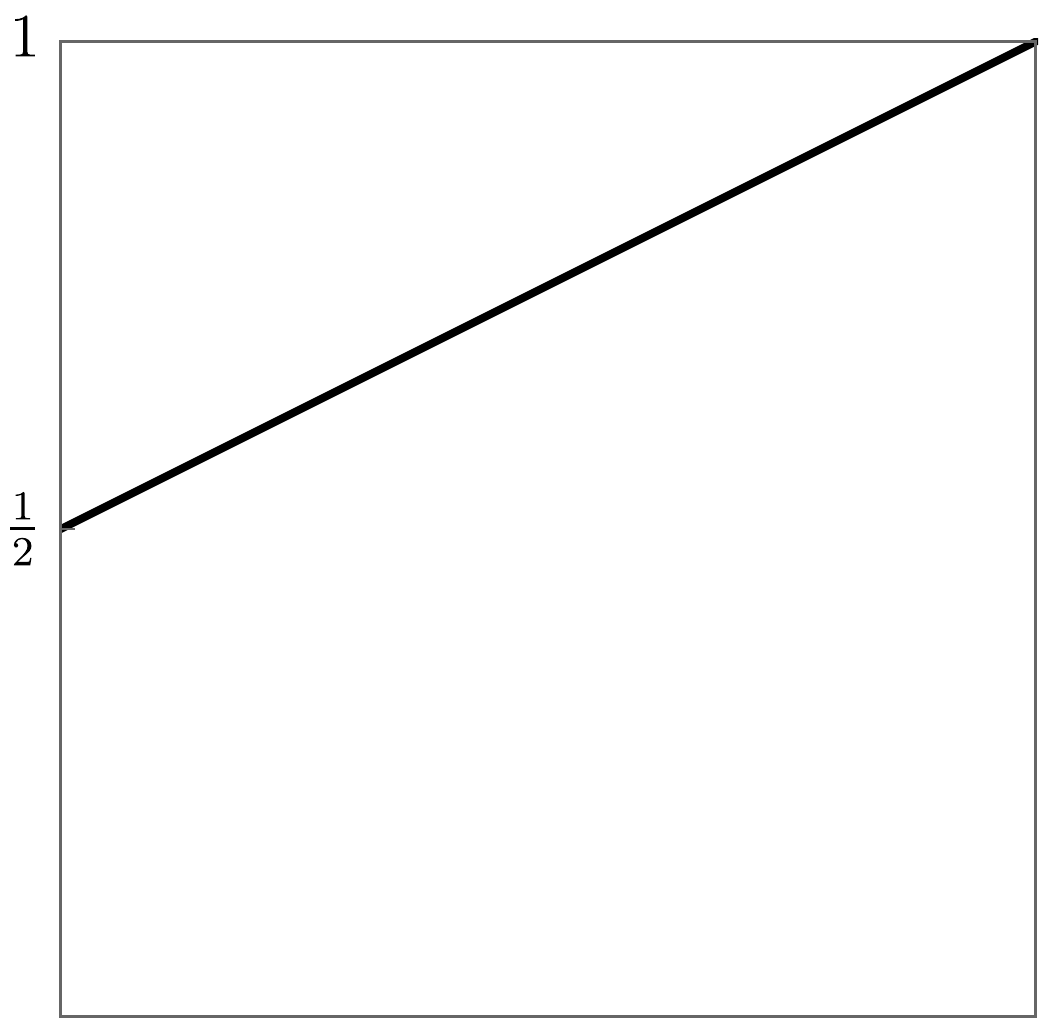} & \includegraphics[width=0.25\textwidth]{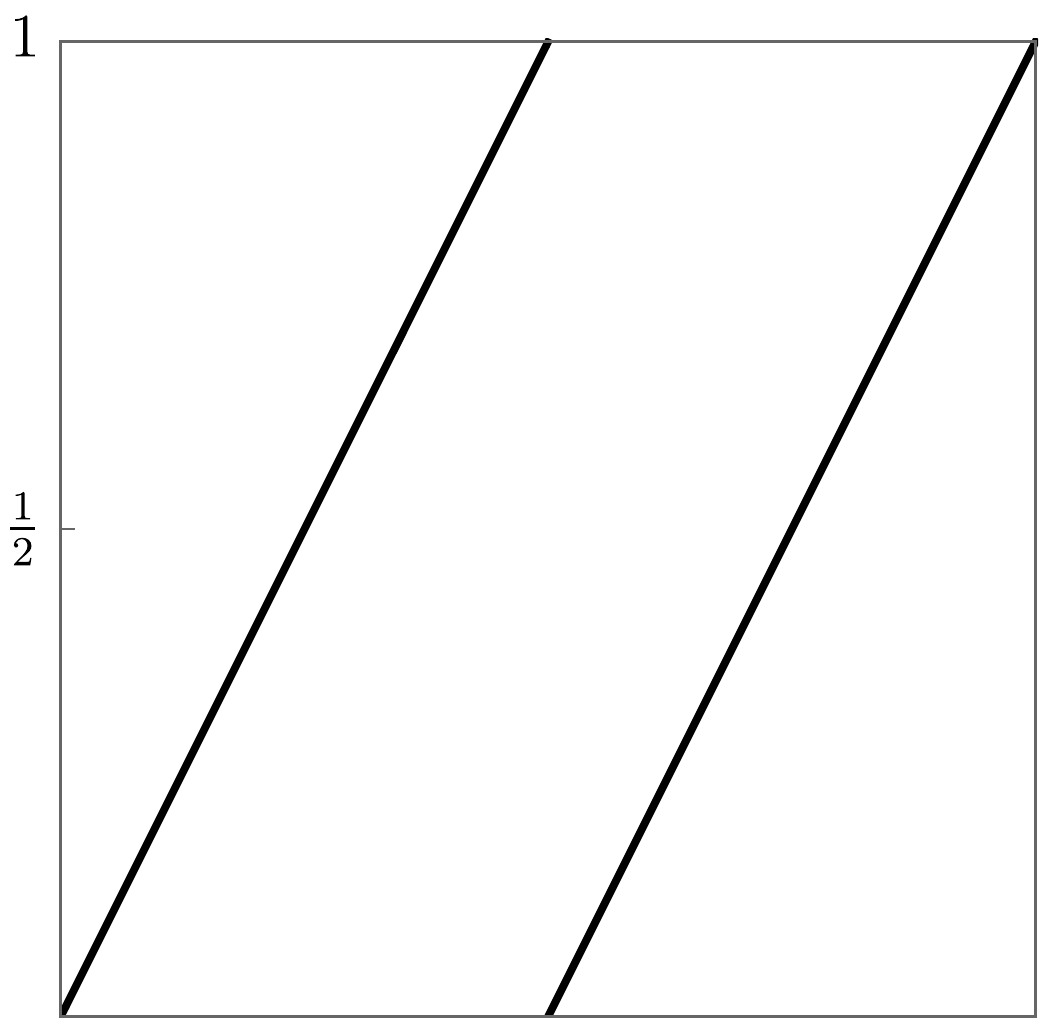}\tabularnewline
 & $\tau_{0}\left(x\right)=\frac{x}{2}$ & $\tau_{1}\left(x\right)=\frac{x+1}{2}$ & $\sigma\left(x\right)=2x$ mod 1\tabularnewline
 &  &  & \tabularnewline
Ex 2 & \includegraphics[width=0.25\textwidth]{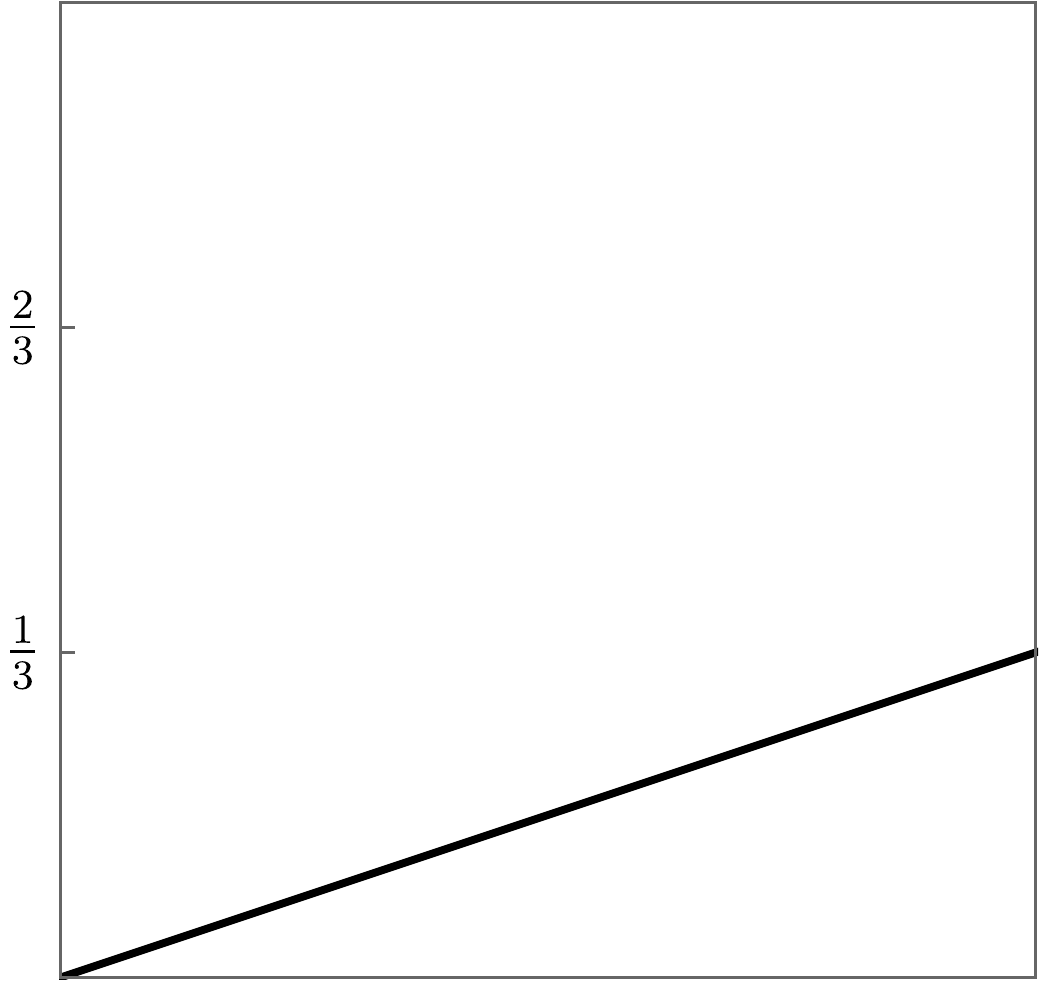} & \includegraphics[width=0.25\textwidth]{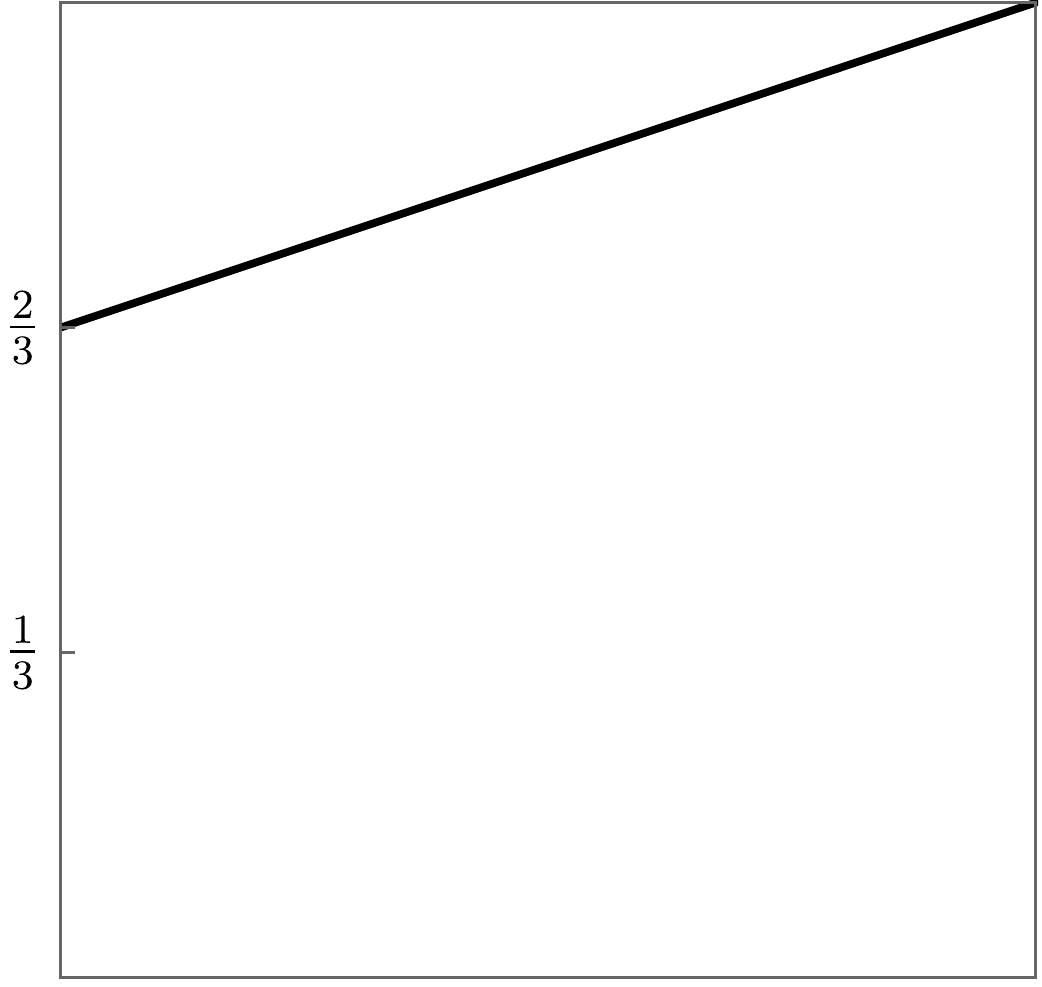} & \includegraphics[width=0.25\textwidth]{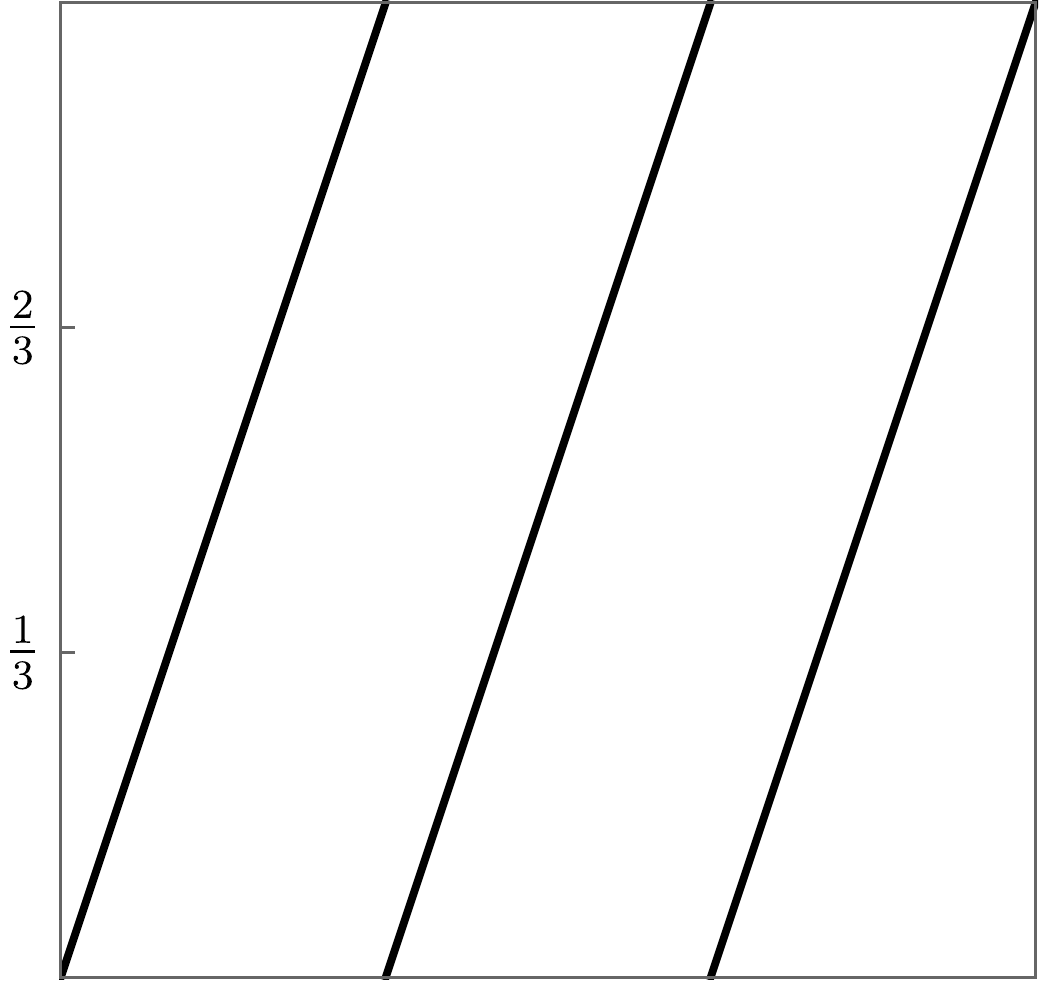}\tabularnewline
 & $\tau_{0}\left(x\right)=\frac{x}{3}$ & $\tau_{1}\left(x\right)=\frac{x+2}{3}$ & $\sigma\left(x\right)=3x$ mod 1\tabularnewline
 &  &  & \tabularnewline
Ex 3 & \includegraphics[width=0.25\textwidth]{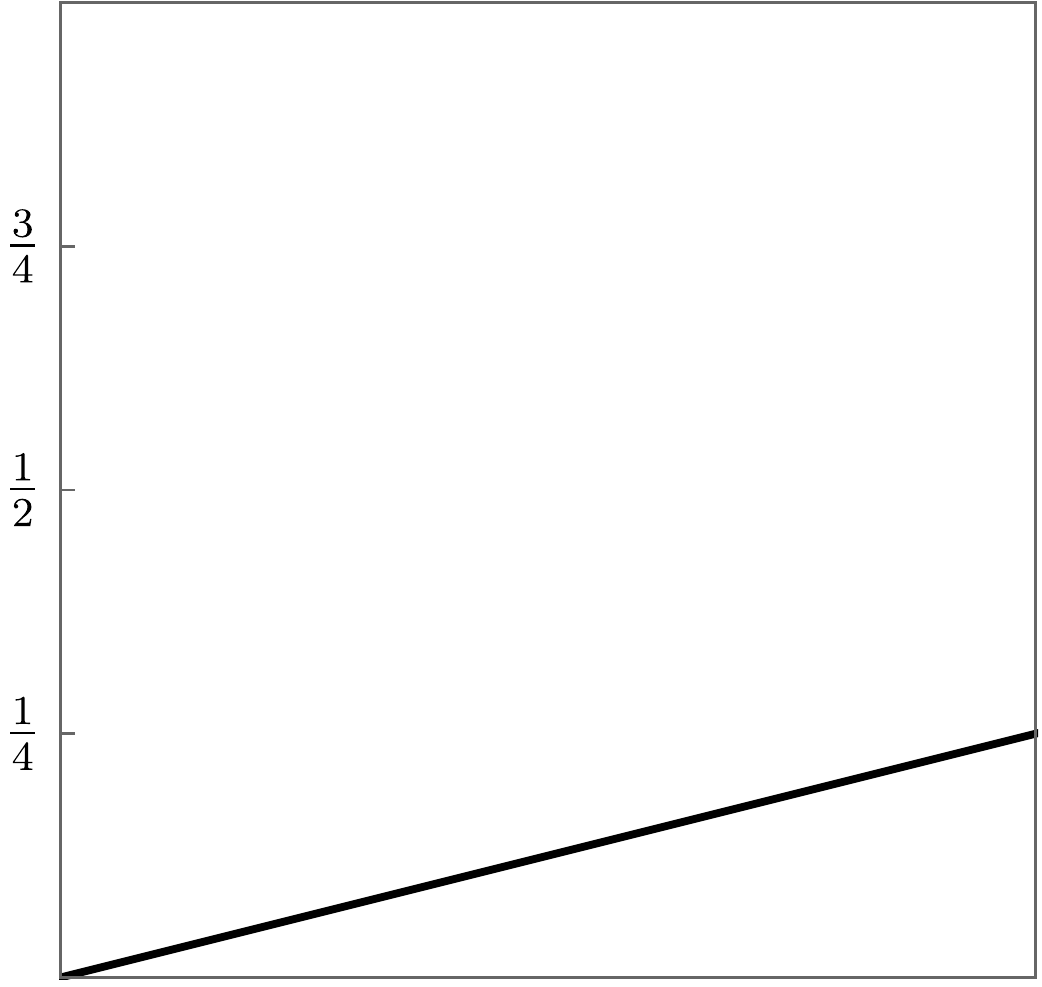} & \includegraphics[width=0.25\textwidth]{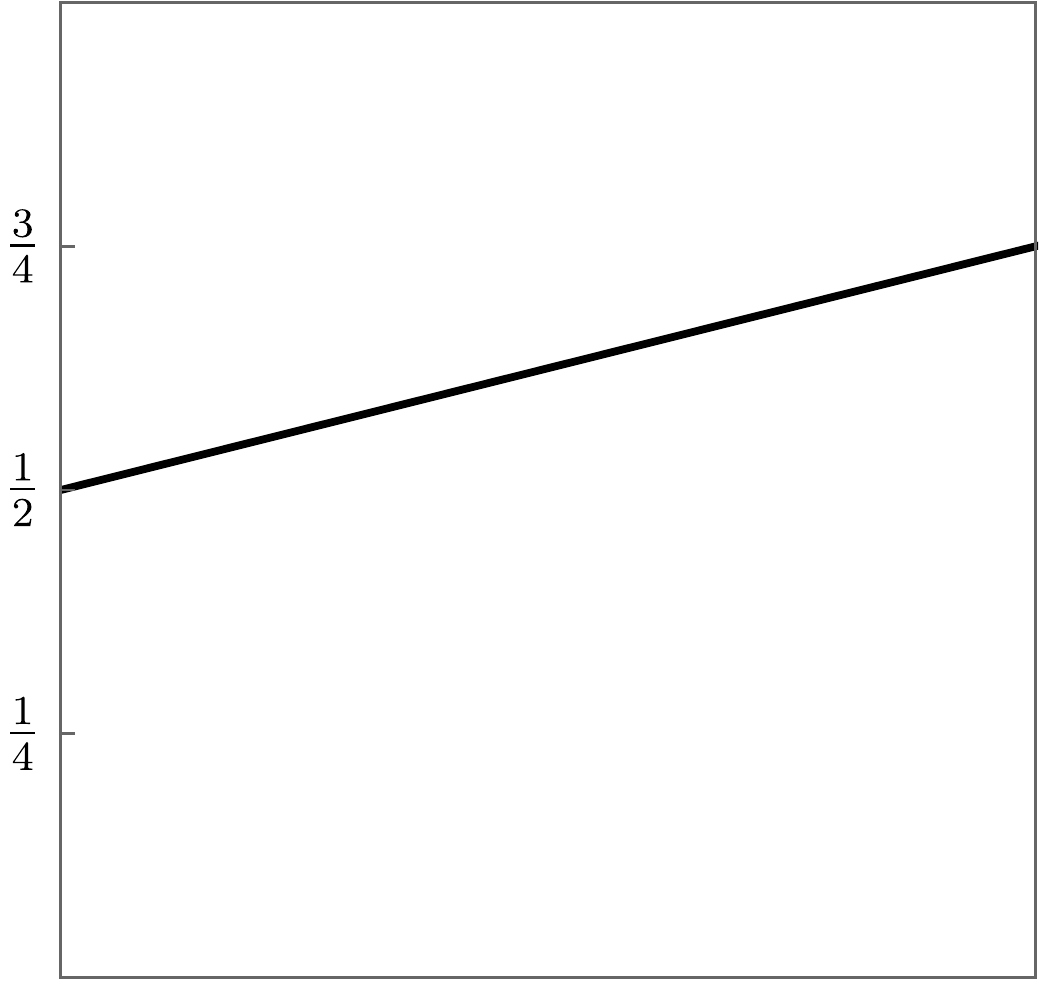} & \includegraphics[width=0.25\textwidth]{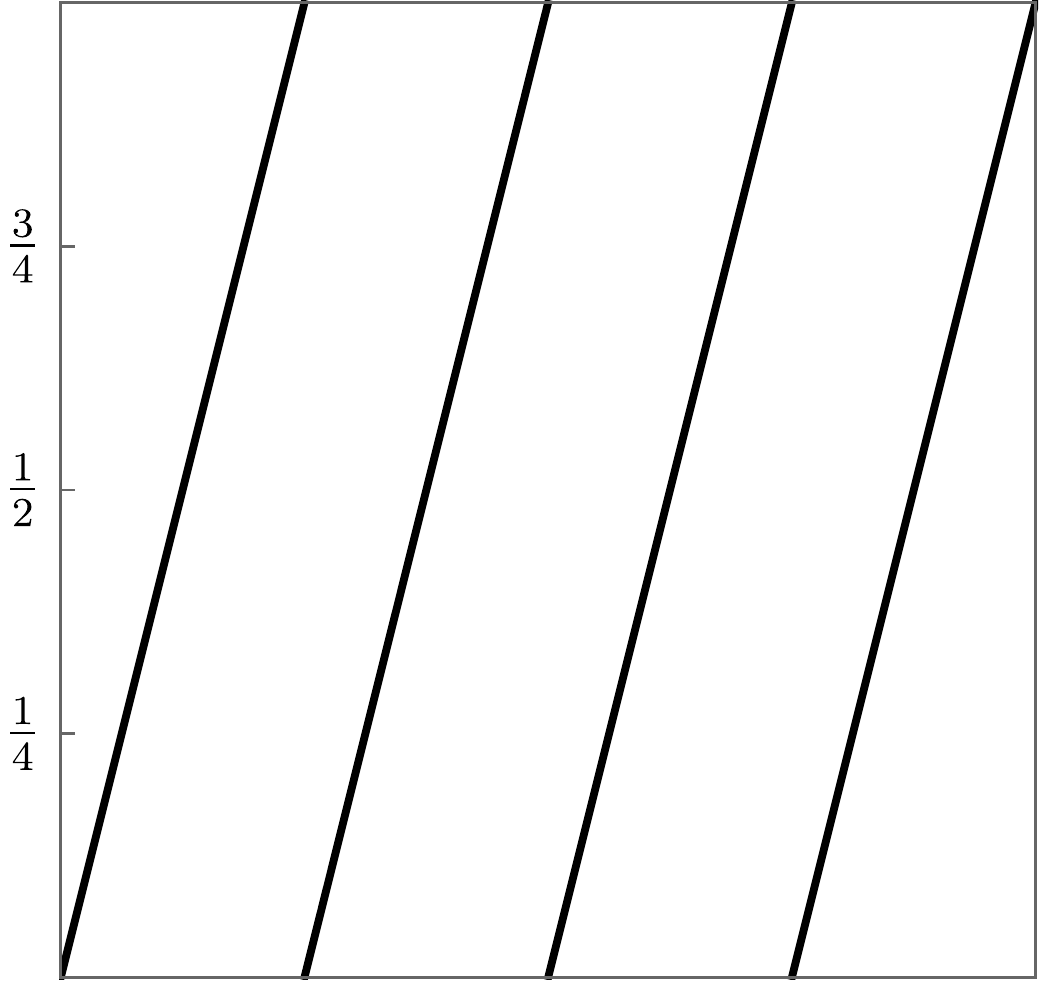}\tabularnewline
 & $\tau_{0}\left(x\right)=\frac{x}{4}$ & $\tau_{1}\left(x\right)=\frac{x+2}{4}$ & $\sigma\left(x\right)=4x$ mod 1\tabularnewline
 &  &  & \tabularnewline
\end{tabular}

\caption{\label{fig:ex3}The \emph{endomorphisms} in the three examples.}
\end{figure}

\begin{flushleft}
\textbf{Bit-representation of the respective IFSs in each of the three
examples.}
\par\end{flushleft}

In the three examples from Table \ref{tab:IFSmeas}, the associated
random variable $Y$ (from Theorem \ref{thm:pp}, eq (\ref{eq:d5}))
is as follows: 
\begin{alignat*}{2}
\text{Ex 1} &  & \qquad & Y_{\lambda}\left(\varepsilon_{i}\right)=\frac{1}{2}\sum_{i=1}^{\infty}\frac{\varepsilon_{i}}{2^{i}},\\
\text{Ex 2} &  &  & Y_{\mu_{3}}\left(\varepsilon_{i}\right)=\sum_{i=1}^{\infty}\frac{\varepsilon_{i}}{3^{i}},\quad\text{and}\\
\text{Ex 3} &  &  & Y_{\mu_{4}}\left(\varepsilon_{i}\right)=\sum_{i=1}^{\infty}\frac{\varepsilon_{i}}{4^{i}},\quad\varepsilon_{i}\in\left\{ 0,2\right\} ,\;\left(\varepsilon_{i}\right)\in\Omega_{2}.
\end{alignat*}

\begin{flushleft}
\textbf{Boundary Representation for the two measures $\mu_{3}$ and
$\mu_{4}$}, (see Theorem \ref{thm:pp}, and Figs \ref{fig:cantor}-\ref{fig:cum}.)
\par\end{flushleft}
\begin{defn}
Let $\mu$ be a (singular) measure with support contained in the interval
$I=\left[0,1\right]\simeq\partial\mathbb{D}$, the boundary of the
disk $\mathbb{D}=\left\{ z\in\mathbb{Z}\mathrel{{;}}\left|z\right|<1\right\} $.

A function $K:\mathbb{D}\times I\rightarrow\mathbb{C}$ is said to
be a \emph{boundary representation} iff (Def) the following four axioms
hold: 
\begin{enumerate}
\item \label{enu:cb1}$K\left(\cdot,x\right)$ is analytic in $\mathbb{D}$
for all $x\in I$; 
\item $K\left(z,\cdot\right)\in L^{2}\left(\mu\right)$, $\forall z\in\mathbb{D}$; 
\item Setting, for $f\in L^{2}\left(I,\mu\right)$, 
\begin{equation}
\left(Kf\right)\left(z\right)=\int_{0}^{1}f\left(x\right)K\left(z,x\right)d\mu\left(x\right),\label{eq:k9}
\end{equation}
then $Kf\in H_{2}\left(\mathbb{D}\right)$, the Hardy-space; and
\item \label{enu:cb4}The following limit exists in the $L^{2}\left(\mu\right)$-norm:
\[
\lim_{\substack{r\nearrow1\\
r<1
}
}\left(Kf\right)\left(r\,e\left(x\right)\right)=f\left(x\right),
\]
where $e\left(x\right):=e^{i2\pi x}$, $x\in I$. 
\end{enumerate}
We say that $K$ is \emph{self-reproducing} if there is a kernel $K^{\mathbb{C}}:\mathbb{D}\times\mathbb{D}\rightarrow\mathbb{C}$
satisfying 
\begin{equation}
\lim_{r\nearrow1}K^{\mathbb{C}}\left(z,r\,e\left(x\right)\right)=K\left(z,x\right),\;\forall z\in\mathbb{D},\:x\in I;\label{eq:k9a}
\end{equation}
and 
\begin{equation}
\int_{0}^{1}K\left(z,x\right)\overline{K\left(w,x\right)}d\mu\left(x\right)=K^{\mathbb{C}}\left(z,w\right),\;\forall\left(z,w\right)\in\mathbb{D}\times\mathbb{D}.\label{eq:k9b}
\end{equation}

\end{defn}

\begin{rem}
When $\left(K,K^{\mathbb{C}}\right)$ satisfy the two conditions (\ref{eq:k9a})-(\ref{eq:k9b}),
it is immediate that $K^{\mathbb{C}}$ is then a \emph{positive definite
kernel} on $\mathbb{D}\times\mathbb{D}$. We shall denote the corresponding
\emph{reproducing kernel Hilbert space} RKHS by $\mathscr{H}\left(K^{\mathbb{C}}\right)$;
see \cite{MR0044033,MR0088706}.

Furthermore, the assignment 
\begin{equation}
T_{\mu}:\underset{{\scriptstyle \text{as a function on \ensuremath{\mathbb{D}}}}}{\underbrace{K^{\mathbb{C}}\left(\cdot,z\right)}}\longmapsto\underset{{\scriptstyle \text{in \ensuremath{L^{2}\left(I,\mu\right)}}}}{\underbrace{K\left(z,\cdot\right)}}\label{eq:k9c}
\end{equation}
extends by linearity, and norm-closure to an \emph{isometry}: 
\begin{equation}
T_{\mu}:\mathscr{H}\left(K^{\mathbb{C}}\right)\longrightarrow L^{2}\left(\mu\right);\label{eq:k9d}
\end{equation}
with ``isometry'' relative to the respective Hilbert norms in (\ref{eq:k9d}).

Moreover, the adjoint operator
\begin{equation}
T_{\mu}^{*}:L^{2}\left(\mu\right)\longrightarrow\mathscr{H}\left(K^{\mathbb{C}}\right)\label{eq:k9e}
\end{equation}
is the original operator $Kf$ specified in (\ref{eq:k9}), i.e.,
for $\forall f\in L^{2}\left(\mu\right)$, we have: 
\begin{equation}
\left(T_{\mu}^{*}f\right)\left(z\right)=\int_{0}^{1}f\left(x\right)K\left(z,x\right)d\mu\left(x\right),\;\forall z\in\mathbb{D}.\label{eq:k9f}
\end{equation}
\end{rem}

\begin{cor}
If the measure $\mu$ (as above) has a self-reproducing kernel $K^{\mathbb{C}}$,
then the corresponding operator $K$ (see (\ref{eq:k9})) satisfies
\begin{equation}
KT_{\mu}=I_{\mathscr{H}\left(K^{\mathbb{C}}\right)}\label{eq:k9g}
\end{equation}
where the subscript refers to the identity operator in the RKHS $\mathscr{H}\left(K^{\mathbb{C}}\right)$. 
\end{cor}

\begin{prop}
Each of the measures $\mu_{3}$ and $\mu_{4}$ from Fig \ref{fig:cum}
has a boundary representation. 
\end{prop}

\begin{proof}
We shall refer the reader to the two papers \cite{MR1655831} and
\cite{MR3796641}. In the case of $\mu_{4}$, the construction is
as follows: 
\begin{equation}
K_{4}\left(z,x\right)=\prod_{n=0}^{\infty}\left(1+\left(\overline{e\left(x\right)}z\right)^{4^{n}}\right),\label{eq:k10}
\end{equation}
and we refer to \cite{MR1655831} for details. 

In the case of $\mu_{3}$, let $b$ be the \emph{inner function} corresponding
to $\mu_{3}$ via the Herglotz-formula; then 
\begin{equation}
K_{3}\left(z,x\right)=\frac{1-b\left(z\right)\overline{b\left(e\left(x\right)\right)}}{1-z\overline{e\left(x\right)}}.\label{eq:k11}
\end{equation}
For the proof details, showing that $K_{3}$ in (\ref{eq:k11}) satisfies
conditions (\ref{enu:cb1})-(\ref{enu:cb4}), readers are referred
to \cite{MR3796641}. 
\end{proof}
\begin{cor}
The two kernels $K_{4}$ and $K_{3}$ are self-reproducing.
\end{cor}

\section{Endomorphisms and Invariance}

The purpose of this section is to make precise connections between
the following three tools from non-commutative analysis: The representations,
$Rep\left(\mathcal{O}_{N},\mathscr{H}\right)$; (ii) endomorphisms
in $\mathscr{B}\left(\mathscr{H}\right)$ of index $N$, and (iii)
certain unitary operators in $\mathscr{H}$.

These interconnections will play a role in the rest of the paper.
Some references relevant to (ii) are \cite{MR1978577,MR1444086,MR2030387}.
\begin{defn}
Let $\mathscr{H}$ be a separable Hilbert space, and let $\alpha:\mathscr{B}\left(\mathscr{H}\right)\longrightarrow\mathscr{B}\left(\mathscr{H}\right)$
be linear, satisfying (for all $A,B\in\mathscr{B}\left(\mathscr{H}\right)$): 
\begin{enumerate}
\item $\alpha\left(AB\right)=\alpha\left(A\right)\alpha\left(B\right)$, 
\item $\alpha\left(A^{*}\right)=\alpha\left(A\right)^{*}$, and 
\item $\alpha\left(I\right)=I$. 
\end{enumerate}
We then say that $\alpha$ is an \emph{endomorphism}.

\end{defn}

\begin{lem}
\label{lem:US}Fix an endomorphism $\alpha\in End\left(\mathscr{B}\left(\mathscr{H}\right)\right)$,
and let $\pi\in Rep\left(\mathcal{O}_{N},\mathscr{H}\right)$, $\pi\left(s_{i}\right):=S_{i}$,
see (\ref{eq:a5})-(\ref{eq:a7}). 
\begin{enumerate}
\item \label{enu:c1}Then a given unitary operator, $U:\mathscr{H}\rightarrow\mathscr{H}$,
satisfies 
\begin{equation}
US_{i}=\alpha\left(S_{i}\right),\;\forall i\in\left\{ 1,\cdots,N\right\} \label{eq:c1}
\end{equation}
if and only if 
\begin{equation}
U=U_{\alpha}=\sum_{i=1}^{N}\alpha\left(S_{i}\right)S_{i}^{*}.\label{eq:c2}
\end{equation}
\item \label{enu:c2}Conversely, if an operator $U_{\alpha}$ is given by
eq. (\ref{eq:c2}), then \uline{it is a unitary} operator in $\mathscr{H}$. 
\end{enumerate}
\end{lem}

\begin{proof}
(\ref{eq:c1})$\Rightarrow$(\ref{eq:c2}). If (\ref{eq:c1}) holds,
right-multiply by $S_{i}^{*}$, and perform the summation $\sum_{i=1}^{N}$,
using (\ref{eq:a7}). 

(\ref{eq:c2})$\Rightarrow$(\ref{eq:c1}). If (\ref{eq:c2}) holds,
right-multiply by $S_{j}$, and use (\ref{eq:a7}). 

We now turn to (\ref{enu:c2}): Given $\alpha\in End\left(\mathscr{B}\left(\mathscr{H}\right)\right)$,
then set 
\begin{equation}
U_{\alpha}=\sum_{i=1}^{N}\alpha\left(S_{i}\right)S_{i}^{*};\label{eq:c3}
\end{equation}
it follows that:
\begin{align*}
U_{\alpha}U_{\alpha}^{*} & =\sum_{i}\sum_{j}\alpha\left(S_{i}\right)\underset{\delta_{ij}I}{\underbrace{S_{i}^{*}S_{j}}}\alpha\left(S_{j}^{*}\right)\\
 & =\sum_{i}\alpha\left(S_{i}\right)\alpha\left(S_{i}^{*}\right)=\alpha\left(\sum_{i}S_{i}S_{i}^{*}\right)=\alpha\left(I\right)=I.
\end{align*}
Similarly, 
\begin{align*}
U_{\alpha}^{*}U_{\alpha} & =\sum_{i}\sum_{j}S_{i}\alpha\left(S_{i}\right)^{*}\alpha\left(S_{j}\right)S_{j}^{*}\\
 & =\sum_{i}\sum_{j}S_{i}\alpha\big(\underset{\delta_{ij}I}{\underbrace{S_{i}^{*}S_{j}}}\big)S_{j}^{*}\\
 & =\sum_{i}S_{i}\alpha\left(I\right)S_{i}^{*}=\sum_{i}S_{i}IS_{i}^{*}=\sum_{i}S_{i}S_{i}^{*}=I,
\end{align*}
which is the desired conclusion. 
\end{proof}
\begin{thm}
Let $\mathscr{H}$ be a separable Hilbert space. Fix $N\in\mathbb{N}$,
$N>2$. Let $\alpha\in End\left(\mathscr{B}\left(\mathscr{H}\right)\right)$,
and $\pi\in Rep\left(\mathcal{O}_{N},\mathscr{H}\right)$, and 
\begin{equation}
X:\Omega_{N}\longrightarrow\mathfrak{M}\label{eq:c4}
\end{equation}
be the random variable from Theorem \ref{thm:LP}, and set $U=U_{\alpha}$
(see (\ref{eq:c2})). 

Then we have: 
\begin{equation}
\alpha\left(S_{\pi_{1}\left(\omega\right)}\right)X\left(\sigma\left(\omega\right)\right)\alpha\left(S_{\pi_{1}\left(\omega\right)}^{*}\right)=U_{\alpha}X\left(\omega\right)U_{\alpha}^{*}\label{eq:c5}
\end{equation}
for all $\omega\in\Omega_{N}$. (For terminology, see Theorem \ref{thm:LP}
above.)

The corresponding covariance formula for the projection valued measure
$\mathbb{Q}_{\pi}\left(\cdot\right)$ from Theorem \ref{thm:pv} is:
\begin{equation}
\alpha\left(S_{\pi_{1}\left(E\right)}\right)\mathbb{Q}_{\pi}\left(E\right)\alpha\left(S_{\pi_{1}\left(E\right)}^{*}\right)=U_{\alpha}\mathbb{Q}_{\pi}\left(E\right)U_{\alpha}^{*},\label{eq:c6}
\end{equation}
for all $E\in\mathscr{C}$ (the $\sigma$-algebra of subsets of $\Omega_{N}$).
\end{thm}

\begin{proof}
By (\ref{eq:a8}) in Theorem \ref{thm:LP}, it is enough to consider
finite words, e.g., $f=\left(i_{1},i_{2},\cdots,i_{k}\right)$, $i_{j}\in\left\{ 1,2,\cdots,N\right\} $.
Set $S_{f}:=S_{i_{1}}S_{i_{2}}\cdots S_{i_{k}}$. For the approximation
to the RHS in (\ref{eq:c5}), we have:
\begin{align*}
U_{\alpha}S_{f}S_{f}^{*}U_{\alpha}^{*} & =\alpha\left(S_{i_{1}}\right)S_{i_{2}}\cdots S_{i_{k}}S_{i_{k}}^{*}\cdots S_{i_{2}}^{*}\alpha\left(S_{i_{1}}^{*}\right)\\
 & \quad\left(\text{by \ensuremath{\left(\ref{eq:c1}\right)} in Lemma \ref{lem:US}}\right)
\end{align*}
and the desired conclusion (\ref{eq:c5}) now follows from (\ref{eq:a8})
in Theorem \ref{thm:LP}. 
\end{proof}

\section{\label{sec:uh}Representations in a Universal Hilbert Space}

Our starting point is a compact Hausdorff space $M$ and continuous
maps $\sigma\colon M\rightarrow M$, $\tau_{i}\colon M\rightarrow M$,
$i=1,\dots,N$, such that 
\begin{equation}
\sigma\circ\tau_{i}=id_{M}.\label{eqInt.6}
\end{equation}
It follows from (\ref{eqInt.6}) that $\sigma$ is onto, and that
each $\tau_{i}$ is one-to-one. We will be especially interested in
the case when there are distinct branches $\tau_{i}\colon M\rightarrow M$
such that
\begin{equation}
\bigcup_{i=1}^{N}\tau_{i}\left(M\right)=M.\label{eqInt.7}
\end{equation}
For such systems, we show that there is a \emph{universal} representation
of $\mathcal{O}_{N}$ in a Hilbert space $\mathscr{H}\left(M\right)$
which is functorial, is naturally defined, and contains every representation
of $\mathcal{O}_{N}$.

The elements in\emph{ the universal Hilbert space} $\mathscr{H}\left(M\right)$
are equivalence classes of pairs $\left(\varphi,\mu\right)$ where
$\varphi$ is a Borel function on $M$ and where $\mu$ is a positive
Borel measure on $M$. We will set $\varphi\sqrt{d\mu}:=\text{class}\left(\varphi,\mu\right)$
for reasons which we spell out below.

While our present methods do adapt to the more general framework when
the space $M$ of (\ref{eqInt.6})-(\ref{eqInt.7}) is not assumed
compact, but only $\sigma$-compact, we will still restrict the discussion
here to the compact case. This is for the sake of simplicity of the
technical arguments. But we encourage the reader to follow our proofs
below, and to formulate for him/her\-self the corresponding results
when $M$ is not necessarily assumed compact. Moreover, if $M$ is
not compact, then there is a variety of special cases to take into
consideration, various abstract notions of ``escape to infinity''.
We leave this wider discussion for a later investigation, and we only
note here that our methods allow us to relax the compactness restriction
on $M$.

There is a classical construction in operator theory which lets us
realize point transformations in Hilbert space. It is called the Koopman
representation; see, for example, \cite[p.~135]{MR1043174}. But this
approach only applies if the existence of invariant, or quasi-invariant,
measures is assumed. In general such measures are not available. We
propose a different way of realizing families of point transformations
in Hilbert space in a general context where no such assumptions are
made. Our Hilbert spaces are motivated by S. Kakutani \cite{MR0023331},
L. Schwartz, and E. Nelson \cite{MR0282379}, among others. The reader
is also referred to an updated presentation of the measure-class Hilbert
spaces due to Tsirelson \cite{MR2029632} and Arveson \cite[Chapter 14]{MR1978577}.

We say that $\left(\varphi,\mu\right)\sim\left(\psi,\nu\right)$ if
there is a third positive Borel measure $\lambda$ on $M$ such that
$\mu\ll\lambda$, $\nu\ll\lambda$, and
\begin{equation}
\varphi\sqrt{\frac{d\mu}{d\lambda}}=\psi\sqrt{\frac{d\nu}{d\lambda}},\quad\lambda\;\mathrm{a.e.}\text{ on }M,\label{eqInt.8}
\end{equation}
where $\ll$ denotes relative absolute continuity, and where $d\mu/d\lambda$
denotes the usual Radon-Nikodym derivative, i.e., $d\mu/d\lambda\in L^{1}\left(\lambda\right)$,
and $d\mu=\left(d\mu/d\lambda\right)d\lambda$.

One checks that $\sim$ for pairs $\left(\varphi,\mu\right)$, i.e.,
(function, measure), indeed defines an \emph{equivalence relation}.
Notation: $\text{class}\left(\varphi,\mu\right)=:\varphi\sqrt{d\mu}$.

We shall review some basic properties of the Hilbert space $\mathscr{H}\left(M\right)$.
This space is called the Hilbert space of $\sigma$-functions, or
square densities, and it was studied for different reasons in earlier
papers of L. Schwartz, E. Nelson \cite{MR0282379}, and W. Arveson
\cite{MR2026010}.
\begin{thm}
\label{thm:iso}Isometries $S_{i}\colon\mathscr{H}\left(M\right)\rightarrow\mathscr{H}\left(M\right)$
are defined by 
\begin{equation}
S_{i}\colon\left(\varphi,\mu\right)\longmapsto\left(\varphi\circ\sigma,\mu\circ\tau_{i}^{-1}\right),\label{eqInt.9}
\end{equation}
or equivalently, $S_{i}\colon\varphi\sqrt{d\mu}\mapsto\varphi\circ\sigma\sqrt{d\mu\circ\tau_{i}^{-1}}$,
and these operators satisfy the Cuntz relations.
\end{thm}

\begin{proof}
Note that, at the outset, it is not even clear \emph{a priori} that
$S_{i}$ in (\ref{eqInt.9}) defines a transformation of $\mathscr{H}\left(M\right)$.
To verify this, we will need to show that if two equivalent pairs
are substituted on the left-hand side in (\ref{eqInt.9}), then they
produce equivalent pairs as output, on the right-hand side. Recalling
the definition (\ref{eqInt.8}) of the equivalence relation $\sim$,
there is no obvious or intuitive reason for why this should be so.

Before turning to the proof, we shall need some preliminaries and
lemmas. 
\end{proof}
To stress the intrinsic transformation rules of $\mathscr{H}\left(M\right)$,
the vectors in $\mathscr{H}\left(M\right)$ are usually denoted $\varphi\sqrt{d\mu}$
rather than $\left(\varphi,\mu\right)$. This suggestive notation
motivates the definition of the \emph{inner product} of $\mathscr{H}\left(M\right)$.
If $\varphi\sqrt{d\mu}$ and $\psi\sqrt{d\nu}$ are in $\mathscr{H}\left(M\right)$,
we define their Hilbert inner product by
\begin{equation}
\left\langle \varphi\sqrt{d\mu},\psi\sqrt{d\nu}\right\rangle :=\int_{M}\bar{\varphi}\psi\sqrt{\frac{d\mu}{d\lambda}}\sqrt{\frac{d\nu}{d\lambda}}d\lambda,\label{eqInt.10}
\end{equation}
where $\lambda$ is some positive Borel measure, chosen such that
$\mu\ll\lambda$ and $\nu\ll\lambda$. For example, we could take
$\lambda=\mu+\nu$. To be in $\mathscr{H}\left(M\right)$, $\varphi\sqrt{d\mu}$
must satisfy
\begin{equation}
\left\Vert \varphi\sqrt{d\mu}\right\Vert ^{2}=\int_{M}\left\vert \varphi\right\vert ^{2}\frac{d\mu}{d\lambda}d\lambda=\int_{M}\left\vert \varphi\right\vert ^{2}d\mu<\infty.\label{eqInt.11}
\end{equation}

\subsection{\label{subsec:IsoH}Isometries in $\mathscr{H}\left(M\right)$}

In this preliminary section we prove three general facts about the
process of \emph{inducing operators in the Hilbert space} $\mathscr{H}\left(M\right)$
from underlying point transformations in $M$. The starting point
is a given continuous mapping $\sigma\colon M\rightarrow M$, mapping
onto $M$. We will be concerned with the special case when $M$ is
a compact Hausdorff space, and when there is one or more continuous
branches $\tau_{i}\colon M\rightarrow M$ of the inverse, i.e., when
\begin{equation}
\sigma\circ\tau_{i}=id_{M}.\label{eqIsoNew.1}
\end{equation}
Recall that \emph{elements in} $\mathscr{H}\left(M\right)$ \emph{are
equivalence classes} of pairs $\left(\varphi,\mu\right)$ where $\varphi$
is a Borel function on $M$, $\mu$ is a positive Borel measure on
$M$, and $\int_{M}\left\vert \varphi\right\vert ^{2}d\mu<\infty$.
An equivalence class will be denoted $\varphi\sqrt{d\mu}$, and we
show that there are isometries 
\begin{equation}
S_{i}\colon\varphi\sqrt{d\mu}\longmapsto\varphi\circ\sigma\sqrt{d\mu\circ\tau_{i}^{-1}},\label{eqIsoNew.2}
\end{equation}
with orthogonal ranges in the Hilbert space $\mathscr{H}\left(M\right)$.
Moreover, we calculate an explicit formula for the adjoint co-isometries
$S_{i}^{\ast}$.

\begin{lem}
\label{LemIso.2}Let $M$ be a compact Hausdorff space, and let the
mapping $\sigma\colon M\rightarrow M$ be onto. Suppose $\tau\colon M\rightarrow M$
satisfies $\sigma\circ\tau=id_{M}$. Assume that both $\sigma$ and
$\tau$ are continuous. Let $\mathscr{H}=\mathscr{H}\left(M\right)$
be the Hilbert space of classes $\left(\varphi,\mu\right)$ where
$\varphi$ is a Borel function on $M$ and $\mu$ is a positive Borel
measure such that $\int\left\vert \varphi\right\vert ^{2}d\mu<\infty$.
The equivalence relation is defined in the usual way: two pairs $\left(\varphi,\mu\right)$
and $\left(\psi,\nu\right)$ are said to be equivalent, written $\left(\varphi,\mu\right)\sim\left(\psi,\nu\right)$,
if for some positive measure $\lambda$, $\mu\ll\lambda$, $\nu\ll\lambda$,
we have the following identity:
\begin{equation}
\varphi\sqrt{\frac{d\mu}{d\lambda}}=\psi\sqrt{\frac{d\nu}{d\lambda}}\quad(\mathrm{a.e.}\;\lambda).\label{eqIso.1}
\end{equation}
Then there is an isometry $S\colon\mathscr{H}\rightarrow\mathscr{H}$
which is well defined by the assignment
\begin{equation}
S\left(\left(\varphi,\mu\right)\right):=\left(\varphi\circ\sigma,\mu\circ\tau^{-1}\right),\label{eqIso.2}
\end{equation}
or
\[
S\colon\varphi\sqrt{d\mu}\longmapsto\varphi\circ\sigma\sqrt{d\mu\circ\tau^{-1}},
\]
where $\mu\circ\tau^{-1}\left(E\right):=\mu\left(\tau^{-1}\left(E\right)\right)$,
and $\tau^{-1}\left(E\right):=\left\{ x\in M\mid\tau\left(x\right)\in E\right\} $,
for $E\in\mathscr{B}\left(M\right)$. 
\end{lem}

\begin{proof}
We leave the verification of the following four facts to the reader;
see also \cite{MR0282379}.
\begin{enumerate}
\item \label{LemIso.2proof(1)}If ${\displaystyle \varphi\sqrt{\frac{d\mu}{d\lambda}}=\psi\sqrt{\frac{d\nu}{d\lambda}}}$
for some $\lambda$ such that $\mu\ll\lambda$, $\nu\ll\lambda$,
and if some other measure $\lambda^{\prime}$ satisfies $\mu\ll\lambda^{\prime}$,
$\nu\ll\lambda^{\prime}$, then 
\[
\varphi\sqrt{\frac{d\mu}{d\lambda^{\prime}}}=\psi\sqrt{\frac{d\nu}{d\lambda^{\prime}}}\quad(\mathrm{a.e.}\,\lambda^{\prime}).
\]
\item \label{LemIso.2proof(2)}The ``vectors\textquotedblright{} in $\mathscr{H}$
are equivalence classes of pairs $\left(\varphi,\mu\right)$ as described
in the statement of the lemma. For two elements $\left(\varphi,\mu\right)$
and $\left(\psi,\nu\right)$ in $\mathscr{H}$, define the sum by
\begin{equation}
\left(\varphi,\mu\right)+\left(\psi,\nu\right):=\left(\phi\sqrt{\frac{d\mu}{d\lambda}}+\psi\sqrt{\frac{d\nu}{d\lambda}},\lambda\right),\label{eqIso.3}
\end{equation}
where $\lambda$ is a positive Borel measure satisfying $\mu\ll\lambda$,
$\nu\ll\lambda$. The sum in (\ref{eqIso.3}) is also written $\varphi\sqrt{d\mu}+\psi\sqrt{d\nu}$.
The definition of the sum (\ref{eqIso.3}) passes through the equivalence
relation $\sim$, i.e., we get an equivalent result on the right-hand
side in (\ref{eqIso.3}) if equivalent pairs are used as input on
the left-hand side. A similar conclusion holds for the definition
(\ref{eqIso.4}) below of the inner product $\left\langle \cdot,\cdot\right\rangle $
in the Hilbert space $\mathscr{H}$.
\item \label{LemIso.2proof(3)}Scalar multiplication, $c\in\mathbb{C}$,
is defined by $c\left(\varphi,\mu\right):=\left(c\mkern2mu \varphi,\mu\right)$,
and \emph{the Hilbert space inner product} is 
\begin{align}
\left\langle \varphi\sqrt{d\mu},\psi\sqrt{d\nu}\right\rangle =\left\langle \left(\varphi,\mu\right),\left(\psi,\nu\right)\right\rangle  & :=\int_{M}\overline{\varphi}\psi\sqrt{\frac{d\mu}{d\lambda}}\sqrt{\frac{d\nu}{d\lambda}}d\lambda\label{eqIso.4}
\end{align}
where $\mu\ll\lambda$, $\nu\ll\lambda$.
\item \label{LemIso.2proof(4)}It is known, see \cite{MR0282379}, that
$\mathscr{H}$ is a Hilbert space. In particular, it is \emph{complete}:
if a sequence $\left(\varphi_{n},\mu_{n}\right)$ in $\mathscr{H}$
satisfies 
\[
\lim_{n,m\rightarrow\infty}\left\Vert \left(\varphi_{n},\mu_{n}\right)-\left(\varphi_{m},\mu_{m}\right)\right\Vert ^{2}=0,
\]
then there is a pair $\left(\varphi,\mu\right)$ with
\begin{equation}
\int_{M}\left\vert \varphi\right\vert ^{2}\frac{d\mu}{d\lambda}d\lambda=\int_{M}\left\vert \varphi\right\vert ^{2}d\mu<\infty,\label{eqIso.5}
\end{equation}
where
\begin{equation}
\lambda:=\sum_{n=1}^{\infty}2^{-n}\mu_{n}\left(M\right)^{-1}\mu_{n},\label{eqIso.6}
\end{equation}
and $\left\Vert \left(\varphi,\mu\right)-\left(\varphi_{n},\mu_{n}\right)\right\Vert ^{2}\underset{n\rightarrow\infty}{\longrightarrow}0$. 
\end{enumerate}
Assuming that the expression in (\ref{eqIso.2}) defines an operator
$S$ in $\mathscr{H}$, it follows from (\ref{eqIso.3}) that $S$
is linear. To see this, let $\left(\varphi,\mu\right)$, $\left(\psi,\nu\right)$,
and $\lambda$ be as stated in the conditions below (\ref{eqIso.3}).
Then $\mu\circ\tau^{-1}\ll\lambda\circ\tau^{-1}$, and $\nu\circ\tau^{-1}\ll\lambda\circ\tau^{-1}$,
and a calculation shows that the following formula holds for the transformation
of the Radon-Nikodym derivatives: setting
\begin{equation}
\frac{d\mu\circ\tau^{-1}}{d\lambda\circ\tau^{-1}}=k_{\mu},\label{eqIso.7}
\end{equation}
we have
\begin{equation}
k_{\mu}\circ\tau=\frac{d\mu}{d\lambda}\quad(\mathrm{a.e.}\,\lambda).\label{eqIso.8}
\end{equation}
Similarly, ${\displaystyle k_{\nu}:=\frac{d\nu\circ\tau^{-1}}{d\lambda\circ\tau^{-1}}}$
satisfies
\begin{equation}
k_{\nu}\circ\tau=\frac{d\nu}{d\lambda}\quad(\mathrm{a.e.}\,\lambda).\label{eqIso.8half}
\end{equation}

The argument above yields: 
\begin{lem}
Let $\tau$ and $\sigma$ be endomorphisms in $M$ such that $\sigma\circ\tau=id_{M}$.
Let $\mu$, $\lambda$ be a pair of positive measures with $\mu\ll\lambda$,
and set $L:=d\mu/d\lambda$; then 
\begin{equation}
\frac{d\left(\mu\circ\tau^{-1}\right)}{d\left(\lambda\circ\tau^{-1}\right)}=L\circ\sigma,\label{eq:Iosa1}
\end{equation}
i.e., composition on the RHS in (\ref{eq:Iosa1}). 
\end{lem}

To show that $S$ is linear, we must calculate the sum
\begin{equation}
\left(\varphi\circ\sigma,\mu\circ\tau^{-1}\right)+\left(\psi\circ\sigma,\nu\circ\tau^{-1}\right),\label{eqIso.9}
\end{equation}
or, in expanded notation, we must verify that
\begin{equation}
\left(\varphi\circ\sigma\sqrt{k_{\mu}}+\psi\circ\sigma\sqrt{k_{\nu}},\lambda\circ\tau^{-1}\right)\sim\left(\left(\varphi\sqrt{\frac{d\mu}{d\lambda}}+\psi\sqrt{\frac{d\nu}{d\lambda}}\,\right)\circ\sigma,\lambda\circ\tau^{-1}\right).\label{eqIso.10}
\end{equation}
We get this class identity by an application of (\ref{eqIso.8}) as
follows:
\[
k_{\mu}\left(x\right)=k_{\mu}\left(\tau\left(\sigma\left(x\right)\right)\right)=\left.\left(\sqrt{\frac{d\mu}{d\lambda}}\circ\sigma\right)\right\vert _{\tau\left(M\right)}\left(x\right)\quad(\mathrm{a.e.}\,\lambda\circ\tau^{-1}).
\]
Similarly, for the other measure, we get
\begin{equation}
k_{\nu}=\left.\left(\sqrt{\frac{d\nu}{d\lambda}}\circ\sigma\right)\right\vert _{\tau\left(M\right)}\quad(\mathrm{a.e.}\,\lambda\circ\tau^{-1}).\label{eqIso.11}
\end{equation}

Assuming again that $S$ in (\ref{eqIso.2}) is well defined, we now
show that it is isometric, i.e., that $\left\Vert S\left(\varphi,\mu\right)\right\Vert ^{2}=\left\Vert \left(\varphi,\mu\right)\right\Vert ^{2}$,
referring to the norm of $\mathscr{H}$. In view of (\ref{eqIso.3})
and (\ref{eqIso.10}), it is enough to show that
\begin{equation}
\int_{M}\left\vert \varphi\circ\sigma\right\vert ^{2}k_{\mu}d\lambda\circ\tau^{-1}=\int_{M}\left\vert \varphi\right\vert ^{2}\frac{d\mu}{d\lambda}d\lambda.\label{eqIso.12}
\end{equation}
But, using (\ref{eqIso.8}), we get 
\begin{eqnarray*}
\int_{M}\left\vert \varphi\circ\sigma\right\vert ^{2}k_{\mu}d\lambda\circ\tau^{-1} & = & \int_{M}\left\vert \varphi\circ\sigma\circ\tau\right\vert ^{2}k_{\mu}\circ\tau\,d\lambda\\
 & \underset{\text{(\ref{eqIso.8})}}{=} & \int_{M}\left\vert \varphi\right\vert ^{2}\frac{d\mu}{d\lambda}d\lambda,
\end{eqnarray*}
which is the desired formula (\ref{eqIso.12}).

It remains to prove that $S$ is well defined, i.e., that the following
implication holds:
\begin{equation}
\left(\varphi,\mu\right)\sim\left(\psi,\nu\right)\Longrightarrow\left(\varphi\circ\sigma,\mu\circ\tau^{-1}\right)\sim\left(\psi\circ\sigma,\nu\circ\tau^{-1}\right).\label{eqIso.13}
\end{equation}
To do this, we go through a sequence of implications which again uses
the fundamental transformation rules (\ref{eqIso.8}) and (\ref{eqIso.11}).

\end{proof}
\begin{lem}
Pick some $\lambda$ such that $\mu\ll\lambda$ and $\nu\ll\lambda$.
We then have the following implication: 
\begin{equation}
\varphi\sqrt{\frac{d\mu}{d\lambda}}=\psi\sqrt{\frac{d\nu}{d\lambda}}\;\left(\mathrm{a.e.}\,\lambda\right)\Longrightarrow\left(\varphi\circ\sigma\right)\sqrt{k_{\mu}}=\left(\psi\circ\sigma\right)\sqrt{k_{\nu}}\;(\mathrm{a.e.}\,\lambda\circ\tau^{-1}),\label{eqIso.14}
\end{equation}
where ${\displaystyle k_{\mu}=\frac{d\mu\circ\tau^{-1}}{d\lambda\circ\tau^{-1}}}$
and ${\displaystyle k_{\nu}=\frac{d\nu\circ\tau^{-1}}{d\lambda\circ\tau^{-1}}}$.
(The desired conclusion (\ref{eqIso.13}) follows from this.)
\end{lem}

\begin{proof}
We now turn to the proof of the implication (\ref{eqIso.14}). We
pick a third measure $\lambda$ as described, and assume the identity
\[
\varphi\sqrt{\frac{d\mu}{d\lambda}}=\psi\sqrt{\frac{d\nu}{d\lambda}}\quad\mathrm{a.e.}\,\lambda.
\]
Let $f$ be a bounded Borel function on $M$. In the following calculations,
all integrals are over the full space $M$, but the measures change
as we make transformations, and we use the definition of the Radon-Nikodym
formula. First note that
\begin{align*}
\int f\,k_{\mu}\left(\frac{d\nu}{d\lambda}\circ\sigma\right)d\lambda\circ\tau^{-1} & =\int f\,\left(\frac{d\nu}{d\lambda}\circ\sigma\right)d\mu\circ\tau^{-1}\\
 & =\int f\circ\tau\,\frac{d\nu}{d\lambda}d\mu=\int f\circ\tau\,\frac{d\nu}{d\lambda}\frac{d\mu}{d\lambda}d\lambda.
\end{align*}
But by symmetry, we also have
\[
\int f\,k_{\nu}\,\left(\frac{d\mu}{d\lambda}\circ\sigma\right)d\lambda\circ\tau^{-1}=\int f\circ\tau\frac{d\nu}{d\lambda}\frac{d\mu}{d\lambda}d\lambda.
\]
Putting the last two formulas together, we arrive at the following
identity:
\[
\int_{M}f\,k_{\mu}\left(\frac{d\nu}{d\lambda}\circ\sigma\right)d\lambda\circ\tau^{-1}=\int_{M}f\,k_{\nu}\left(\frac{d\mu}{d\lambda}\circ\sigma\right)d\lambda\circ\tau^{-1}.
\]
Since the function $f$ is arbitrary, we get 
\[
k_{\mu}\left(\frac{d\nu}{d\lambda}\circ\sigma\right)=k_{\nu}\left(\frac{d\mu}{d\lambda}\circ\sigma\right)\quad\mathrm{a.e.}\,\lambda\circ\tau^{-1}
\]
and, of course,
\[
\sqrt{k_{\mu}}\sqrt{\frac{d\nu}{d\lambda}}\circ\sigma=\sqrt{k_{\nu}}\sqrt{\frac{d\mu}{d\lambda}}\circ\sigma\quad\mathrm{a.e.}\,\lambda\circ\tau^{-1}.
\]
Using now the identity
\[
\varphi\sqrt{\frac{d\mu}{d\lambda}}=\psi\sqrt{\frac{d\nu}{d\lambda}}\quad\mathrm{a.e.}\,\lambda,
\]
we arrive at the formula
\[
\varphi\circ\sigma\sqrt{k_{\mu}}\sqrt{\frac{d\mu}{d\lambda}}\circ\sigma\sqrt{\frac{d\nu}{d\lambda}}\circ\sigma=\psi\circ\sigma\sqrt{k_{\nu}}\sqrt{\frac{d\mu}{d\lambda}}\circ\sigma\sqrt{\frac{d\nu}{d\lambda}}\circ\sigma,
\]
and by cancellation,
\[
\varphi\circ\sigma\sqrt{k_{\mu}}=\psi\circ\sigma\sqrt{k_{\nu}}\quad\mathrm{a.e.}\,\lambda\circ\tau^{-1}.
\]
This completes the proof of the implication (\ref{eqIso.14}), and
therefore also of (\ref{eqIso.13}). This means that, if the linear
operator $S$ is defined as in (\ref{eqIso.2}), then the result is
independent of which element is chosen \emph{in the equivalence class}
represented by the pair $\left(\varphi,\mu\right)$. Putting together
the steps in the proof, we conclude that $S\colon\mathscr{H}\rightarrow\mathscr{H}$
is an \emph{isometry}, and that it has the properties which are stated
in the lemma. 

Combining the lemmas, the proof of Theorem \ref{thm:iso} is now completed. 
\end{proof}
\begin{lem}
\label{LemIso.3}Let $M$ be a compact Hausdorff space, and let $\sigma$
be as in the statement of Lemma \textup{\ref{LemIso.2}}, i.e., $\sigma\colon M\rightarrow M$
is onto and continuous. Suppose $\sigma$ has two distinct branches
of the inverse, i.e., $\tau_{i}\colon M\rightarrow M$, $i=1,2$,
continuous, and satisfying $\sigma\circ\tau_{i}=id_{M}$, $i=1,2$.
Let $S_{i}\colon\mathscr{H}\rightarrow\mathscr{H}$ be the corresponding
isometries, i.e.,
\begin{equation}
S_{i}\left(\left(\varphi,\mu\right)\right):=\left(\varphi\circ\sigma,\mu\circ\tau_{i}^{-1}\right),\label{eqIso.15}
\end{equation}
or
\[
S_{i}\colon\varphi\sqrt{d\mu}\longmapsto\varphi\circ\sigma\;\sqrt{d\mu\circ\tau_{i}^{-1}}.
\]
Then the two isometries have orthogonal ranges, i.e.,
\begin{equation}
\left\langle \,S_{1}\left(\left(\varphi,\mu\right)\right),S_{2}\left(\left(\psi,\nu\right)\right)\right\rangle =0\label{eqIso.16}
\end{equation}
for all pairs of vectors in $\mathscr{H}$, i.e., all $\left(\varphi,\mu\right)\in\mathscr{H}$
and $\left(\psi,\nu\right)\in\mathscr{H}$. 
\end{lem}

\begin{proof}
Note that in the statement (\ref{eqIso.16}) of the conclusion, we
use $\left\langle \cdot,\cdot\right\rangle $ to denote the inner
product of the Hilbert space $\mathscr{H}$, as it was defined in
(\ref{eqIso.4}).

With the two measures $\mu$ and $\nu$ given, then the expression
in (\ref{eqIso.16}) involves the transformed measures $\mu\circ\tau_{1}^{-1}$
and $\nu\circ\tau_{2}^{-1}$. Now pick some measure $\lambda$ such
that $\mu\circ\tau_{1}^{-1}\ll\lambda$ and $\nu\circ\tau_{2}^{-1}\ll\lambda$.
Then the expression in (\ref{eqIso.16}) is
\begin{equation}
\int_{M}\overline{\varphi\circ\sigma}\;\psi\circ\sigma\,\sqrt{\frac{d\mu\circ\tau_{1}^{-1}}{d\lambda}}\sqrt{\frac{d\nu\circ\tau_{2}^{-1}}{d\lambda}}d\lambda.\label{eqIso.17}
\end{equation}
But ${\displaystyle \frac{d\mu\circ\tau_{1}^{-1}}{d\lambda}}$ is
supported on $\tau_{1}\left(M\right)$, while ${\displaystyle \frac{d\nu\circ\tau_{2}^{-1}}{d\lambda}}$
is supported on $\tau_{2}\left(M\right)$. Since $\tau_{1}\left(M\right)\cap\tau_{2}\left(M\right)=\emptyset$
by the choice of distinct branches for the inverse of $\sigma$, we
conclude that the integral in (\ref{eqIso.17}) vanishes. 
\end{proof}
\begin{cor}
\label{CorRep.3}Let $M$ be a compact Hausdorff space, and let $N\in\mathbb{N}$,
$N\geq2$, be given. Let $\sigma\colon M\rightarrow M$ be continuous
and onto. Suppose there are $N$ distinct branches of the inverse,
i.e., continuous $\tau_{i}\colon M\rightarrow M$, $i=1,\dots,N$,
such that
\begin{equation}
\sigma\circ\tau_{i}=id_{M}.\label{eqRep.9bis}
\end{equation}
Suppose there is a positive Borel measure $\mu$ such that $\mu\left(M\right)=1$,
and
\begin{equation}
\mu\circ\tau_{i}^{-1}\ll\mu,\quad i=1,\dots,N.\label{eqRep.10}
\end{equation}
Then the isometries
\begin{equation}
S_{i}\varphi:=\varphi\circ\sigma\sqrt{\frac{d\mu\circ\tau_{i}^{-1}}{d\mu}}\label{eqRep.11}
\end{equation}
satisfy
\begin{equation}
\sum_{i=1}^{N}S_{i}S_{i}^{\ast}=I_{L^{2}\left(\mu\right)}\label{eqRep.12}
\end{equation}
if and only if
\begin{equation}
\bigcup_{i=1}^{N}\tau_{i}\left(M\right)=M.\label{eqRep.13}
\end{equation}
\end{cor}

\begin{proof}
We already know from Lemma \ref{LemIso.3} that the isometries $S_{i}\colon L^{2}\left(\mu\right)\rightarrow L^{2}\left(\mu\right)$
are mutually orthogonal, i.e., that
\begin{equation}
S_{i}^{\ast}S_{j}=\delta_{i,j}I_{L^{2}\left(\mu\right)}.\label{eqRep.14}
\end{equation}
It follows that the terms in the sum (\ref{eqRep.12}) are \emph{commuting
projections}. Hence
\begin{equation}
\sum_{i=1}^{N}S_{i}S_{i}^{\ast}\leq I_{L^{2}\left(\mu\right)}.\label{eqRep.15}
\end{equation}
Moreover, we conclude that (\ref{eqRep.12}) holds if and only if
\begin{equation}
\sum_{i=1}^{N}\left\Vert S_{i}^{\ast}\varphi\right\Vert ^{2}=\left\Vert \varphi\right\Vert ^{2},\quad\varphi\in L^{2}\left(\mu\right).\label{eqRep.16}
\end{equation}
Setting $p_{i}:=\frac{d\mu\circ\tau_{i}^{-1}}{d\mu}$, we get
\begin{equation}
S_{i}^{\ast}\varphi=\varphi\circ\tau_{i}\,\left(p_{i}\circ\tau_{i}\right)^{-1/2}.\label{eqRep.17}
\end{equation}
It follows that 
\begin{align*}
\left\Vert S_{i}^{\ast}\varphi\right\Vert ^{2} & =\int_{M}\left\vert \varphi\circ\tau_{i}\right\vert ^{2}\left(p_{i}\circ\tau_{i}\right)^{-1}\,d\mu\\
 & =\int_{\tau_{i}\left(M\right)}\left\vert \varphi\right\vert ^{2}p_{i}^{-1}d\mu\circ\tau_{i}^{-1}=\int_{\tau_{i}\left(M\right)}\left\vert \varphi\right\vert ^{2}d\mu.
\end{align*}
Recall that the branches $\tau_{i}$ of the inverse are distinct,
and so the sets $\tau_{i}\left(M\right)$ are non-overlapping. The
equivalence (\ref{eqRep.12})$\Leftrightarrow$(\ref{eqRep.13}) now
follows directly from the previous calculation. 
\end{proof}

\subsection{\label{subsec:Dist}Distributions}

Consider the following setting, generalizing that of the three examples
in Section \ref{subsec:3ex}: Let $\left(\Omega_{N},\mathscr{C},\mathbb{P}\right)$
be a probability space, and $\left(M,\mathscr{B}\right)$ be a measure
space; see Section \ref{subsec:br} for definitions.

Let $\mathscr{H}\left(M\right)$ be \emph{the Hilbert space of equivalence
classes}, see Lemma \ref{LemIso.2} above. As shown in \cite{MR0282379},
if $\mu$ is a fixed positive $\sigma$-finite measure on $\left(M,\mathscr{B}\right)$,
then the subspace $\left\{ f\sqrt{d\mu}\mid f\in L^{2}\left(\mu\right)\right\} $
in $\mathscr{H}\left(M\right)$ is closed, denoted $\mathscr{H}\left(\mu\right)$;
and 
\begin{equation}
L^{2}\left(\mu\right)\ni f\longmapsto f\sqrt{d\mu}\in\mathscr{H}\left(\mu\right)\label{eq:L1}
\end{equation}
is an \emph{isometric isomorphism}; called the canonical isomorphism. 

The following is known; see e.g. \cite{MR0282379}: For two $\sigma$-finite
positive measures $\mu_{1}$, $\mu_{2}$ on $\left(M,\mathscr{B}\right)$,
we have the following three equivalences: 
\begin{align}
\mu_{1}\ll\mu_{2} & \Longleftrightarrow\mathscr{H}\left(\mu_{1}\right)\subseteq\mathscr{H}\left(\mu_{2}\right),\label{eq:L2}\\
\left(\begin{matrix}\text{\text{\ensuremath{\mu_{1}} and \ensuremath{\mu_{2}} are }}\\
\text{mutually singular}
\end{matrix}\right) & \Longleftrightarrow\mathscr{H}\left(\mu_{1}\right)\perp\mathscr{H}\left(\mu_{2}\right),\;\text{and}\label{eq:L3}\\
\left(\begin{matrix}\text{\text{\ensuremath{\mu_{1}} and \ensuremath{\mu_{2}} are }}\\
\text{equivalent}
\end{matrix}\right) & \Longleftrightarrow\mathscr{H}\left(\mu_{1}\right)=\mathscr{H}\left(\mu_{2}\right).\label{eq:L4}
\end{align}

\begin{cor}
\label{cor:Yexp}Let $Y_{i}:\Omega_{N}\rightarrow M$, $i=1,2$, be
two random variables; i.e., the two are measurable functions w.r.t.
the respective $\sigma$-algebras $\mathscr{C}$ and $\mathscr{B}$.
The corresponding distributions 
\begin{equation}
\mu_{i}:=\mathbb{P}\circ Y_{i}^{-1},\quad i=1,2\label{eq:L5}
\end{equation}
are measures on $\left(M,\mathscr{B}\right)$; and 
\begin{equation}
T_{i}f:=f\circ Y_{i},\quad i=1,2,\label{eq:L6}
\end{equation}
\textup{(see Fig \ref{fig:fy1})} are isometries
\begin{equation}
L^{2}\left(\mu_{i}\right)\simeq\mathscr{H}\left(\mu_{i}\right)\xrightarrow{\quad T_{i}\quad}L^{2}\left(\mathbb{P}\right),\quad i=1,2.\label{eq:L7}
\end{equation}

\begin{figure}[H]
\[
\xymatrix{\Omega_{N}\ar[rr]^{Y_{i}}\ar[dr]_{f\circ Y_{i}} &  & M\ar[dl]^{f}\\
 & \mathbb{R}
}
\]

\caption{\label{fig:fy1}}

\end{figure}

Hence the three conditions in (\ref{eq:L2}), (\ref{eq:L3}) and (\ref{eq:L4})
are statements about the two random variables. 

For the operators $T_{2}^{*}T_{1}$, see Fig \ref{fig:fy2}, we have
the following: For $f\in L^{2}\left(\mu_{1}\right)$, and $x\in M$:
\begin{equation}
\left(T_{2}^{*}T_{1}\right)\left(f\right)\left(x\right)=\mathbb{E}_{Y_{2}=x}\left(f\circ Y_{1}\mid\mathscr{F}_{Y_{2}}\right).\label{eq:L8}
\end{equation}

\begin{figure}[H]
\[
\xymatrix{\mathscr{H}\left(\mu_{1}\right)\ar[rr]^{T_{2}^{*}T_{1}}\ar[dr]_{T_{1}} &  & \mathscr{H}\left(\mu_{2}\right)\ar[dl]^{T_{2}}\\
 & L^{2}\left(\mathbb{P}\right)
}
\]

\caption{\label{fig:fy2}}
\end{figure}
\end{cor}

\begin{proof}
For $f\in L^{2}\left(\mu_{1}\right)$ and $g\in L^{2}\left(\mu_{2}\right)$,
we have: 
\begin{align*}
\left\langle T_{2}^{*}T_{1}f,g\right\rangle _{\mathscr{H}\left(\mu_{2}\right)} & =\left\langle T_{1}f,T_{2}g\right\rangle _{L^{2}\left(\mathbb{P}\right)}\\
 & =\mathbb{E}\left[\left(f\circ Y_{1}\right)\left(g\circ Y_{2}\right)\right]\\
 & =\mathbb{E}\left[\mathbb{E}\left(f\circ Y_{1}\mid\mathscr{F}_{Y_{2}}\right)\left(g\circ Y_{2}\right)\right]\\
 & =\int_{M}\mathbb{E}_{\left(Y_{2}=x\right)}\left(f\circ Y_{1}\mid\mathscr{F}_{Y_{2}}\right)g\left(x\right)d\mu_{2}\left(x\right),
\end{align*}
and the desired formula (\ref{eq:L8}) follows from this, and (\ref{eq:L1}).
\end{proof}

\subsection{\label{subsec:fc}Fractional Calculus}

In recent papers \cite{MR3427068,MR3729651}, a number of authors
have studied \emph{gradient operators} computed with respect to singular
measures. The purpose of this subsection is to combine results from
the present Sections \ref{sec:IFS} and \ref{sec:uh} to display some
operator theoretic properties of these gradients $\nabla_{\mu}$ ,
and to connect them to our boundary analysis. 

In order to add clarity, we shall consider singular measures $\mu$
supported on compact subsets of the real line $\mathbb{R}$, but the
ideas extend to more general measure spaces. For particular examples,
readers are referred to the three examples in Section \ref{subsec:3ex}
above.

Let $I=\left[0,1\right]$ be the unit-interval with the Borel $\sigma$-algebra.
By $\mathscr{H}\left(I\right)$ we shall denote the Hilbert space
of equivalence classes as in Section \ref{subsec:IsoH}. When $\mu$
is a fixed positive measure, we considered the isometric isomorphism
$T_{\mu}:L^{2}\left(\mu\right)\simeq\mathscr{H}\left(\mu\right)$
in (\ref{eq:L1}).

In Proposition \ref{prop:grad} below, we shall identity the gradient
$\nabla_{\mu}$ with the adjoint operator $T_{\mu}^{*}$ , referring
to the respective inner products from (\ref{eq:L1}).
\begin{defn}
Let $F$ be a function on $\mathbb{R}$ of bounded variation, and
let $dF$ be the corresponding Stieltjes measure, with variation measure
$\left|dF\right|$ defined in the usual way. If $\left|dF\right|\ll\mu$
, then the Radon-Nikodym derivative 
\begin{equation}
RN_{\mu}\left(dF\right)=:\nabla_{\mu}F\label{eq:m1}
\end{equation}
is well defined; we have:
\begin{equation}
\left(dF\right)\left(B\right)=\int_{B}\left(\nabla_{\mu}F\right)d\mu,\;\forall B\in\mathscr{B},\label{eq:m2}
\end{equation}
where $\mathscr{B}$ is the Borel $\sigma$-algebra. For the case
of $\left(I,\mathscr{B}\right)$, (\ref{eq:m2}) is equivalent to
\begin{equation}
F\left(x\right)=\int_{0}^{x}dF=\int_{0}^{x}\nabla_{\mu}F\:d\mu,\;\forall x\in\left[0,1\right]\label{eq:m3}
\end{equation}
(We shall adopt the normalization $F\left(0\right)=0$.)
\end{defn}

\begin{prop}
\label{prop:grad}If $T_{\mu}:L^{2}\left(\mu\right)\rightarrow\mathscr{H}\left(\mu\right)$
is as in (\ref{eq:L1}), then the adjoint operator $T_{\mu}^{*}$
agrees with $\nabla_{\mu}$.
\end{prop}

\begin{proof}
In view of Corollary \ref{cor:Yexp} in Section \ref{subsec:Dist},
the desired conclusion will follow if we check that, when $F$ is
of bounded variation with $\left|dF\right|\ll\mu$, and if $\varphi\in L^{2}\left(\mu\right)$,
then 
\begin{equation}
\big\langle\underset{T_{\mu}\varphi}{\underbrace{\varphi\sqrt{d\mu}}},dF\big\rangle_{\mathscr{H}\left(\mu\right)}=\big\langle\varphi,\underset{T_{\mu}^{*}F}{\underbrace{\nabla_{\mu}F}}\big\rangle_{L^{2}\left(\mu\right)}.\label{eq:m4}
\end{equation}
But, using our analysis from Sections \ref{subsec:IsoH}-\ref{subsec:Dist}
above, the verification of (\ref{eq:m4}) is equivalent to: 
\[
\text{LHS}_{\left(\ref{eq:m4}\right)}=\underset{\text{\ensuremath{L^{2}\left(\mu\right)-}inner product}}{\underbrace{\int_{I}\varphi\underset{\text{\ensuremath{{\scriptscriptstyle \text{Radon-Nikodym der}}}}}{\underbrace{\left(\nabla_{\mu}F\right)}}d\mu}}=\text{RHS}_{\left(\ref{eq:m4}\right)};
\]
and the conclusion follows.
\end{proof}
\begin{acknowledgement*}
The co-authors thank the following colleagues for helpful and enlightening
discussions: Professors Sergii Bezuglyi, Ilwoo Cho, Carla Farsi, Elizabeth
Gillaspy, Sooran Kang, Judy Packer, Wayne Polyzou, Myung-Sin Song,
and members in the Math Physics seminar at The University of Iowa.
\end{acknowledgement*}
\bibliographystyle{amsalpha}
\bibliography{ref}

\newcommand{\etalchar}[1]{$^{#1}$}
\providecommand{\bysame}{\leavevmode\hbox to3em{\hrulefill}\thinspace}
\providecommand{\MR}{\relax\ifhmode\unskip\space\fi MR }
\providecommand{\MRhref}[2]{%
  \href{http://www.ams.org/mathscinet-getitem?mr=#1}{#2}
}
\providecommand{\href}[2]{#2}
\begin{thebibliography}{FGJ{\etalchar{+}}18b}

\bibitem[AJ15]{MR3402823}
Daniel Alpay and Palle Jorgensen, \emph{Spectral theory for {G}aussian
  processes: reproducing kernels, boundaries, and {$L^2$}-wavelet generators
  with fractional scales}, Numer. Funct. Anal. Optim. \textbf{36} (2015),
  no.~10, 1239--1285. \MR{3402823}

\bibitem[AJL17]{MR3687240}
Daniel Alpay, Palle Jorgensen, and David Levanony, \emph{On the equivalence of
  probability spaces}, J. Theoret. Probab. \textbf{30} (2017), no.~3, 813--841.
  \MR{3687240}

\bibitem[AJL18]{MR3796644}
Daniel Alpay, Palle Jorgensen, and Izchak Lewkowicz, \emph{{$W$}-{M}arkov
  measures, transfer operators, wavelets and multiresolutions}, Frames and
  {H}armonic {A}nalysis, Contemp. Math., vol. 706, Amer. Math. Soc.,
  Providence, RI, 2018, pp.~293--343. \MR{3796644}

\bibitem[Aro50]{MR0044033}
N.~Aronszajn, \emph{Introduction to the {T}heory of {H}ilbert {S}paces. {V}ol.
  {I}}, Mathematical Monographs, Oklahoma Agricultural and Mechanical College,
  Stillwater, Okla., 1950. \MR{0044033}

\bibitem[Arv03a]{MR2026010}
William Arveson, \emph{Four lectures on noncommutative dynamics}, Advances in
  quantum dynamics ({S}outh {H}adley, {MA}, 2002), Contemp. Math., vol. 335,
  Amer. Math. Soc., Providence, RI, 2003, pp.~1--55. \MR{2026010}

\bibitem[Arv03b]{MR1978577}
\bysame, \emph{Noncommutative dynamics and {$E$}-semigroups}, Springer
  Monographs in Mathematics, Springer-Verlag, New York, 2003. \MR{1978577}

\bibitem[AS57]{MR0088706}
N.~Aronszajn and K.~T. Smith, \emph{Characterization of positive reproducing
  kernels. {A}pplications to {G}reen's functions}, Amer. J. Math. \textbf{79}
  (1957), 611--622. \MR{0088706}

\bibitem[BHS08]{MR2431670}
Michael~F. Barnsley, John~E. Hutchinson, and \"Orjan Stenflo,
  \emph{{$V$}-variable fractals: fractals with partial self similarity}, Adv.
  Math. \textbf{218} (2008), no.~6, 2051--2088. \MR{2431670}

\bibitem[BJ97]{MR1444086}
O.~Bratteli and P.~E.~T. Jorgensen, \emph{Endomorphisms of
  {${\mathscr{B}}({\mathscr{ H}})$}. {II}. {F}initely correlated states on
  {${\mathscr{O}}_{n}$}}, J. Funct. Anal. \textbf{145} (1997), no.~2, 323--373.
  \MR{1444086}

\bibitem[BJ02]{MR1913212}
Ola Bratteli and Palle Jorgensen, \emph{Wavelets through a looking glass},
  Applied and Numerical Harmonic Analysis, Birkh\"auser Boston, Inc., Boston,
  MA, 2002, The world of the spectrum. \MR{1913212}

\bibitem[BJKR01]{MR1869063}
Ola Bratteli, Palle E.~T. Jorgensen, Ki~Hang Kim, and Fred Roush,
  \emph{Decidability of the isomorphism problem for stationary {AF}-algebras
  and the associated ordered simple dimension groups}, Ergodic Theory Dynam.
  Systems \textbf{21} (2001), no.~6, 1625--1655. \MR{1869063}

\bibitem[BJKR02]{MR1889566}
\bysame, \emph{Computation of isomorphism invariants for stationary dimension
  groups}, Ergodic Theory Dynam. Systems \textbf{22} (2002), no.~1, 99--127.
  \MR{1889566}

\bibitem[BJKW00]{MR1740897}
O.~Bratteli, P.~E.~T. Jorgensen, A.~Kishimoto, and R.~F. Werner, \emph{Pure
  states on {$\mathscr{O}_{d}$}}, J. Operator Theory \textbf{43} (2000), no.~1,
  97--143. \MR{1740897}

\bibitem[BJOk04]{MR2030387}
Ola Bratteli, Palle E.~T. Jorgensen, and Vasyl\cprime Ostrovs\cprime~ky\u\i,
  \emph{Representation theory and numerical {AF}-invariants. {T}he
  representations and centralizers of certain states on {$\mathscr{O}_{d}$}},
  Mem. Amer. Math. Soc. \textbf{168} (2004), no.~797, xviii+178. \MR{2030387}

\bibitem[BKW12]{MR2976663}
Marek Bo\.zejko, Anna Krystek, and \L~ukasz Wojakowski (eds.),
  \emph{Noncommutative harmonic analysis with applications to probability
  {III}}, Banach Center Publications, vol.~96, Polish Academy of Sciences,
  Institute of Mathematics, Warsaw, 2012, Papers from the 13th Workshop held in
  B\k{e}dlewo, July 11--17, 2010. \MR{2976663}

\bibitem[Bou98]{MR1726779}
Nicolas Bourbaki, \emph{General topology. {C}hapters 1--4}, Elements of
  Mathematics (Berlin), Springer-Verlag, Berlin, 1998, Translated from the
  French, Reprint of the 1989 English translation. \MR{1726779}

\bibitem[Cun77]{MR0467330}
Joachim Cuntz, \emph{Simple {$C\sp*$}-algebras generated by isometries}, Comm.
  Math. Phys. \textbf{57} (1977), no.~2, 173--185. \MR{0467330 (57 \#7189)}

\bibitem[DHJ15]{MR3250475}
Dorin~Ervin Dutkay, John Haussermann, and Palle E.~T. Jorgensen, \emph{Atomic
  representations of {C}untz algebras}, J. Math. Anal. Appl. \textbf{421}
  (2015), no.~1, 215--243. \MR{3250475}

\bibitem[DJ14]{MR3217056}
Dorin~Ervin Dutkay and Palle E.~T. Jorgensen, \emph{Monic representations of
  the {C}untz algebra and {M}arkov measures}, J. Funct. Anal. \textbf{267}
  (2014), no.~4, 1011--1034. \MR{3217056}

\bibitem[DJ15]{MR3394108}
\bysame, \emph{Representations of {C}untz algebras associated to
  quasi-stationary {M}arkov measures}, Ergodic Theory Dynam. Systems
  \textbf{35} (2015), no.~7, 2080--2093. \MR{3394108}

\bibitem[EGW17]{MR3601650}
Steven~N. Evans, Rudolf Gr\"ubel, and Anton Wakolbinger, \emph{Doob-{M}artin
  boundary of {R}\'emy's tree growth chain}, Ann. Probab. \textbf{45} (2017),
  no.~1, 225--277. \MR{3601650}

\bibitem[FGJ{\etalchar{+}}17]{2017arXiv170900592F}
C.~{Farsi}, E.~{Gillaspy}, P.~{Jorgensen}, S.~{Kang}, and J.~{Packer},
  \emph{{Separable representations of higher-rank graphs}}, ArXiv e-prints
  (2017).

\bibitem[FGJ{\etalchar{+}}18a]{2018arXiv180403455F}
\bysame, \emph{{Monic representations of finite higher-rank graphs}}, ArXiv
  e-prints (2018).

\bibitem[FGJ{\etalchar{+}}18b]{2018arXiv180308779F}
C.~{Farsi}, E.~{Gillaspy}, P.~E.~T. {Jorgensen}, S.~{Kang}, and J.~{Packer},
  \emph{{Representations of higher-rank graph $C^*$-algebras associated to
  $\Lambda$-semibranching function systems}}, ArXiv e-prints (2018).

\bibitem[FHH17]{MR3729651}
U.~Freiberg, B.~M. Hambly, and John~E. Hutchinson, \emph{Spectral asymptotics
  for {$V$}-variable {S}ierpinski gaskets}, Ann. Inst. Henri Poincar\'e Probab.
  Stat. \textbf{53} (2017), no.~4, 2162--2213. \MR{3729651}

\bibitem[FNW92]{MR1158756}
M.~Fannes, B.~Nachtergaele, and R.~F. Werner, \emph{Finitely correlated states
  on quantum spin chains}, Comm. Math. Phys. \textbf{144} (1992), no.~3,
  443--490. \MR{1158756}

\bibitem[FNW94]{MR1266319}
\bysame, \emph{Finitely correlated pure states}, J. Funct. Anal. \textbf{120}
  (1994), no.~2, 511--534. \MR{1266319}

\bibitem[FS15]{MR3427068}
Uta Freiberg and Christian Seifert, \emph{Dirichlet forms for singular
  diffusion in higher dimensions}, J. Evol. Equ. \textbf{15} (2015), no.~4,
  869--878. \MR{3427068}

\bibitem[Gli60]{Gli60}
James~G. Glimm, \emph{On a certain class of operator algebras}, Trans. Amer.
  Math. Soc. \textbf{95} (1960), 318--340. \MR{0112057 (22 \#2915)}

\bibitem[Gli61]{Gli61}
James Glimm, \emph{Type {I} {$C^{\ast} $}-algebras}, Ann. of Math. (2)
  \textbf{73} (1961), 572--612. \MR{0124756 (23 \#A2066)}

\bibitem[Hel69]{MR0243367}
Gilbert Helmberg, \emph{Introduction to spectral theory in {H}ilbert space},
  North-Holland Series in Applied Mathematics and Mechanics, Vol. 6,
  North-Holland Publishing Co., Amsterdam-London; Wiley Interscience Division
  John Wiley \& Sons, Inc., New York, 1969. \MR{0243367}

\bibitem[Hid80]{MR562914}
Takeyuki Hida, \emph{Brownian motion}, Applications of Mathematics, vol.~11,
  Springer-Verlag, New York-Berlin, 1980, Translated from the Japanese by the
  author and T. P. Speed. \MR{562914}

\bibitem[HJr94]{MR1278486}
J.~Hoffmann-J\o~rgensen, \emph{Probability with a view toward statistics.
  {V}ol. {II}}, Chapman \& Hall Probability Series, Chapman \& Hall, New York,
  1994. \MR{1278486}

\bibitem[HJW18]{MR3796641}
John~E. Herr, Palle E.~T. Jorgensen, and Eric~S. Weber, \emph{A matrix
  characterization of boundary representations of positive matrices in the
  {H}ardy space}, Frames and harmonic analysis, Contemp. Math., vol. 706, Amer.
  Math. Soc., Providence, RI, 2018, pp.~255--270. \MR{3796641}

\bibitem[Hut81]{MR625600}
John~E. Hutchinson, \emph{Fractals and self-similarity}, Indiana Univ. Math. J.
  \textbf{30} (1981), no.~5, 713--747. \MR{625600}

\bibitem[Hut95]{MR1656855}
\bysame, \emph{Fractals: a mathematical framework}, Complex. Int. \textbf{2}
  (1995), 14 HTML documents. \MR{1656855}

\bibitem[JLW12]{MR2833578}
Hongbing Ju, Ka-Sing Lau, and Xiang-Yang Wang, \emph{Post-critically finite
  fractal and {M}artin boundary}, Trans. Amer. Math. Soc. \textbf{364} (2012),
  no.~1, 103--118. \MR{2833578}

\bibitem[Jor07]{MR2362879}
Palle E.~T. Jorgensen, \emph{The measure of a measurement}, J. Math. Phys.
  \textbf{48} (2007), no.~10, 103506, 15. \MR{2362879}

\bibitem[JP98]{MR1655831}
Palle E.~T. Jorgensen and Steen Pedersen, \emph{Dense analytic subspaces in
  fractal {$L^2$}-spaces}, J. Anal. Math. \textbf{75} (1998), 185--228.
  \MR{1655831}

\bibitem[JR05]{MR2176941}
Yunping Jiang and David Ruelle, \emph{Analyticity of the susceptibility
  function for unimodal {M}arkovian maps of the interval}, Nonlinearity
  \textbf{18} (2005), no.~6, 2447--2453. \MR{2176941}

\bibitem[JT15]{MR3450534}
Palle Jorgensen and Feng Tian, \emph{Discrete reproducing kernel {H}ilbert
  spaces: sampling and distribution of {D}irac-masses}, J. Mach. Learn. Res.
  \textbf{16} (2015), 3079--3114. \MR{3450534}

\bibitem[JT17]{MR3642406}
\bysame, \emph{Non-commutative analysis}, World Scientific Publishing Co. Pte.
  Ltd., Hackensack, NJ, 2017, With a foreword by Wayne Polyzou. \MR{3642406}

\bibitem[Kak43]{MR0014404}
Shizuo Kakutani, \emph{Notes on infinite product measure spaces. {II}}, Proc.
  Imp. Acad. Tokyo \textbf{19} (1943), 184--188. \MR{0014404}

\bibitem[Kak48]{MR0023331}
\bysame, \emph{On equivalence of infinite product measures}, Ann. of Math. (2)
  \textbf{49} (1948), 214--224. \MR{0023331}

\bibitem[KLW17]{MR3709130}
Shi-Lei Kong, Ka-Sing Lau, and Ting-Kam~Leonard Wong, \emph{Random walks and
  induced {D}irichlet forms on self-similar sets}, Adv. Math. \textbf{320}
  (2017), 1099--1134. \MR{3709130}

\bibitem[Kol83]{MR735967}
A.~N. Kolmogorov, \emph{On logical foundations of probability theory},
  Probability theory and mathematical statistics ({T}bilisi, 1982), Lecture
  Notes in Math., vol. 1021, Springer, Berlin, 1983, pp.~1--5. \MR{735967}

\bibitem[Kor08]{MR2384480}
Dmitry Korshunov, \emph{The key renewal theorem for a transient {M}arkov
  chain}, J. Theoret. Probab. \textbf{21} (2008), no.~1, 234--245. \MR{2384480}

\bibitem[LN12]{MR2891314}
Ka-Sing Lau and Sze-Man Ngai, \emph{Martin boundary and exit space on the
  {S}ierpinski gasket}, Sci. China Math. \textbf{55} (2012), no.~3, 475--494.
  \MR{2891314}

\bibitem[LN14]{MR3130523}
\bysame, \emph{Boundary theory on the {H}ata tree}, Nonlinear Anal. \textbf{95}
  (2014), 292--307. \MR{3130523}

\bibitem[Mac89]{MR1043174}
George~W. Mackey, \emph{Unitary group representations in physics, probability,
  and number theory}, second ed., Advanced Book Classics, Addison-Wesley
  Publishing Company, Advanced Book Program, Redwood City, CA, 1989.
  \MR{1043174}

\bibitem[Mat98]{MR1665280}
Taku Matsui, \emph{A characterization of pure finitely correlated states},
  Infin. Dimens. Anal. Quantum Probab. Relat. Top. \textbf{1} (1998), no.~4,
  647--661. \MR{1665280}

\bibitem[Mey93]{MR1222649}
Paul-Andr\'e Meyer, \emph{Quantum probability for probabilists}, Lecture Notes
  in Mathematics, vol. 1538, Springer-Verlag, Berlin, 1993. \MR{1222649}

\bibitem[Mil06]{MR2193309}
John Milnor, \emph{Dynamics in one complex variable}, third ed., Annals of
  Mathematics Studies, vol. 160, Princeton University Press, Princeton, NJ,
  2006. \MR{2193309}

\bibitem[MO86]{MR863346}
Hartmut Milbrodt and Ulrich~G. Oppel, \emph{Projective systems and projective
  limits of vector measures}, Bayreuth. Math. Schr. (1986), no.~22, 1--86.
  \MR{863346}

\bibitem[MU15]{MR3375595}
Volker Mayer and Mariusz Urba{\'n}ski, \emph{Countable alphabet random subhifts
  of finite type with weakly positive transfer operator}, J. Stat. Phys.
  \textbf{160} (2015), no.~5, 1405--1431. \MR{3375595}

\bibitem[Nel69]{MR0282379}
Edward Nelson, \emph{Topics in dynamics. {I}: {F}lows}, Mathematical Notes,
  Princeton University Press, Princeton, N.J.; University of Tokyo Press,
  Tokyo, 1969. \MR{0282379}

\bibitem[Ohn07]{MR2294415}
H.~Ohno, \emph{Factors generated by {$C^*$}-finitely correlated states},
  Internat. J. Math. \textbf{18} (2007), no.~1, 27--41. \MR{2294415}

\bibitem[Pap15]{MR3368972}
Pietro Paparella, \emph{Matrix functions that preserve the strong
  {P}erron-{F}robenius property}, Electron. J. Linear Algebra \textbf{30}
  (2015), 271--278. \MR{3368972}

\bibitem[Rue04]{MR2129258}
David Ruelle, \emph{Thermodynamic formalism}, second ed., Cambridge
  Mathematical Library, Cambridge University Press, Cambridge, 2004, The
  mathematical structures of equilibrium statistical mechanics. \MR{2129258}

\bibitem[Rug16]{MR3459161}
Hans~Henrik Rugh, \emph{The {M}ilnor-{T}hurston determinant and the {R}uelle
  transfer operator}, Comm. Math. Phys. \textbf{342} (2016), no.~2, 603--614.
  \MR{3459161}

\bibitem[SBM07]{MR2357627}
Roger~B. Sidje, Kevin Burrage, and Shev Macnamara, \emph{Inexact uniformization
  method for computing transient distributions of {M}arkov chains}, SIAM J.
  Sci. Comput. \textbf{29} (2007), no.~6, 2562--2580. \MR{2357627}

\bibitem[Sto13]{MR3124323}
Luchezar Stoyanov, \emph{Ruelle operators and decay of correlations for contact
  {A}nosov flows}, C. R. Math. Acad. Sci. Paris \textbf{351} (2013), no.~17-18,
  669--672. \MR{3124323}

\bibitem[Tak11]{MR2851247}
Masayoshi Takeda, \emph{A large deviation principle for symmetric {M}arkov
  processes with {F}eynman-{K}ac functional}, J. Theoret. Probab. \textbf{24}
  (2011), no.~4, 1097--1129. \MR{2851247}

\bibitem[Tju72]{MR0443014}
Tue Tjur, \emph{On the mathematical foundations of probability}, Institute of
  Mathematical Statistics, University of Copenhagen, Copenhagen, 1972, Lecture
  Notes, No. 1. \MR{0443014}

\bibitem[Tsi03]{MR2029632}
Boris Tsirelson, \emph{Non-isomorphic product systems}, Advances in quantum
  dynamics ({S}outh {H}adley, {MA}, 2002), Contemp. Math., vol. 335, Amer.
  Math. Soc., Providence, RI, 2003, pp.~273--328. \MR{2029632}

\bibitem[Tum08]{MR2393052}
Roderich Tumulka, \emph{A {K}olmogorov extension theorem for {POVM}s}, Lett.
  Math. Phys. \textbf{84} (2008), no.~1, 41--46. \MR{2393052}

\bibitem[Vee12]{MR2954650}
William~A. Veech, \emph{Martin boundary for the similarity walk in a planar
  triangle}, Houston J. Math. \textbf{38} (2012), no.~2, 525--548. \MR{2954650}

\bibitem[Wil04]{MR2048350}
Stephen Willard, \emph{General topology}, Dover Publications, Inc., Mineola,
  NY, 2004, Reprint of the 1970 original [Addison-Wesley, Reading, MA;
  MR0264581]. \MR{2048350}

\end{thebibliography}

\end{document}